\documentclass[11pt,reqno]{amsart}
\usepackage{fullpage}
\usepackage{amsmath,amsthm,verbatim,amssymb,amsfonts,amscd, graphicx,bbm,bm}
\usepackage{graphics}
\usepackage{mathtools}
\usepackage{mathrsfs}
\usepackage{eufrak}
\usepackage{esint}
\usepackage{color}
\usepackage{tikz,tikz-cd}
\usepackage{subfigure}

\usepackage[pdfstartview=FitH, bookmarksnumbered=true,bookmarksopen=true, colorlinks=true, pdfborder=001, citecolor=blue, linkcolor=blue,urlcolor=blue]{hyperref}

\newtheorem{theorem}{Theorem}
\newtheorem{proposition}[theorem]{Proposition}
\newtheorem{lemma}[theorem]{Lemma}

\theoremstyle{definition}
\newtheorem{remark}{Remark}

\newcommand{\cref}[1]{Corollary~\ref{c.#1}}

\numberwithin{equation}{section}
\numberwithin{theorem}{section}

\newcommand{\R}{\mathbb{R}}
\newcommand{\N}{\mathbb{N}}
\newcommand{\Z}{\mathbb{Z}}
\newcommand{\ol}{\overline}
\newcommand{\eps}{\varepsilon}

\newcommand{\womega}{\widetilde{\omega}}

\title[Optimal convergence rate with high-contrast coefficients]{Convergece rate and uniform Lipschitz estimate in periodic homogenization of high-contrast elliptic systems}% with high-contrast coefficients}
\author{Xin Fu}
\address[X. Fu]{Yau Mathematical Sciences Center, Tsinghua University, Beijing 100084, P.R. China}
\email{fux20@mails.tsinghua.edu.cn}

\author{Wenjia Jing}
\address[W. Jing]{Yau Mathematical Sciences Center, Tsinghua University, Beijing 100084 and Beijing Institute of Mathematical Sciences and Applications, Beijing 101408, P.R. China}
\email{wjjing@tsinghua.edu.cn}

\date{\today}

\begin{document}
	\maketitle

	\begin{abstract}
		We consider the Dirichlet problem for elliptic systems with periodically distributed inclusions whose conduction parameter exhibits a significant contrast compared to the background media. We develop a unified method to quantify the convergence rates both as the periodicity of inclusions tends to zero and as the parameter approaches either zero or infinity. Based on the obtained convergence rates and a Campanato-type scheme, we also derive the regularity estimates that are uniform both in the periodicity and the contrast.
		
		
		\smallskip
		
		\noindent{\bf Mathematics subject classification (MSC 2020)}: 35B27, 35J70, 35P20
	\end{abstract}
	
	%\tableofcontents
	
	\section{Introduction}
	In this paper we study quantitative homogenization of elliptic systems with rapidly oscillating and high-contrast periodic coefficients. 

To state the problem in precisely, we need to impose two assumptions: the coefficient condition ($\mathbf{ H1}$)  and the geometric condition ($\mathbf{{H2}}$). Throughout this paper, we will assume that $d\geq 2$ and $m\geq 1$ are integers.

	\noindent{({\bf H1}) \emph{Assumptions on the coefficients.} We assume the coefficient 
 $$A(y) = \big( a_{ij}^{\alpha \beta}(y)\big):\mathbb{R}^d \rightarrow \mathbb{R}^{(dm)^2}, \qquad 1\leq i,j\leq d,\   1\leq \alpha ,\beta \leq m$$
satisfies
	\begin{itemize}
\item[(a)] (Symmetry) For any $y \in \mathbb{R}^d$,
\begin{equation}\label{symmetry}
    a_{ij}^{\alpha \beta} (y) = a_{ji}^{ \beta \alpha} (y) ,\qquad 1\leq i,j\leq d,\   1\leq \alpha ,\beta \leq m.
\end{equation}
\item[(b)] (Strong ellipticity) There exists $\mu>0$ such that 
		\begin{equation}\label{ellipticc}
			\mu |\xi|^2 \leq a_{ij}^{\alpha \beta}(y )\xi_i^{\alpha}\xi_j^{\beta} \leq \mu^{-1} |\xi|^2, \qquad \forall y\in \mathbb{R}^d, \ \forall \xi  \in \mathbb{R}^{dm}.
		\end{equation} 
		\item[(c)] (Periodicity) 
		\begin{equation}\label{periodicc}
			A(y+\mathbf{n}) = A(y), \qquad \forall y\in \mathbb{R}^d , \ \forall \mathbf{n}\in \mathbb{Z}^d.
		\end{equation} 
		\item[(d)] (H\"{o}lder continuity)  There exists $\lambda \in (0,1)$ and $\tau>0$ such that
  \begin{equation}\label{smoothc}
			|A(y)-A(w)|\leq \tau |y-w|^{\lambda},\qquad \forall y,w\in \mathbb{R}^d.
		\end{equation}
	\end{itemize}

 \noindent{({\bf H2}) \emph{Assumptions on the Geometric set-up.} Let $Y=(-1/2,1/2)^d$ denote the unit cell,
	and let $\omega \subset Y$ be a simply connected domain with Lipschitz boundary such that $\mathrm{dist}(\omega,\partial Y)>0$. Consider $\Omega$ as a bounded Lipschitz domain in $\mathbb{R}^d$ (sometimes we may impose more regularity on the boundary $\partial \Omega $ such that $\Omega$ is a $ C^{1,\alpha}$ domain). For each $\varepsilon \in (0,1)$ and $\mathbf{n}\in \mathbb{Z}^d$, let $\omega_\eps^{\mathbf{n}} := \eps(\mathbf{n}+\omega)$ be a rescaled and translated version of $\omega$, for a fixed $\kappa >0$, we then define the indices set
 $$
\mathcal{I}_\eps = \{\mathbf{n} \in \Z^d \,:\, \omega^{\mathbf{n}}_\eps \subset \Omega \,\text{ and } \, \mathrm{dist}(\omega^{\mathbf{n}}_\eps,\partial\Omega) >\kappa \eps\},
$$
and let
	\begin{equation}
	\label{eq:DOeps}
		D_{\varepsilon} := \bigcup\limits_{\mathbf{n}\in \mathcal{I}_\eps} \omega^{\mathbf{n}}_\eps.
	\end{equation}
 The set $D_{\eps}$ consists of periodic inclusions of size $\varepsilon$ within $\Omega$ which are well-separated from $\partial \Omega$. The complementary domain $\Omega_{\varepsilon} := \Omega \setminus \overline{D_{\varepsilon}}$ then represents the background matrix surrounding these inclusions.

 The pair $(\Omega, D_{\varepsilon} )$ satisfying the geometric condition ($\mathbf{GC}$) is referred to as the type II perforated domain, as defined in \cite{oleinik_mathematical_1992}. In comparison, the type I perforated domain refers to the following construction:
 \begin{equation}\label{type I}
     D_{\varepsilon} : = \Omega \cap \bigcup_{\mathbf{n} \in \mathbb{Z}^d} \omega^{\mathbf{n}}_{\varepsilon}, \qquad \Omega_{\varepsilon} := \Omega \setminus \overline{D_{\varepsilon}} ,
 \end{equation}
 see Figure \ref{figure 1}.

  \begin{figure}\label{figure 1}
		\centering
\subfigure[Type I domain]{
			\begin{minipage}[t]{0.45\textwidth}
				\centering
				\includegraphics[scale=0.4]{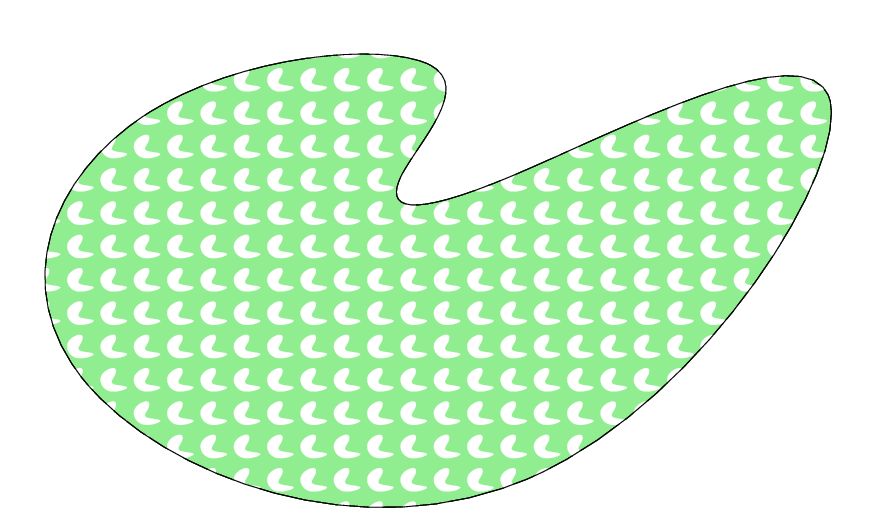}
				%\caption{fig1}
			\end{minipage}
		}
		\hfill
		\subfigure[Type II domain]{
			\begin{minipage}[t]{0.45\textwidth}
				\centering
				\includegraphics[scale=0.4]{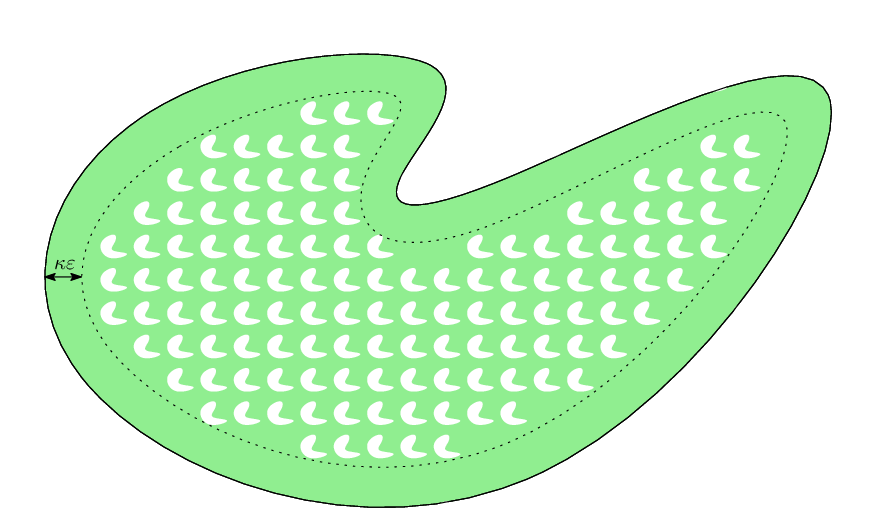}
				%\caption{fig2}
			\end{minipage}
		}
  \caption{Comparison between Type I and Type II domain}
		\centering
\end{figure}

 We are now in the position of describing our problem. For any $\delta \in (0,\infty)$, let 
 $$\Lambda_{\eps,\delta}(x) := \delta \mathbbm{1}_{D_\eps}(x) + \mathbbm{1}_{\Omega_\eps}(x)$$ be the contrast function, where $\delta$ denotes the contrast of physical parameters between inclusions and background (in practice, $\delta$ could be compressibility, permeability, conductivity and etc.). We study the quantitative homogenization of the Dirichlet problem
	\begin{equation}\label{model}
		\left\{
		\begin{aligned}
			& \mathcal{L}_{\varepsilon,\delta}( u_{\varepsilon,\delta} ) = f && \mathrm{in} \ \Omega, \\
			& u_{\varepsilon,\delta}   = g && \mathrm{on}\ \partial \Omega.
		\end{aligned}
		\right.
	\end{equation}
where $\mathcal{L}_{\eps,\delta}$ is given by
\begin{equation*}
 \mathcal{L}_{\eps,\delta} (u) := -\mathrm{div} \left[ \Lambda_{\varepsilon,\delta} (x) A\left( \frac{x}{\varepsilon}\right) \nabla u \right] = -\frac{\partial}{\partial x_i} \left(\Lambda_{\eps,\delta}(x)a^{\alpha\beta}_{ij}\left(x/\eps\right)\frac{\partial}{\partial x_j} u^\beta\right) ,
\end{equation*}
	which acts on functions defined on $\Omega \subset \R^d$ with values in $\R^m$. 
Note that since we allow $\delta$ approaches extremal value ($0$ or $\infty$), the coefficient $\Lambda_{\varepsilon ,\delta} (x)A(x/\varepsilon)$ exhibits high contrast between inclusions and background, and thus $\mathcal{L}_{\varepsilon ,\delta}$ models the media with high-contrast small inclusions.

For each fixed $\delta \in (0,\infty)$, it is well known that the periodic homogenization is valid: as $\varepsilon\rightarrow 0$, the solution $u_{\varepsilon,\delta}$ converges (e.g., in $L^2(\Omega)$) to $\widehat{u}_{\delta}$, the solution of the \textit{homogenized} problem 
	\begin{equation}\label{homogenized problem}
		\left\{
		\begin{aligned}
			& \widehat{\mathcal{L}}_{\delta}( \widehat{u}_{\delta} ) = - \mathrm{div}\,( \widehat{A}_{\delta} \nabla \widehat{u}_{\delta} )= f && \mathrm{in} \ \Omega, \\
			& \widehat{u}_{\delta}   = g && \mathrm{on}\ \partial \Omega,  
		\end{aligned} 
		\right. 
	\end{equation}
where $\widehat{\mathcal{L}}_{\delta} = - \mathrm{div}\,( \widehat{A}_{\delta} \nabla ) $ is the homogenized operator with coefficient
 $\widehat{A}_{\delta} = \big(\widehat{a}_{\delta,ij}^{\alpha \beta}\big)$ defined by
 \begin{equation}
 \label{eq:Abardel}
 	\widehat{a}_{\delta,ij}^{\alpha \beta} = \int_Y \Lambda_{\delta}(y) \Big[ a_{ij}^{\alpha \beta}(y) + a_{ik}^{\alpha \gamma}(y)   \frac{\partial}{\partial y_k} \chi_{\delta,j}^{\gamma \beta}\Big]\,dy. 
 \end{equation}
 Here $\Lambda_{\delta}(y) = \delta \mathbbm{1}_{\omega}(y) + \mathbbm{1}_{Y\setminus \overline{\omega}}(y)$ is the contrast function in the unit cell, and for each $i$ and $\alpha$, $\chi_{\delta,i}^{\alpha }= \big(\chi_{\delta,i}^{\alpha 1}, \cdots, \chi_{\delta,i}^{\alpha m}\big) \in H^1(\mathbb{T}^d;\R^m)$ is the solution of the \textit{cell problem}
	\begin{equation}
		\label{eq:cellproblem}
		\left\{
		\begin{aligned}
			& \mathcal{L}_{1,\delta} \big(\chi_{\delta,i}^{\alpha} + y_ie^{\alpha}\big)  =0 \qquad \mathrm{in}\ Y, \\
			& \chi_{\delta,i}^{\alpha} \ \mathrm{is \ } Y \mathrm{- periodic\ and \ mean \ zero}, 
		\end{aligned}
		\right.
	\end{equation}
	where $e^{\alpha} = (0,\cdots, 1 ,\cdots, 0) \in \mathbb{R}^m  $ with $1$ in the $\alpha^{\mathrm{th}}$ position. We denote the tensor $\big(\chi_{\delta,i}^{\alpha \beta} \big)$ by $\chi_{\delta}  $. Note that, for $\delta=1$, $\chi_1$ is the solution to the classical cell problem in standard periodic homogenization. The optimal $H^1$ convergence rate $O(\varepsilon)$ with some proper correctors and the regularity estimates for $u_{\varepsilon,\delta}$ are also obtained, see the monograph \cite{shen_periodic_2018} and references therein.
	
	It is natural to ask whether the above quantitative results are uniform for $\delta \in (0,\infty)$. Such problems are basic for understanding physical and engineering models with rapidly oscillating and high-contrast coefficients. For instance,
	\begin{itemize}
		\item[(a)] In biomedical imaging, there is a need to probe anomalous tissues, such as tumors or thrombus. The elastic modulus of these anomalous tissues is often significantly greater than that of the surrounding normal tissues, as observed in \cite{manduca2001magnetic}.
		\item[(b)] Designing meta-materials exhibiting novel electromagnetic or elastic properties. Meta-materials are composite materials typically composed of arrays of small resonators structured on the microscale, the resonance occurs when the electromagnetic or elastic modulus of resonators are high-contrast compared to that of the background, see \cite{sakoda2004optical, MR3769919}.
		\item[(c)] The double porosity problem, studied by Arbogast, Douglas, and Hornung \cite{doi:10.1137/0521046}, which models single phase flow in fractured porous media. In their context, $u_{\varepsilon,\delta}$ is the fluid pressure, and $\delta = \varepsilon^2$ is the permeability that is much smaller in the porous rocks $D_{\varepsilon}$ than in the network of fractures $\Omega_{\varepsilon}$.
	\end{itemize}

\subsection*{Main results.} The first main result of this article concerns the uniform convergence rate of $u_{\varepsilon,\delta}$ in $0<\delta<\infty$. As far as we know, this is the first result in the full inhomogeneous boundary value problem (that is, with nonzero source term and boundary data) in the high contrast setting.
	
	\begin{theorem}\label{thm:suboptimal convergence rate}
		Assume \eqref{symmetry}-\eqref{periodicc} and  $\mathbf{(H1)}$, let $u_{\varepsilon,\delta}$ be the solution of \eqref{model} and $\widehat{\mathcal{L}}_\delta^{-1}(f)$ be the solution of \eqref{homogenized problem}. Let $v=\mathcal{L}_{D_{\varepsilon}}^{-1}(f) \in H^1_0(\Omega)$ be the unique solution of 
	\begin{equation}\label{auxioperator}
		\left\{
		\begin{aligned}
			& \mathcal{L}_{\eps,1}(v)= f & \mathrm{in} \ D_{\varepsilon} , \\
			& v=0& \mathrm{on}\ \Omega \setminus D_{\varepsilon}.
		\end{aligned}\right.
	\end{equation} 
Then we have 
		\begin{equation}
		\label{eq:thm1-error}
			\begin{aligned}
				\Big\|   u_{\varepsilon,\delta} - \widehat{u}_{\delta} -  \varepsilon \chi_{\delta}\left(\frac{x}{\varepsilon}\right) S_{\varepsilon} \Big(\eta_{\varepsilon}\nabla \widehat{u}_{\delta}\Big) &- \delta^{-1} \mathcal{L}_{D_{\varepsilon}}^{-1}(f) \Big\|_{H^1(\Omega)} \\
				&\leq C\varepsilon^{1/2} \big( \|f \|_{L^2(\Omega) } +\|\nabla_{\mathrm{tan}}g\|_{L^2(\partial \Omega)} \big).
			\end{aligned}
		\end{equation}
		Here $C$ depends only on $d, m, \kappa, \mu,\Omega$ and $\omega$, $S_{\varepsilon}$ is a smoothing operator defined in \eqref{epsi_smoother}, and $\eta_{\varepsilon}$ is a cut-off function supported away from $\partial \Omega$ defined in \eqref{cut-off_def}.
	\end{theorem}
	\begin{remark}\label{rem:fterm}
		By a simple rescaling argument, it is clear that
		\begin{equation*}
			\varepsilon^{-2}\|  \mathcal{L}^{-1}_{D_{\varepsilon}} (f) \|_{L^2(\Omega)} +\varepsilon^{-1}\|  \nabla \mathcal{L}^{-1}_{D_{\varepsilon}} (f) \|_{L^2(\Omega)}   \leq C \|f\|_{L^2(\Omega)} ,
		\end{equation*}
		so if $\delta \gg \varepsilon$, the term $\delta^{-1} \mathcal{L}_{D_{\varepsilon}}^{-1}(f)$ is negligible as $\varepsilon\rightarrow 0$. Similarly, if we modify $f$ in \eqref{model} by 
\begin{equation}
\label{eq:modf}
f_{\eps,\delta}(x) = \begin{cases} f(x), \qquad &\text{for } \, \delta\ge 1,\\
\Lambda_{\eps,\delta}(x) f(x), \qquad &\text{for }\, \delta\in (0,1),
\end{cases}
\end{equation}
and replace $\widehat{u}_\delta$ by $\widehat{\mathcal{L}}_\delta^{-1}(f_{\eps,\delta})$ in \eqref{eq:thm1-error}, we can drop the term $\delta^{-1}\mathcal{L}_{D_\eps}^{-1}(f)$ there as well.
	\end{remark}

% \begin{remark}
% 	Condition \eqref{distancecondition} is necessary in our proof. Without \eqref{distancecondition}, the exterior boundary of $\Omega_{\varepsilon}$ might be very irregular, even not Lipschitz. This would cause difficulty in the establishment of energy comparison between $D_{\varepsilon}$ and $\Omega_{\varepsilon}$, which is a key ingredient of our proof.
% \end{remark}	%WJ: not true.

The second main objective of this paper is to establish uniform regularity estimates for $u_{\varepsilon,\delta}$ both in $0<\varepsilon<1$ and $0<\delta<\infty$. From the aforementioned convergence rate result, it is reasonable that we replace $f$ by $f_{\eps,\delta}$ for $\delta \in (0,1)$.

% \begin{theorem}[Interior Lipschitz estimate]
% 		\label{interiorlipshitz}
% 		Assuming that $A \in \Lambda(\mu,\lambda,\tau) \cap {\color{blue}H^1(Y)}$, let $u_{\varepsilon,\delta}$ be a weak solution of $\mathcal{L}_{\varepsilon,\delta}(u_{\varepsilon,\delta}) =0$ in $2B$, where $B=B_R(x_0)$ for some $x_0 \in \mathbb{R}^d$ and $R>0$. Then
% 		\begin{equation*}
% 			\|\nabla u_{\varepsilon,\delta}\|_{L^{\infty}(B)}\leq C\| \nabla u_{\varepsilon,\delta} \|_{\underline{L}^2(2B)} .
% 		\end{equation*}
% \end{theorem}

% Fix $x_0 \in \partial \Omega$, we denote $D_R$ and $\Delta_R$ by
% \begin{equation}
% 	D_R = B_R(x_0) \cap \Omega , \qquad \Delta_R = B_R(x_0)\cap \partial \Omega,
% \end{equation}
% respectively.

\begin{theorem}
\label{thm:globalLip} Assume $(\mathbf{H1})$ and $(\mathbf{H2})$, further assume $\Omega, \omega$ are $C^{1,\alpha}$ domains for some $\alpha \in (0,1)$, $f\in L^p(\Omega)$ for some $p\ge d$ and $g\in C^{1,\alpha}(\partial \Omega)$. %WJ: We may consider $g\ne 0$ later once we are familar with the details.
For each $\eps \in (0,1)$ and $\delta\in (0,\infty)$, let $f_{\eps,\delta}$ be defined by \eqref{eq:modf}, in other words, for relatively soft inclusions, we assume that the source term inside $D_\eps$ is of order $O(\delta)$. Then there is a constant $C>0$ that depends only on $\kappa,\mu,\lambda, \tau,d,m,p,\alpha,\Omega$ and $\omega$ such that if $u_{\eps,\delta}$ solves \eqref{model} with $f$ replaced by $f_{\eps,\delta}$, then we have 
	\begin{equation}\label{eq:globalLip}
		\|\nabla u_{\eps,\delta}\|_{L^\infty(\Omega)} \le C\left( \|f\|_{L^p(\Omega)} + \|g\|_{C^{1,\alpha}(\partial \Omega)}\right).
		%C\left\{\|f\|_{L^p(\Omega)} + \|g\|_{C^{1,\rho}}\right\}.
	\end{equation}
\end{theorem}

This is a global Lipschitz regularity estimate for $u_{\eps,\delta}$ that is uniform both in $\eps$ and $\delta$. The interior version was essentially established in Shen's work. We base our work on Shen's and establish the boundary regularity by carrying out the approximation scheme. 
	
\begin{remark}\label{rem:extreme} Several remarks are in order.

(1) It is possible to include the two extreme settings, namely $\delta = 0$ and $\delta = \infty$ in our analysis. For $\delta = 0$, we consider the heterogeneous elliptic system posed on the perforated domain $\Omega_\eps$ with zero Neumann boundary condition on the interior boundary $\Omega\cap \partial \Omega_\eps$:
\begin{equation}\label{eq:het0}
		\left\{
		\begin{aligned}
			& -\partial_i\left(a^{\alpha\beta}_{ij}(\tfrac{x}{\eps})\partial_j u^\beta_\eps\right) = f^\alpha && \mathrm{in} \ \Omega_\eps, \\
			& \nu^i a^{\alpha\beta}_{ij}(\tfrac{x}{\eps})\partial_j u^\beta_\eps(x) = 0  && \mathrm{on} \ \Omega \cap \partial \Omega_\eps,\\
			& u_{\varepsilon,\delta}   = g && \mathrm{on}\ \partial \Omega.
		\end{aligned}
		\right.
	\end{equation}
For $\delta = \infty$, we consider the so-called stiff inclusion problem: find $u_\eps = (u_\eps^\beta)$ in $H^1_0(\Omega)$ so that $u_\eps$ is a constant vector in each of the inclusion $\omega_\eps^{\mathbf{n}}$ of $D_\eps$, and $u_\eps$ satisfies
\begin{equation}\label{eq:hetinfty}
		\left\{
		\begin{aligned}
			& -\partial_i\left(a^{\alpha\beta}_{ij}(\tfrac{x}{\eps})\partial_j u^\beta_\eps\right) = f^\alpha && \mathrm{in} \ \Omega_\eps, \\
			& \int_{\partial \omega_\eps^{\mathbf{n}}} e^\alpha \cdot \left.\left(\frac{\partial u_\eps}{\partial \nu_\eps}\right)\right\rvert_{+}\,d\sigma  = \int_{\omega_\eps^{\mathbf{n}}} e^\alpha \cdot f  && \text{for each } \omega_\eps^{\mathbf{n}} \subset D_\eps,\\
			& u_{\varepsilon,\delta}   = g && \mathrm{on}\ \partial \Omega.
		\end{aligned}
		\right.
	\end{equation}
The quantitative homogenization of \eqref{eq:het0} is similar to the analysis in \cite{MR4567768,russell_quantitative_2018}; that for the stiff inclusion problem is similar to \cite{MR548785}. In particular, the homogenized tensors are $\widehat A_0$ and $\widehat A_\infty$ defined by \eqref{homotensor0} and \eqref{homotensorinfy}, respectively. 

(2) The ellipticity condition \eqref{ellipticc} is often referred to as \emph{strong} ellipticity, to be distinguished from the \emph{elasticity} setting.
It is possible to combine the method of this paper and those of Shen to analyze the same problem for the elasticity system. The main difference is, in the latter setting $m=d$, and the ellipticity is checked only for symmetric $d\times d$ matrices  $\xi$ (this is usually called weak ellipticity or Legendre–Hadamard ellipticity). 
By using Korn's inequalities, one may carry out most of the analysis here in a similar way.
\end{remark}

%%%%%%%	
	\subsection*{History and strategy of the proof.}
	This is a problem in monograph by Jikov, Kozlov and Oleinik \cite{MR1329546}. 
In the scalar case with $0\leq \delta <1$, $A=I$ and $\omega$ being sufficiently smooth, using the compactness method in \cite{https://doi.org/10.1002/cpa.3160400607}, the $W^{1,p}$ and Lipschitz estimates were obtained by Yeh \cite{https://doi.org/10.1002/mma.1163, YEH20111828, MR3401586, MR3494385}. Also see earlier work by Schweizer \cite{https://doi.org/10.1002/1097-0312(200009)} and Masmoudi \cite{57222a4e435842a084866fe8defdbab1} for related uniform estimates in the case $\delta =0$. Russell \cite{russell_homogenization_2017} established the large-scale interior Lipschitz estimate for the system of elasticity with bounded measurable coefficients in the case $\delta =0$, using an approximation method originated by Armstrong and Smart \cite{MR3481355}. The case $0<\delta <1$ was treated by Russell \cite{russell_quantitative_2018}. Shen \cite{shen_homogenization_2021} obtained the boundary regularity estimates and nontangential-maximal-function estimates for H\"{o}lder continuous $A$ with $0<\delta <1$.
	In the stochastic setting with $\delta =0$, Armstrong and Dario \cite{MR3847767} obtained quantitative homogenization and large-scale regularity results for the random conductance model on a supercritical percolation.
	
	For the case $1<\delta <\infty$, Shen \cite{shen_large-scale_2021} established the interior Lipschitz estimates by adapting a recent method proposed by Armstrong et al. \cite{MR4544795}, which is based on a Caccioppoli type inequality and the idea that transfer the higher-order regularity of $u$ in terms of the difference operator to higher-order regularity of $u$ at a large scale through Caccioppoli and Poincar\'{e}'s inequalities.
	
	We briefly sketch our approach to the proof. In this paper, we adapt the approximation method originated by Armstrong and Smart \cite{MR3481355}, which was used by Russell \cite{russell_homogenization_2017, russell_quantitative_2018} for treating the case $0\leq \delta <1$. The main difficulty of this method in the case $1<\delta<\infty $ is the establishment of uniform $H^1$ convergence rate. To this end, we use the periodic layer potentials to obtain the uniform controll for quantities associated to the cell problem, i.e., the homogenized tensor $\widehat{A}$, the flux corrector and etc. This combine with a standard elliptic estimate yield the $H^1$ convergence rate in $D_{\varepsilon}$. Then a key observation is that the discrepancy function could be decomposed to a regular one and a singular one, the regular part is easy to be bounded, and $H^1(\Omega_{\varepsilon})$ norm of the singular one is bounded by its $H^1(D_{\varepsilon})$ norm. These complete the proof of $H^1(\Omega)$ convergence rate. Via the results of convergence rates, a Caccippoli-type inequality by Shen \cite{shen_large-scale_2021} and its variants, and an approximation method (also known as a Campanato-type scheme) by Armstrong and Smart \cite{MR3481355}, we establish the regularity estimates.

	\textbf{Notations.} We collect some notations used throughout this paper. 
\begin{itemize}
    \item 	The boundary $t$-layer of $\Omega$ is denoted by $\mathcal{O}_t = \{ x\in \Omega: \mathrm{dist}(x,\partial \Omega) <t\}$. 
    \item We use the following notation for the product of $A$ with $\nabla u$:
\begin{equation*}
A\nabla u = (a^{\alpha\beta}_{ij}\partial_j u^\beta)_i^\alpha, \quad A\nabla u \nabla v=a^{\alpha\beta}_{ij}\partial_j u^\beta \partial_i v^\alpha.
\end{equation*}
\item The open cube centered at $x_0$ with radius $R$ is denoted by $Q_R(x_0)=\{x \in \mathbb{R}^d:\|x-x_0\|_{\infty} < R\}$. Mention of the center $x_0$ is omitted whenever this does not cause confusion.
\item  We use $\mathcal{L}$ to denote the operator $ \mathcal{L}_{1,1} = - \mathrm{div}\,(A\nabla)$. 
\item 	We use 
	\begin{equation}
		\|f \|_{\underline{L}^p(E)}  = \Big(\fint_E |f|^p\,dx\Big)^{1/p} = \Big(|E|^{-1} \int_E |f|^p\,dx\Big)^{1/p} 
	\end{equation}
	for $1<p<\infty$, to denote the $L^p$ mean of $f$ over a set $E$.
 \item  We use
	\begin{equation}
		J^E(u,v) = \int_E a_{ij}^{\alpha \beta}(x) \frac{\partial u^{\beta}}{\partial x_j}(x)\frac{\partial v^{\alpha}}{\partial x_i}(x)\,dx
	\end{equation}
	to denote the bilinear form of $u$ and $v$ over a set $E$, we write $J^E(u)=J^E(u,u)$. By energy of $u$ in a set $E$ we usually mean the $L^2(E)$ norm of $\nabla u$. Due to uniform ellipticity, it is roughly $J^E(u)$. 
	\item  We denote $\varepsilon(\omega +\mathbf{n})$ by $\omega_{\varepsilon}^{\mathbf{n}}$ and $\varepsilon(Y +\mathbf{n})$ by $Y_{\varepsilon}^{\mathbf{n}}$ for $\mathbf{n} \in \Z^d$.
 \item $\widetilde{\omega}$ denotes an enlarging dilation of $\omega$ so that $\overline{\omega}\subset \widetilde{\omega} \subset  Q_{7/8}(0)$. Similarly, $\widetilde{\omega}^{\mathbf{n}}_\eps$ denotes $\eps(\mathbf{n}+\widetilde{\omega})$. 
 \item We use the flat-torus $\mathbb{T}^d$ to represent the unit cell $Y$ if the periodic boundary condition is emphasised. 
 \item  We use $C$ to denote a positive constant that may depends on $d, m,\kappa, \mu,\lambda,\tau,\omega,\widetilde{\omega}$ and $\Omega$, but independents on $\varepsilon$ and $\delta$. If $C$ depends also on other parameters, it will be stated explicitly. The exact value of $C$ may change from line to line.
\end{itemize}

	\textbf{Organization of the paper.} The rest of the paper is organized as follows. In section \ref{preliminaries} we collect some results on periodic layer potentials and extension operators that are useful to deal with functions on perforated domains, together with basic elliptic estimates. In section \ref{uniformcontrollcell} we obtain uniform estimates for the cell problems in the high contrast setting, and establish convergence rates in the high contrast limit. We then give the proof of Theorem \ref{thm:suboptimal convergence rate} in section \ref{proofof theorem1}. The Caccioppoli estimates in the high contrast setting are established in section \ref{sec:caccioppoli}, and we prove in section \ref{sec:prooflipschitz} the uniform regularity of the elliptic system in the high contrast setting.

	\section{Layer potentials, Extension operators and basic elliptic estimates}
	\label{preliminaries}

	\subsection{Periodic layer potential}
	\label{sec:periodiclayerpotential}
	
	Let the coefficient $A $ satisfies \eqref{ellipticc}-\eqref{smoothc}, in view of the interior Lipschitz regularity for $\mathcal{L}$, proved by Avellaneda and Lin \cite{https://doi.org/10.1002/cpa.3160400607}, one has the periodic fundamental solution $\Gamma(x,y) = (\Gamma^{\alpha \beta}(x,y)): \mathbb{T}^d \times \mathbb{T}^d \rightarrow \mathbb{R}^{m\times m}$, i.e., 
	\begin{equation}\label{existenceperiodicfs}
		\left\{
		\begin{aligned}
			& \mathcal{L}(x) \big(\Gamma^{\alpha}(x,y)\big)= (\delta_y(x) -1) e^{\alpha} ,& (x, y) \in Y \times Y, \\
			& \int_Y \Gamma^{\alpha} (x,y)\,dx = 0 , & y\in Y,
		\end{aligned}\right.
	\end{equation}
	where $\delta_y(x)$ is the Dirac measure massed at $y$, and the estimates
	\begin{equation}\label{periodicfunctionestimate}
		\left\{
		\begin{aligned}
			& |\Gamma(x,y)| \leq C|x-y|^{2-d}, \\
			& |\nabla_x \Gamma(x,y)| + |\nabla_y \Gamma(x,y)|\leq C|x-y|^{1-d}.
   \end{aligned}\right.
	\end{equation}
The function $\Gamma(x,y) $ is called the periodic fundamental solution for the operator $\mathcal{L}$ with pole at $y$. 
	For construction of the periodic fundamental solution and estimates \eqref{periodicfunctionestimate}, we refer to Hofmann and Kim \cite{hofmann2007green}, where they treat the case of zero Dirichlet boundary condition, but their methods apply also in the periodic setting.
	
	For $f \in L^2(\partial \omega)$, based on the estimates \eqref{periodicfunctionestimate}, one can define the single-layer potential $\mathcal{S}f$ by
	\begin{equation}
		\big( \mathcal{S}f  \big)^{\alpha}(x) = \int_{\partial \omega} \Gamma^{\alpha \beta}(x,y) f^{\beta}(y)\,d\sigma(y)
	\end{equation}
	while the Neumann-Poincar\'{e} (NP) operator $\mathcal{K}^*$ is defined by
	\begin{equation}
	\label{eq:NP}
		\big( \mathcal{K}^*f \big)^{\alpha} (x)= \mathrm{P.V.} \int_{\partial \omega}  \left( \frac{\partial \Gamma^{\beta}}{\partial \nu_x}\right)^{\alpha } (x,y) f^{\beta}(y) \,d\sigma(y),
	\end{equation}
	where P.V.\,stands for the Cauchy principal value integral. For more regular $\omega$, say $\partial \omega \in C^{1,\alpha}$ for some $\alpha \in (0,1)$, the above can be replaced by the regular integral.
	
	Denote by $L^2_0(\partial \omega)$ the space $\{\phi \in L^2(\partial\omega) \,:\, \int_{\partial \omega} \phi=0\}$. The following properties are standard: 
	\begin{itemize}
		\item[(a)] For $\psi \in L^2(\partial \omega)$, $\mathcal{S}\psi \in H^1(\mathbb{T}^d)$.
		\item[(b)] For $\psi \in L^2(\partial \omega)$, jump relation holds, i.e.,
			\begin{equation}
			\frac{\partial}{\partial \nu} \big(\mathcal{S}\psi \big) \Big|_{\omega\pm} = \Big(\pm \frac{1}{2}I + \mathcal{K}^* \Big)\psi.
		\end{equation}
	\item[(c)] For $\psi \in L^2_0(\partial \omega)$, $\mathcal{L}(\mathcal{S}\psi) =0$ in $\mathbb{T}^d \setminus \partial \omega$.
	\item[(d)] The spectrum of $\mathcal{K}^*$ on $L^2_0(\partial \omega)$ is contained in $(-1/2,1/2)$.
	\end{itemize}
We refer to \cite{kenig_layer_2011,MR4075336} and \cite[Section 2.8]{MR2327884} for more results on layer potential theory with varying coefficients (i.e., for general elliptic operators). Note that above and in the rest of the paper,
\begin{equation*}
 \frac{\partial}{\partial \nu} w = \left((\partial_\nu w)^\alpha\right), \quad \text{with}\  \ (\partial_\nu w)^\alpha:= \sum_{i=1}^d \nu_i A^{\alpha\beta}_{ij} \partial_j w^\beta.
\end{equation*} 

\subsection{Extension operators}

The following extension operator is useful in the treatment of perforated domains. The problem setting for high contrast inclusions is in a similar situation.
Recall that 
$\widetilde{\omega}$ denotes an enlarging dilation of $\omega$ so that $\overline{\omega}\subset \widetilde{\omega} \subset  Q_{7/8}(0)$.

\begin{lemma} \label{lem:extension} Assume $(\mathbf{H2})$, there exists a linear extension operator $\mathcal{P}_{\varepsilon}:H^1(\Omega_{\varepsilon}) \rightarrow H^1(\Omega)$ such that for any $v \in H^1(\Omega_{\varepsilon})$,
	\begin{equation*}
		\mathcal{P}_{\varepsilon} v= v \quad  \mathrm{in}\  \Omega_{\varepsilon} \qquad \mathrm{and} \qquad\| \nabla \mathcal{P}_{\varepsilon} v \|_{L^2(\Omega)}\leq C \| \nabla v\|_{L^2(\Omega_{\varepsilon})}.
	\end{equation*}
 %\item[(2)] Under Assumption {\upshape{\bf G}(d)}, there is a linear extension operator $\mathcal{P}_{\varepsilon}:H^1(\Omega) \rightarrow H^1(\Omega)$ so that, if $v\in H^1(\Omega)$ and $g = v\rvert_{\partial \Omega}$ is its trace on $\partial \Omega$, then
 %$$
 %\mathcal{P}_{\varepsilon} v= v \quad  \mathrm{in}\  \Omega_{\varepsilon} \qquad \mathrm{and} \qquad\| \nabla \mathcal{P}_{\varepsilon} v \|_{L^2(\Omega)}\leq C \left(\| \nabla v\|_{L^2(\Omega_{\varepsilon})} + \|g\|_{H^{\frac12}(\partial \Omega)}\right).
 %$$
 %where $C<\infty$ is a universal constant.
 %\end{itemize}
\end{lemma}
\begin{proof}
We refer to \cite{MR548785} or \cite[Lemma 3.3]{MR1329546}; see also \cite[Appendix B]{MR3479187}. The basic idea is, at the unit scale around a typical obsctale $\omega$ and for a function $v\in H^1(\womega\setminus\ol \omega)$ with mean value $c=\fint_{\widetilde\omega\setminus \ol \omega} v$, 
we can apply the usual extension operator to $v-c$ and obtain a $v_1 \in H^1(\womega)$ so that 
$$
\|v_1\|_{H^1(\womega)} \le \|v-c\|_{H^1(\womega\setminus\ol \omega)} \le C\|\nabla v\|_{L^(\womega\setminus \ol\omega)},
$$
where the last step results from Poincar\'e-Wirtinger inequality. Then $\tilde v := c+v_1$ is an extension of $v$ from $\womega\setminus \ol\omega$ to $\womega$ and satisfies 
$$
\|\tilde v\|_{L^2(\widetilde\omega)}\le C\|v\|_{L^2(\widetilde\omega\setminus \ol \omega)},\quad\text{and}\quad 
\|\nabla \tilde v\|_{L^2(\widetilde\omega)}\le C\|\nabla v\|_{L^2(\widetilde\omega\setminus \ol \omega)}.
$$
Under $(\mathbf{H2})$, the desired extension $\mathcal{P}_\eps$ in the whole domain $\Omega_\eps$ is by gluing the rescaled local extensions together.
%
%Under  {\bf G}(d), let $\widetilde\Omega$ be an enlarging dialation of $\Omega$ so that $\mathrm{dist}(\Omega,\partial\widetilde\Omega) \ge 2$. We first use the inverse trace theorem and the standard extension to find a $\tilde v \in H^1(\widetilde\Omega)$ satisfying
%\begin{equation*}
%    \|\tilde v\|_{H^1(\widetilde \Omega)} \le  C\|g\|_{H^{1/2}(\partial \Omega)}.
%\end{equation*}
%Let $w = v\mathbf{1}_{\Omega} + \tilde v\mathbf{1}_{\widetilde\Omega\setminus \Omega}$. Now for the inclusions $\omega^{\mathbf{n}}_\eps$ that cut $\partial \Omega$, we add the full of $\omega^{\mathbf{n}}_\eps$ to $D_\eps$ and denote the extended inclusions set by $\widetilde D_\eps$. Then we check that $w\in H^1(\widetilde\Omega)$ and
%$$
%\|\nabla w\|_{L^2(\widetilde\Omega\setminus \widetilde D_\eps)} \le C\left\{\|\nabla v\|_{L^2(\Omega_\eps)} + \|g\|_{H^{1/2}(\partial \Omega)}\right\}.
%$$
%Since $\widetilde D_\eps$ is now well separated from $\partial\widetilde \Omega$, the extensions earlier apply and we get the desired result.
\end{proof}

As we will see, the extension above is useful for small $\delta$ in the high contrast setting of this paper. For large $\delta$, we need the following operator that extends functions defined in the inclusions with zero mean in each inclusion to the outside.

\begin{lemma}\label{lem:extendout} 
	Assume $(\mathbf{H2})$, there exists a linear map $\mathcal{Q}_{\varepsilon}:H^1(D_\eps) \rightarrow H^1_0(\Omega)$
 such that  $\mathcal{Q}_\eps(v)$ is supported in $\cup_{\mathbf{n}} \widetilde \omega^{\mathbf{n}}_\eps$, and 
 \begin{equation*}
     \mathcal{Q}_\eps(v) = v - q^{\mathbf{n}}_{\varepsilon}, \qquad q^{\mathbf{n}}_{\varepsilon} = \int_{\omega^{\mathbf{n}}_{\varepsilon}} v , \qquad \mathrm{in\ each\ }  \omega^{\mathbf{n}}_{\varepsilon} \subset D_{\varepsilon}.
 \end{equation*}
 Moreover,
	\begin{equation*}
	\|\mathcal{Q}_\eps(v)\|_{H^1(\Omega)} \le C\|\nabla v\|_{L^2(D_\eps)}.
	\end{equation*}
\end{lemma}
\begin{proof} As usual we first consider an extension operator in the unit scale. Suppose $\omega$ is the model inclusion of unit length, and $\widetilde \omega$ is the enlarged inclusion. Then there is a standard linear extension operator $\mathcal{Q}$ from $H^1(\omega)$ to $H^1_0(\widetilde \omega)$ so that 
$$
\|\mathcal{Q}(v)\|_{H^1(\widetilde\omega)} \le C\|v\|_{H^1(\omega)}.
$$
Since the Poincar\'e-Wirtinger inequality holds for $v- q$, where $q = \int_{\omega} v$, we get
$$
\|\mathcal{Q}(v-q)\|_{H^1(\widetilde\omega)} \le C\|v-q\|_{H^1(\omega)} \le C\|\nabla v\|_{L^2(\omega)}.
$$

In each small inclusion $\omega_\eps^{\mathbf{n}}$ that is contained in $\Omega$, we let
\begin{equation*}
\mathcal{Q}^{\mathbf{n}}_\eps(v)(x) := \mathcal{Q}\Big(v(\eps (\cdot + \mathbf{n})) -q^{\mathbf{n}}_{\varepsilon} \Big)\Big(\frac{x-\eps \mathbf{n}}{\eps}\Big), \qquad x\in \widetilde\omega^{\mathbf{n}}_\eps.
\end{equation*}
Then we see that $\mathcal{Q}^{\mathbf{n}}_\eps(v)$ belongs to $H^1_0(\widetilde\omega_\eps^{\mathbf{n}})$ and satisfies $\mathcal{Q}^{\mathbf{n}}_\eps(v) = v - q^{\mathbf{n}}_{\varepsilon}$, and
\begin{equation*}
\|\nabla \mathcal{Q}^{\mathbf{n}}_\eps(v)\|_{L^2(\widetilde\omega_\eps^{\mathbf{n}})} \le C\|\nabla v\|_{L^2(\omega_\eps^{\mathbf{n}})},
\end{equation*}
where the constant $C$ does not change due to the scaling-invariant feature of the estimate. For $v\in H^1(D_\eps)$, we use the map piecewise for each component $\omega^{\mathbf{n}}_\eps$ of $D_\eps$ and glue them together, 
we then get a linear map $\mathcal{Q}_{\varepsilon}$ with the desired properties.
\end{proof}

%%%%%%%
\subsection{Other useful tools and estimates}  
\label{sec:tools}

We will use the now standard smoothing operator: fix a function $\xi \in C_0^{\infty}(B(0,\frac{1}{2}))$ such that $0\leq \xi (x)\leq 1$ and $\int_{\mathbb{R}^d}\xi (x)\,dx=1$, for any  $v \in L^1_{\rm loc}(\mathbb{R}^d)$, $S_\eps(v)$ is the smooth function defined by
\begin{equation}\label{epsi_smoother}
	S_{\varepsilon}(v)(x):=\varepsilon^{-d} \int_{\mathbb{R}^d} v(x-y)\xi(y/\varepsilon)\,dy.
\end{equation}
The operator is very useful in quantitative homogenization theory; it allows one to use two-scale expansions because the regularization error can be well controlled. In particular, we have the following.
\begin{lemma}\label{appen_auxi_0}
	Let $f \in L^2(\mathbb{T}^d)$, then for any $g \in L^2(\mathbb{R}^d)$, we have
	\begin{equation*}
		\| f(x/\varepsilon) S_{\varepsilon}(g ) \|_{L^2(\mathbb{R}^d)}\leq C\| f \|_{L^2(Y)} \|g \|_{L^2(\mathbb{R}^d)}.
	\end{equation*}
\end{lemma}
We refer to Shen \cite[Proposition 3.1.5]{shen_periodic_2018} for the proof, together with more useful properties of $S_\eps$. We end the section by the following useful but standard estimate for elliptic systems.
\begin{lemma}\label{elliptic_regularity_estimate}
	Let $A= \big(a_{ij}^{\alpha \beta} \big)$ be a constant tensor satisfying  \eqref{symmetry}-\eqref{ellipticc}, let $\Omega$ be a bounded Lipschitz domain with connected boundary in $\mathbb{R}^d$, and let $u$ be the solution of 
	\begin{equation*}\left\{
		\begin{aligned}
			& -\mathrm{div}\,(A\nabla u) = f & \mathrm{in}\ \Omega, \\
			& u=g & \mathrm{on}\ \partial \Omega.
		\end{aligned}\right.
	\end{equation*}
	then
	\begin{equation*}\| \nabla u \|_{L^2(\mathcal{O}_t)} + t\| \nabla^2 u \|_{L^2(\Omega \setminus \mathcal{O}_t)} \leq C t^{1/2} \big(\|f \|_{L^2(\Omega)} + \|\nabla_{\mathrm{tan}}g\|_{L^2(\partial \Omega)} \big), 
	\end{equation*}
	where $C$ depends only on $\mu$ and the Lipschitz character of $\Omega$.
\end{lemma}
\begin{proof}
	This follows from the interior estimates and the nontangential-maximal-function estimate
	for second-order elliptic equations with constant coefficients. See \cite{kenig1994harmonic} for reference.
\end{proof}
\begin{remark}
	Note that $\Omega$ is a Lipschitz domain, it is impossible to control $\| \nabla^2 u \|_{L^2(\Omega)}$, since the solution may have a singular behavior near the irregular points of $\partial \Omega$, even in the simplest case of the Laplace equation, if $\Omega$ is a non-convex polygon in $\mathbb{R}^2$, it is well known \cite{doi:10.1137/1.9781611972030} that $u\notin H^2(\Omega)$ in general, even if $f\in C^{\infty}(\overline{\Omega})$ and $g=0$.	
\end{remark}

%%%%%%%%%
%%%%%%%%%
\section{Limiting behavior associated to cell problem}\label{uniformcontrollcell}

Our overall strategy in this work is to treat $\eps$ and $\delta$ as two independent parameters. This approach was mentioned already in Chap 3 of \cite{MR1329546}. In general, we can consider the diagram
\begin{equation}
\label{eq:diagram}
\begin{CD}
u_{\eps,\delta} @>{\delta \to 0 \text{ or } \infty}>> u_{\eps}\\
@V{\eps\to 0}VV  @VV{\eps\to 0}V\\
\widehat{u}_{\delta} @>{\delta \to 0 \text{ or }  \infty}>> \widehat{u}
\end{CD}
\end{equation}
We show in this section that $\widehat{u}_\delta$, or the equation for it, behaves quite well uniformly in $\delta$. In the scalar setting, such uniform in $\delta$ results were established in Jikov using variational techniques. Here, we obtain the results for the elliptic system using layer-potential techniques.

Note that to use layer potentials, we require the coefficient $A$ satisfies \eqref{ellipticc}-\eqref{smoothc} (see Section \ref{sec:periodiclayerpotential}), nevertheless, by using a 'weak' form of layer potentials, one can obtain all of results in this section, see Appendix \ref{appen: weak layer potential} for the construction of 'weak' form of layer potentials. For simplicity, we assume \eqref{ellipticc}-\eqref{smoothc} for $A$ in this section, but as we have mentioned, the readers should keep in mind that the results in this section also holds for $A$ satisfying \eqref{ellipticc}-\eqref{periodicc}.

%%%%%%%%
\subsection{Uniform bounds for the cell problems} 

Recall that 
$$
\chi_{\delta} = (\chi_{\delta,j}^\alpha)_{1\le j\le d,1\le \alpha\le m},
$$
denotes the solution to the cell problem that arises in the standard homogenization of \eqref{model} for each fixed $\delta \in (0,1)$. What is most important is to obtain estimates that are either uniform in $\delta$ or have explicit dependence in $\delta$. We achieve this goal through the layer potential representation of the cell problem.

\begin{lemma}
	For fixed $0<\delta <\infty$, let $\chi_\delta = (\chi_{\delta,i}^\alpha)$, $1\le i\le d$ and $1\le \alpha \le m$, be the solution of the cell problem \eqref{eq:cellproblem}. Then for $\delta \ne 1$, $\chi_\delta$ has the folowing layer potential representation:
	\begin{equation}
		\chi_{\delta,i}^{\alpha} =  \chi^{\alpha}_{1,i} + \mathcal{S}\left[\big( \lambda(\delta) I -\mathcal{K}^*  \big)^{-1}  \phi_i^{\alpha }\right] -q_{\delta,i}^{\alpha}, \quad \lambda(\delta) = \frac{1}{2} +\frac{1}{\delta -1},
	\end{equation}
	where 
	\begin{equation}\label{densitydef}
		\phi_i^{\alpha } =\frac{\partial }{\partial \nu}\Big|_+  \big( y_i e^{\alpha}+ \chi_{1,i}^{\alpha} \big)=\frac{\partial }{\partial \nu}\Big|_-  \big( y_i e^{\alpha}+ \chi_{1,i}^{\alpha} \big) \in H_0^{-1/2}(\partial \omega),
	\end{equation}
	and 
	\begin{equation}
		q_{\delta,i}^{\alpha} = \int_Y \mathcal{S}\left[\big( \lambda(\delta) I -\mathcal{K}^*  \big)^{-1}  \phi_i^{\alpha }\right] (y)\,dy.
	\end{equation}
	Here, the periodic single-layer potential operator $\mathcal{S}$ and the associated Neumann-Poincar\'e operator $\mathcal{K}^*$ are defined by \eqref{eq:NP}.
\end{lemma}
\begin{proof}
	Let $u = \chi^{\alpha}_{\delta,i} - \chi^{\alpha}_{1,i}$, then
	$u$ is the solution of the transmission problem 
	\begin{equation*}\left\{
		\begin{aligned}
			& \mathcal{L} u = 0 & \mathrm{in}\ \mathbb{T}^d \setminus \partial \omega, \\
			& \delta \frac{\partial u}{\partial \nu}\Big|_-- \frac{\partial u}{\partial \nu}\Big|_+= (1-\delta) \phi_i^{\alpha} & \mathrm{on}\ \partial \omega, \\
			& \int_Y u (y)\,dy =0.
		\end{aligned}\right.
	\end{equation*}
We aim to represent the solution by $u= \mathcal{S}\psi -q$, where $q \in \mathbb{R}^m$. In view of the properties of layer potentials (See Section \ref{sec:periodiclayerpotential}), $\psi$ and $q$ must satisfy
\begin{equation*}
	\big( \lambda(\delta) I - \mathcal{K}^* \big) \psi = \phi_i^{\alpha}, \qquad q = \int_Y \mathcal{S}\psi (y)\,dy.
\end{equation*}
The integral system above is solvble for $\psi \in L_0^2(\partial \omega)$, since $|\lambda(\delta)|>1/2$ when $0<\delta <\infty$ and the spectrum of $\mathcal{K}^*$ on $L^2_0(\partial \omega)$ is contained in $(-1/2,1/2)$. This completes the proof.
\end{proof}

Note that $\lambda(\delta)$ converges to $-1/2$ as $\delta\to 0+$, and converges to $1/2$ as $\delta \to \infty$. We define the limit tensors $\chi_0 = (\chi^{\alpha}_{0,i})$ and $\chi_{\infty} = (\chi^{\alpha}_{\infty,i})$, $1\le i\le d$, $1\le \alpha\le m$, by
\begin{equation}\label{cell_limit}
	\left\{
	\begin{aligned}
		& 	\chi^{\alpha}_{0,i}:=\chi^{\alpha}_{1,i}+\mathcal{S}\big( -I/2-\mathcal{ K}^* \big)^{-1} \phi_i^{\alpha}-q^{\alpha}_{0,i} , \\
		& 	\chi_{\infty,i}^{\alpha} :=\chi^{\alpha}_{1,i}+\mathcal{S} \big( I/2- \mathcal{K}^* \big)^{-1} \phi_i^{\alpha}-q^{\alpha}_{\infty,i} ,
	\end{aligned}
	\right.
\end{equation}
where $\phi_i^{\alpha}$ is defined in \eqref{densitydef} and 
\begin{equation}
	q_{0,i}^{\alpha} = \int_Y \mathcal{S} \left[\big( -I/2 -\mathcal{ K}^* \big)^{-1} \phi_i^{\alpha} \right](y)\,dy ,  \qquad 
	q_{\infty,i}^{\alpha} = \int_Y \mathcal{S} \left[\big( I/2-\mathcal{ K}^* \big)^{-1} \phi_i^{\alpha}\right] (y)\,dy.
\end{equation}

Indeed, as we show next, $\chi_0$ and $\chi_{\infty}$ are the limits of $\chi_{\delta}$ as $\delta $ tends to $0$ and to $\infty$, respectively. In fact, they correspond to the cell problem of the soft ($\delta=0$) and hard ($\delta=\infty$) inclusion problems, respectively.

\begin{theorem}\label{thm:cell_limit}
	For $1\leq i\leq d$ and $1\leq \alpha \leq m$, the function $\chi_{0,i}^{\alpha} \in H^1(\mathbb{T}^d)$ is the solution of the  soft problem
	\begin{equation}\label{eq:cell_problem_insulated}
		\left\{
		\begin{aligned}
			& \mathcal{L} (\chi_{0,i}^{\alpha} )= - \mathcal{L} (y_i e^{\alpha}) & \mathrm{in}\  \mathbb{T}^d \setminus \partial \omega, \\
			&\frac{\partial \chi_{0,i}^{\alpha}}{\partial \nu}\Big|_+ = -\frac{\partial (y_ie^{\alpha})}{\partial \nu}\Big|_+  & \mathrm{on}\ \partial \omega, \\
			& \int_Y \chi_{0,i}^{\alpha}(y)\,dy =0.
		\end{aligned}
		\right.
	\end{equation}
	The function $\chi^{\alpha}_{\infty,i} \in H^1(\mathbb{T}^d)$ is the solution of the stiff problem
	\begin{equation}\label{eq:cell_problem_perfect}
		\left\{
		\begin{aligned}
			& \mathcal{L}( \chi_{\infty,i}^{\alpha} )= - \mathcal{L} (y_i e^{\alpha})  & \mathrm{in}\  \mathbb{T}^d \setminus \partial \omega, \\
			&\frac{\partial \chi_{\infty,i}^{\alpha}}{\partial \nu}\Big|_- = -\frac{\partial (y_ie^{\alpha})}{\partial \nu}\Big|_-  & \mathrm{on}\ \partial \omega, \\
			& \int_Y \chi_{\infty,i}^{\alpha}(y)\,dy =0.
		\end{aligned}
		\right.
	\end{equation}
	Moreover, there is a universal constant $C<\infty$ so that
	\begin{equation}\label{convrate_periodic}
		\| \chi_{\delta} - \chi_0 \|_{H^1(Y)} \leq C \delta  \quad \mathrm{and}\quad \| \chi_{\delta}- \chi_{\infty} \|_{H^1(Y)} \leq C \delta^{-1},
	\end{equation}
	and hence also $\|\chi_\delta\|_{H^1}\le C$.
\end{theorem}
\begin{proof}
	In view of Section \ref{sec:periodiclayerpotential}, especially the jump relation, it is clear that $\chi^{\alpha}_{0,i}$ and $\chi^{\alpha}_{\infty,i}$ is the solution of problem \eqref{eq:cell_problem_insulated} and \eqref{eq:cell_problem_perfect}, respectively. We now prove \eqref{convrate_periodic}. 
	
	Since  $\lambda(\delta) I - \mathcal{K}^* \rightarrow \mp \frac{1}{2}I- \mathcal{K}^*$ as $\delta $ tends to $0$ or $\infty$ and $\mp \frac{1}{2}I - \mathcal{K}^*$ are bijections on $L^2_0(\partial \omega)$, by the basic theory of functional analysis, we obtain that $\| (  \lambda(\delta)I- \mathcal{K}^* )^{-1} \|$ is uniformly bounded in $0<\delta <\infty$.
	Therefore,
	\begin{equation*}
		\begin{aligned}
			\big\| \chi_{\delta,i}^{\alpha} - \chi_{0,i}^{\alpha} \big\|_{H^1(Y)} \leq\, &\big\|  \mathcal{S}\big[( \lambda(\delta) I- \mathcal{K}^* )^{-1} -( -I/2- \mathcal{K}^* )^{-1}  \big] \phi_i^{\alpha} \big\|_{H^1(Y)} + \big|q_{\delta,i}^{\alpha} - q_{0,i}^{\alpha}\big| \\
			\leq\,&2 \big\|  \mathcal{S} \big[( \lambda(\delta) I- \mathcal{K}^* )^{-1} -( -I/2- \mathcal{K}^* )^{-1}  \big] \phi_i^{\alpha} \big\|_{H^1(Y)} \\
			\leq\,& C\big(\lambda(\delta)+1/2\big) \big\| \big(  \lambda(\delta )I- \mathcal{K}^* \big)^{-1} \big\|  \cdot
			\big\| (I/2+ \mathcal{K}^* )^{-1} \big\| \\
			\leq\,& C\delta.
		\end{aligned}
	\end{equation*}
	This proves the convergence rate when $\delta \rightarrow 0$. The proof for rate of $\delta \rightarrow\infty$ follows the same lines and is hence omitted. The uniform estimates for $\|\chi_\delta\|_{H^1}$ follows also because $\|\chi_1\|_{H^1}$ is easily bounded, and by \eqref{cell_limit} the bounds for $\chi_0,\chi_\infty$ also hold.
\end{proof}

\begin{remark}
In Appendix \ref{appen: weak layer potential},
we define a 
'weak' form of single layer potentials via Lax-Milgram theorem, which maps from $H^{-1/2}$ to $H^{1/2}$, and the corresponding NP operator without smoothness condition \eqref{smoothc}. Using the 'weak' form of layer potentials, all above results can be recovered by the same analysis.
\end{remark}

%%%%%%%%
\subsection{Uniform bound for the homogenized tensor}

Recall that, for each fixed $\delta > 0$, the homogenized tensor for \eqref{model} is $\hat{A}_\delta$ defined by \eqref{eq:Abardel}. We show that $\widehat{A}_\delta$ enjoys uniform (in dependent of $\delta$) ellipticity. 

\begin{lemma}\label{lem_esti_two}
	For any $0<\delta<\infty$, $1\leq i ,j \leq d$, $1\leq \alpha ,\beta \leq m$, we have
	\begin{equation}\label{estimatetwo}
		\Big\| \Lambda_{\delta}  \Big( a_{ij}^{\alpha \beta} +a_{ik}^{\alpha \gamma}  \frac{\partial }{\partial y_k} \chi_{\delta,j}^{\gamma \beta} \Big)  \Big\|_{L^2(Y )} \leq C .
	\end{equation}
\end{lemma}
\begin{proof}
	The aim is to bound the $L^2(Y)$ norm of 
	$$
	\sum_{k=1}^d\sum_{\gamma=1}^m \Lambda_\delta(y) a_{ik}^{\alpha\gamma}(y) \frac{\partial}{\partial y_k} \left(\chi^{\gamma\beta}_{\delta,j}(y) + y_j e^\beta\right),
	$$
	for each fixed $i,j,\alpha$ and $\beta$. %Also, we may assume $\delta\ne 1$ as the estimate is well known otherwise.
By Theorem \ref{thm:cell_limit}, we have
	\begin{equation}\label{Tomega}
		\Big\|\frac{\partial}{\partial y_i}(\chi_{\delta} -\chi_0 ) \Big\|_{L^2(Y )} \leq C\delta \quad\mathrm{and}\quad \Big\| \frac{\partial}{\partial y_i}(\chi_{\delta} -\chi_{\infty} ) \Big\|_{L^2(Y )} \leq C\delta^{-1}.
	\end{equation}
	For $\delta \in (0,1]$, the desired estimates follow from the boundedness of $a$, $\|\chi_0\|_{H^1}$ and $\delta\le 1$. For $\delta \in (1,\infty)$, after replacing $\chi_\delta$ by $\chi_\infty$ and applying the above inequality, it remains to control  
	$$
	\|\sum_{k=1}^d\sum_{\gamma=1}^m \Lambda_\delta(y) a_{ik}^{\alpha\gamma}(y) \frac{\partial}{\partial y_k} \left(\chi^{\gamma\beta}_{\infty,j}(y) + y_j e^\beta\right)\|_{L^2(\omega)},
	$$
	because $\Lambda_\delta =1$ outside $\omega$ and $\|\nabla \chi_\infty\|_{L^2}$ is bounded. Now in view of the stiff cell problem \eqref{eq:cell_problem_perfect} (for the part in $\omega$), we must have $\chi_{\infty,j}^\beta + y_je^\beta = c$ for some constant vector $c\in \R^d$. Hence, the above quantity vanishes. This completes the proof.
\end{proof}

\begin{lemma}\label{continuous_homogenized_coefficient}
	The homogenized tensor $\widehat{A}_{\delta}$ is continuous as a function of $\delta \in (0,\infty)$, we denote $\widehat{A}_0$ by $\lim_{\delta \rightarrow 0 } \widehat{A}_{\delta}$ and $\widehat{A}_{\infty}$ by $\lim_{\delta \rightarrow \infty} \widehat{A}_{\delta}$, then
	\begin{equation}\label{homotensor0}
		\widehat{a}_{0,ij}^{\alpha \beta}= \int_{Y\setminus \overline{\omega}}  \Big[ a_{ij}^{\alpha \beta}+ a_{ik}^{\alpha \gamma}   \frac{\partial}{\partial y_k} \chi_{0,j}^{\gamma \beta}\Big]\,dy,
	\end{equation}
	and
	\begin{equation}\label{homotensorinfy}
		\widehat{a}_{\infty,ij}^{\alpha \beta}= \int_{\omega} a_{ik}^{\alpha \gamma}\frac{\partial w_j^{\gamma \beta} }{\partial y_k} \,dy+ \int_{Y\setminus \overline{\omega}} \Big[ a_{ij}^{\alpha \beta} + a_{ik}^{\alpha \gamma}   \frac{\partial}{\partial y_k} \chi_{\infty,j}^{\gamma \beta}\Big]\,dy,
	\end{equation}
	where for each fixed $j=1$ and $\alpha$, the function $w^{\alpha}_j \in H^1(\omega;\R^m)$ is the unique (up to an additive constant) solution of 
	\begin{equation}\label{secondorderexpansion}
		\left\{
		\begin{aligned}
			& \mathcal{L}(w_j^{\alpha})=0 & \mathrm{in}\ \omega,\\%\mathbb{T}^d \setminus \partial \omega, \\
			&\frac{\partial w_j^{\alpha}}{\partial \nu}\Big|_-= \frac{\partial }{\partial \nu} ( \chi_{\infty,j}^{\alpha}+y_je^{\alpha})\Big|_+ &\mathrm{on}\ \partial \omega.
		\end{aligned}\right.
	\end{equation}
\end{lemma}

It is worth to mention that, the tensor $\widehat A_0$ (respectively, $\widehat A_\infty$) is the homogenized tensor for the soft inclusion problem (respectively, stiff inclusion problem); see Remark \ref{rem:extreme}.

\begin{proof}
	By the definition \eqref{eq:Abardel} of $\widehat{A}_{\delta}$ and the single-layer potential representations of the solutions to cell problem, it is clear that $\widehat{A}_{\delta}$ is continuous as a function of $\delta \in (0,\infty)$, and \eqref{homotensor0} follows directly from the convergence $\| \chi_{\delta} - \chi_0\|_{H^1(Y)}\leq C\delta$. It remains to verify
	\eqref{homotensorinfy}.
	
	First we note that the problem for $w_j^\alpha$ is well posed. We need to check that the Neumann data integrates to zero over $\partial \omega$. By the definition \eqref{cell_limit} and the jump relation of layer potentials, 
\begin{equation*}
\begin{aligned}
\frac{\partial }{\partial \nu} ( \chi_{\infty,j}^{\alpha}+y_je^{\alpha})\Big|_+ &= \frac{\partial }{\partial \nu} ( \chi_{\infty,j}^{\alpha}+y_je^{\alpha})\Big|_-  + \big( I/2- \mathcal{K}^* \big)^{-1} \phi_j^{\alpha}\\
&= \big( I/2- \mathcal{K}^* \big)^{-1} \phi_j^{\alpha}.
\end{aligned}
\end{equation*}
The solvability of $w_j^\delta$ is verified. Moreover, we check that
\begin{equation*}
\begin{aligned}
&\mathcal{L}(\chi_{\delta,j}^\alpha-\chi_{\infty,j}^\alpha-\frac1{\delta-1} w_j^\alpha) = 0 \;\text{ in } \omega,\\
&\frac{\partial}{\partial \nu}(\chi_{\delta,j}^\alpha- \chi_{\infty,j}^\alpha - \frac{1}{\delta-1} w_j^\alpha)\Big\rvert_- = \left[I + (-I/2+\mathcal{K}^*)(\lambda(\delta)I - \mathcal{K}^*)^{-1} - \frac1\delta(I/2-\mathcal{K}^*)^{-1}\right]\phi_j^\alpha.
\end{aligned}
\end{equation*}
The operators $(I/2-\mathcal{K}^*)$ and its inverse act on $L^2_0(\partial \omega)$. For $\delta>1$, the inverse operator can be written as Neumann series, and the second line becomes
\begin{equation*}
\frac{\partial}{\partial \nu}(\chi_{\delta,j}^\alpha- \chi_{\infty,j}^\alpha - \frac{1}{\delta-1} w_j^\alpha)\Big\rvert_- = \sum_{\ell=2}^\infty \frac{(-1)^\ell}{(\delta-1)^\ell}(I/2-\mathcal{K}^*)^{-\ell}. 
\end{equation*}
	By basic estimates of Neumann problem of elliptic system we get
	\begin{equation}\label{second_order}
		\|\nabla(\chi_{\delta,j}^{\alpha} -\chi^{\alpha}_{\infty,j}  - \frac{1}{\delta-1}w^{\alpha}_j)\|_{L^2(\omega)} \leq C\delta^{-2},
	\end{equation}
	for all $\delta>1$ sufficiently large, and $C<\infty$ depends only on the norm of $\|(I/2-\mathcal{K}^*)^{-1}\|$. Using the fact that $\chi_{\infty,j}^\alpha+y_je^\alpha$ is a constant vector in $\omega$, we rewrite
	\begin{equation*}
		\begin{aligned}
			\widehat{a}_{\delta,ij}^{\alpha \beta} - \widehat{a}_{\infty,ij}^{\alpha \beta}&= \delta \int_{\omega} a_{ik}^{\alpha \gamma} \frac{\partial}{\partial y_k}\big( \chi_{\delta,j}^{\gamma \beta} -\chi^{\gamma \beta}_{\infty,j} - \frac{1}{\delta-1} w_j^\beta\big) \,dy+ \int_{Y \setminus \overline{\omega}} a_{ik}^{\alpha \gamma}   \frac{\partial}{\partial y_k} \big(\chi_{\delta,j}^{\gamma \beta}-\chi_{\infty,j}^{\gamma \beta}\Big) \,dy \\
			&\qquad\qquad + \frac{1}{\delta-1} \int_{\omega} a_{ik}^{\alpha \gamma}\frac{\partial u_j^{\gamma \beta} }{\partial y_k} \,dy.
		\end{aligned}
	\end{equation*}
Send $\delta\to \infty$, due to the bounds \eqref{second_order} and \eqref{estimatetwo}, we obtain the desired result. In fact, we get $|\widehat{A}_\delta-\widehat{A}_\infty|\le C\delta^{-1}$ for sufficiently large $\delta$.
\end{proof}

As a corollary of Lemma \ref{continuous_homogenized_coefficient}, we show that $\widehat{A}_{\delta}$ is uniform elliptic in $0<\delta <\infty$.

\begin{theorem} 
	\label{uniform elliptic homogenized tensor}
	Assume $A$ satisfies \eqref{ellipticc}\eqref{periodicc}, then $\widehat{A}_{\delta}$ is uniformly elliptic in $\delta$, i.e. there exists a universal positive number $\mu_1$, so that for any $\delta \in (0,\infty)$, we have $|\widehat{A}_\delta| \le \mu_1^{-1}$ and
	\begin{equation}
	\label{eq:Adelgmu1}
		\widehat{a}_{\delta,ij}^{\alpha \beta} \xi_i^{\alpha}\xi_j^{\beta} \geq \mu_1 |\xi|^2\quad \mathrm{for}\ \mathrm{any} \ \xi \in \mathbb{R}^{dm} .
	\end{equation}
\end{theorem}
\begin{proof}
% 	To show $\widehat{A}_{\delta}$ is symmetric, note that $\mathcal{L}_{1,\delta}(y_j e^{\beta} +\chi_{\delta,j}^{\beta}) = 0$ in $Y$, so for any $u \in H^1(\mathbb{T}^d)$, 
% 	\begin{equation}
% 		( J^{Y\setminus \overline{\omega}} + \delta   J^{\omega}) (u, y_je^{\beta} + \chi_{\delta,j}^{\beta} ) = 0 .
% 	\end{equation}
% So $\widehat{a}_{\delta,ij}^{\alpha \beta}  = ( J^{Y\setminus \overline{\omega} } + \delta   J^{\omega})  ( y_i e^{\alpha } + u, y_j e^{\beta} + \chi_{\delta,j}^{\beta} ) $, 
% 	let $u= \chi^{\alpha}_{\delta,i}$, we see that  $\widehat{A}_{\delta}$ is symmetric.
%WJ: Again, I do not see the necessity to have symmetry. For the homogenized tensor, in the system setting, it worths to investigate whether we can always write the system in non-div form, etc.
	
We have established a uniform in $\delta$ bound of $|\widehat{A}_\delta|$, so we concentrate here on proving \eqref{eq:Adelgmu1}. Given $\xi =(\xi_i^{\alpha})\in \mathbb{R}^{dm}$, let $w_{\delta}=\xi^{\alpha}_i (y_i e^{\alpha} +\chi_{\delta,i}^{\alpha})$ be the combination of the correctors $\chi_\delta$'s. We have
	\begin{equation*}
		\begin{aligned}
			\widehat{a}_{\delta,ij}^{\alpha \beta} \xi_i^{\alpha}\xi_j^{\beta}&= ( J^{Y\setminus \overline{\omega} } + \delta   J^{\omega}) (w_{\delta})  \\
			&\geq \min\{\delta,1\} \mu\int_Y |\nabla w_\delta |^2 \,dy\\
			& \geq \min\{\delta,1\} \mu\left| \int_Y \nabla w_{\delta} \right|^2\,dy \\
			&=   \min\{\delta,1\} \mu|\xi|^2.
		\end{aligned}
	\end{equation*}
	The last identity holds because the integral of $\nabla \chi_{\delta,i}^{\alpha}$ on $\mathbb{T}^d$ vanishes. Therefore, for $\delta\ge 1$, the lower ellipticity constant for $\widehat{A}_\delta$ is uniformly bounded by $\mu$ from below. For $\delta\in (0,1)$, in view of the continuity of $\widehat{A}_{\delta}$ with respect to $\delta$, it sufficies to prove that $\widehat{A}_0$ is elliptic, but this follows by the same reasoning as above.
	% By \eqref{homotensor0}, we have
	% \begin{equation*}
	% 	\widehat{a}_{0,ij}^{\alpha \beta} \xi_i^{\alpha}\xi_j^{\beta} = J^{Y\setminus \overline{\omega}} (w_0),
	% \end{equation*} 
	% suppose that for a $\xi \in \mathbb{R}^{dm}$ we have $\widehat{a}_{0,ij}^{\alpha \beta} \xi_i^{\alpha}\xi_j^{\beta} =0$, then $w_0$ is a constant vector in $Y \setminus \overline{\omega}$, this implies that $\xi_i^{\alpha} y_i e^{\alpha}= w_0 -  \xi^{\alpha}_i \chi^{\alpha}_{0,i}$ is periodic in $Y \setminus \overline{\omega}$, which is impossible. So $\widehat{a}_{0,ij}^{\alpha \beta} \xi_i^{\alpha}\xi_j^{\beta} >0$ for all $\xi \in \mathbb{R}^{dm}$, by linearity, $\widehat{A}_0$ is elliptic.
\end{proof}
%%%%%%%%
\subsection{Uniform bound of the flux corrector}

Given a tensor $F=(F^{\alpha \beta}_{ij})$ on $Y$, for $1\leq i,j \leq d$, $1\leq \alpha,\beta \leq m$, such that
\begin{equation}\label{requireflux}
	\frac{\partial }{\partial y_i} F^{\alpha \beta}_{ij} =0   \quad \mathrm{and} \quad \int_Y F^{\alpha \beta}_{ij}(y)\,dy =0.
\end{equation}
It is proved, for example see \cite{shen_periodic_2018}, there exists a 
tensor $\Psi = (\Psi^{\alpha \beta}_{kij})$ on $\mathbb{T}^d$, for $1\leq k, i,j \leq d$ and  $1\leq \alpha,\beta \leq m$, such that
\begin{equation}
\label{eq:FtoPsi}
	\frac{\partial }{\partial y_k} \Psi^{\alpha \beta}_{kij} = F_{ij}^{\alpha \beta} \quad \mathrm{and} \quad  \Psi_{kij}^{\alpha \beta} =- \Psi_{ikj}^{\alpha \beta},
\end{equation}
and there exists a constant $C>0$ such that $
	\| \Psi \|_{L^2(Y)} \leq C  \|F\|_{L^2(Y)}
$.
$\Psi$ is called the \textit{flux corrector} associated to $F$. 

In this paper, we need to use the flux corrector $\Psi_{\delta} = (\Psi_{\delta,kij}^{\alpha \beta})$ associated to $F_{\delta} = (F_{\delta,ij}^{\alpha \beta})$, where
\begin{equation}\label{def:Fdelta}
	F_{\delta,ij}^{\alpha \beta}(y)=\Lambda_{\delta} (y) \Big( a_{ij}^{\alpha \beta}(y) +a_{ik}^{\alpha \gamma}(y)  \frac{\partial }{\partial y_k} \chi_{\delta,j}^{\gamma \beta}(y)\Big) - \widehat{a}^{\alpha \beta}_{\delta,ij} ,\qquad 0<\delta<\infty.
\end{equation}
It is easy to check condition \eqref{requireflux} for $F_{\delta}$, so $\Psi_{\delta}$ exists and $\| \Psi_{\delta} \|_{L^2(Y)}\leq C \| F_{\delta} \|_{L^2(Y)}$. 

\begin{theorem}\label{thm:boundflux}
	Let $F_{\delta}$ be defined by \eqref{def:Fdelta}, and let $\Psi$ be the skew-symmetric flux-tensor related to $F_\delta$ by \eqref{eq:FtoPsi}. Then
	\begin{equation*}
		\| F_{\delta} \|_{L^2(Y)} + \| \Psi_{\delta} \|_{L^2(Y)} \leq C.
	\end{equation*}
\end{theorem}

This theorem follows directly from the uniform in $\delta$ bounds \eqref{estimatetwo}. The uniform estimates for $F_\delta$ and $\Psi_\delta$ above will be the key to obtain uniform in $\delta$ convergence rate.

%%%%%%%
%%%%%%%
\section{Proof of Theorem \ref{thm:suboptimal convergence rate}} \label{proofof theorem1}

In this section, we prove the uniform (in $\eps$ and $\delta$) convergence rate in $H^1(\Omega)$. %For those result, we do not need the well-separation condition in {\bf G}(e). More precisely, our proof works for both of the geometric set-ups in Assumption {\bf G}. 
%%%%%%%
\subsection{Uniform energy estimates}

We first derive uniform $H^1$ energy estimates for the solution to the heterogeneous equation in the high contrast setting. With the extension operators of the previous section, the estimates become routine.

\begin{theorem}\label{thm:unifenergy} Assume \eqref{ellipticc}-\eqref{periodicc} and $(\mathbf{H2})$, for $\delta \in (0,1]$ we have
\begin{equation}
\label{eq:energy}
\|u_\eps\|_{H^1(\Omega)}\le C\left(\|f\|_{L^2(\Omega)} + \eps\delta^{-1}\|f\|_{L^2(D_\eps)} + \|g\|_{H^{\frac12}(\partial \Omega)}\right).
\end{equation}
For $\delta > 1$, the above still holds with the $\eps\delta^{-1} \|f\|_{L^2(D_\eps)}$ term removed from the right hand side. In particular, for $\delta \in (0,1]$, if we replace $f$ on the right hand side of \eqref{model} by $f_{\eps,\delta}$ defined by \eqref{eq:modf} we may ignore the second term on the right hand side above. 
\end{theorem}

\begin{proof} Note that for fixed $\eps,\delta$, the uniquely solvability of $u_\eps \in H^1(\Omega)$ is routine. Our goal is to establish the above uniform energy estimates. 

\emph{Step 1: Control of the energy inside the inclusions}. Let $\widetilde u_\eps$ be the extension $\mathcal{P}_\eps (u_\eps\rvert_{\Omega_\eps})$ given by Lemma \ref{lem:extension}. Then $u_\eps -\widetilde u_\eps \in H^1_0(\Omega)$, from the equation \eqref{model} and the fact that $u_\eps = \widetilde u_\eps$ in $\Omega_\eps$, we get
\begin{equation*}
\delta \int_{D_\eps} A(x/\eps)\nabla u_\eps \cdot (\nabla u_\eps-\nabla \widetilde u_\eps) = \int_{D_\eps} f\cdot(u_\eps-\widetilde u_\eps).
\end{equation*}
The ellipticity of $A$ and Poincar\'{e} inequality $\|u_\eps-\widetilde u_\eps\|_{L^2(D_\eps)} \leq C\varepsilon\|\nabla u_\eps-\nabla \widetilde u_\eps\|_{L^2(D_\eps)}$ yields
\begin{equation*}
 \|\nabla u_\eps\|^2_{L^2(D_\eps)} \le C\|\nabla u_\eps\|_{L^2(D_\eps)}\|\nabla \widetilde u_\eps\|_{L^2(D_\eps)} + C\varepsilon\delta^{-1}\|f\|_{L^2(D_\eps)}\|\nabla u_\eps-\nabla \widetilde u_\eps\|_{L^2(D_\eps)},
\end{equation*}
this together with Young's inequality and $\|\nabla \widetilde u_\eps\|_{L^2(D_{\varepsilon})} \le C\|\nabla u_\eps\|_{L^2(\Omega_\eps)}$ leads to
\begin{equation}
\label{eq:energy-in}
\|\nabla u_\eps\|_{L^2(D_\eps)} \le C\Big(\varepsilon \delta^{-1} \|f\|_{L^2(D_\eps)} + \|\nabla u_\eps\|_{L^2(\Omega_\eps)}\Big)
\end{equation}
with some bounding constant $C$ that is independent of $\eps,\delta$ and $f$. This estimate will be useful for $\delta \in (0,1]$. 

For $\delta > 1$, we use Lemma \ref{lem:energyinlarged} below (which is from Lemma 3.1 of \cite{shen_large-scale_2021}) to get
\begin{equation}
\label{eq:energy-in-ld}
\delta \|\nabla u_\eps\|_{L^2(D_\eps)} \le C(\|f\|_{L^2(D_\eps)} + \|\nabla u_\eps\|_{L^2(\Omega_\eps)}).
\end{equation}

\emph{Step 2: Control of the total energy}. Let $G \in H^1(\Omega)$ be a function such that $G=g$ on $\partial \Omega$ and $\|G\|_{H^1(\Omega)}\le C\|g\|_{H^{1/2}(\partial \Omega)}$. Test $u_\eps -G \in H^1_0(\partial \Omega)$ in \eqref{model}, the energy estimate yields
\begin{equation*}
\|\nabla u_\eps \|_{L^2(\Omega_\eps)}^2 \le C\|f\|_{L^2(\Omega)}\|\nabla u_\eps -\nabla G\|_{L^2(\Omega)} + C\big( \delta \|\nabla u_\eps\|_{L^2(D_\eps)} +  \|\nabla u_\eps\|_{L^2(\Omega_\eps)} \big) \|\nabla G\|_{L^2(\Omega)}.
\end{equation*}

First consider the case of $\delta \le 1$.  Using the fact that $\delta$ is not large, we get
\begin{equation*}
\|\nabla u_\eps\|^2_{L^2(\Omega_\eps)} \le  C\|f\|_{L^2(\Omega)}\left(\|\nabla u_\eps\|_{L^2(\Omega)} +\|\nabla G\|_{L^2(\Omega)}\right) +  C\|\nabla u_\eps\|_{L^2(\Omega)} \|\nabla G\|_{L^2(\Omega)} .
\end{equation*}
Using \eqref{eq:energy-in} we then get
\begin{equation*}
\|\nabla u_\eps\|_{L^2(\Omega_\eps)} \le C\left(\|f\|_{L^2(\Omega)} + \eps\delta^{-1}\|f\|_{L^2(D_\eps)} + \|\nabla G\|_{L^2(\Omega)}\right).
\end{equation*}
The above combined with \eqref{eq:energy-in} again yields the desired estimate \eqref{eq:energy}. 

\smallskip

For the case of $\delta > 1$, we only need to use the bound \eqref{eq:energy-in-ld} in the energy estimate and the argument above, the desired inequality for $\delta > 1$ then follows.
\end{proof}

%%%%%%%
\subsection{Convergence rates in homogenization}

We consider the discrepancy function $w_{\varepsilon,\delta} $ defined by
\begin{equation}
\label{eq:discrepancy}
	w_{\varepsilon,\delta} := u_{\varepsilon,\delta} - \widehat{u}_{\delta} -  \varepsilon \chi^{\alpha}_{\delta,i}(x/\varepsilon) S_{\varepsilon} \Big(\eta_{\varepsilon} \frac{\partial \widehat{u}_{\delta}^{\alpha}}{\partial x_i} \Big) ,
\end{equation}
	where $S_{\varepsilon}$ is the standard smoothing operator defined in section \ref{sec:tools}, and $\eta_{\varepsilon}$ is a smooth cut-off function supported away from $\partial \Omega$. More precisely, $\eta_\eps$ is chosen so that
\begin{equation}\label{cut-off_def}
	\left\{
	\begin{aligned}
		& 0\leq \eta_{\varepsilon} (x)\leq 1 \quad \mathrm{for}\ x\in \Omega, \\
		& \mathrm{supp}\,\eta_{\varepsilon}\subset \Omega \setminus \mathcal{O}_{3\varepsilon}, \\
		&\eta_{\varepsilon} =1 \quad \mathrm{in}\  \Omega \setminus \mathcal{O}_{4\varepsilon}, \\
		& \|\nabla \eta_{\varepsilon}\|_{\infty} \leq C\varepsilon^{-1}.
	\end{aligned}\right.
\end{equation}

The following lemma is the start point of our proof.

\begin{lemma}\label{equation_r_epsi}
	Assume $A$ satisfies \eqref{ellipticc}-\eqref{periodicc}, then the discrepancy function $w_{\eps,\delta}$ defined by \eqref{eq:discrepancy} satisfies
	\begin{equation}\label{eq:equation_discrepancy}
		\big\{\mathcal{L}_{\varepsilon,\delta} (w_{\varepsilon,\delta}) \big\}^{\alpha}=  - \frac{\partial}{\partial x_i} \big( \Lambda_{\delta}(x/\varepsilon) P^{\alpha}_{\varepsilon,\delta,i} \big)- \frac{\partial}{\partial x_i}Q_{\varepsilon,\delta,i}^{\alpha} ,
	\end{equation}
	where $P_{\varepsilon,\delta} = ( P_{\varepsilon,\delta,i}^{\alpha} )$ is defined by
	\begin{equation}
	\label{eq:Ped}
		P_{\varepsilon,\delta,i}^{\alpha} =   a_{ij}^{\alpha \beta} \Big(\frac{x}{\varepsilon}\Big)\Big[S_{\varepsilon} \Big( \eta_{\varepsilon} \frac{\partial \widehat{u}_{\delta}^{\beta}}{\partial x_j} \Big)   - \frac{\partial \widehat{u}_{\delta}^{\beta} }{\partial x_j}   - \varepsilon   \chi_{\delta,k}^{\gamma \beta}  \Big(\frac{x}{\varepsilon}\Big)\frac{\partial }{\partial x_j}S_{\varepsilon} \Big(\eta_{\varepsilon} \frac{\partial \widehat{u}_{\delta}^{\gamma} }{\partial x_k} \Big)  \Big] 
	\end{equation}
	and $Q_{\varepsilon,\delta} = (Q_{\varepsilon,\delta,i}^{\alpha} )$ is defined by
	\begin{equation}
	\label{eq:Qed}
		Q_{\varepsilon,\delta,i}^{\alpha} =\varepsilon  \Psi_{\delta,kij}^{\alpha \beta} \Big(\frac{x}{\varepsilon} \Big) \frac{\partial}{\partial x_k}  S_{\varepsilon} \Big( \eta_{\varepsilon} \frac{\partial \widehat{u}_{\delta}^{\beta}}{\partial x_j} \Big) -  \widehat{a}_{\delta,ij}^{\alpha \beta} \Big[ S_{\varepsilon} \Big(\eta_{\varepsilon} \frac{\partial \widehat{u}_{\delta}^{\beta} }{\partial x_j} \Big)  -\frac{\partial \widehat{u}^{\beta}_{\delta} }{\partial x_j} \Big]  .
	\end{equation}
	Further assume $(\mathbf{H2})$, then they satisfy the following estimates
	\begin{equation}\label{01first}
		\| \nabla w_{\varepsilon,\delta} \|_{L^2(\Omega_{\varepsilon})}  \leq C\big(\| P_{\varepsilon,\delta} \|_{L^2(\Omega)} + \|Q_{\varepsilon,\delta} \|_{L^2(\Omega) }  \big) ,\qquad \mathrm{for}\ 0<\delta<1 ,
	\end{equation}
	and
	\begin{equation}\label{1inftyfirst}
		\| \nabla w_{\varepsilon,\delta} \|_{L^2(D_{\varepsilon})}  \leq C\big(\| P_{\varepsilon,\delta} \|_{L^2(\Omega)} + \delta^{-\frac12}\|Q_{\varepsilon,\delta} \|_{L^2(\Omega) }  \big)  ,\qquad \mathrm{for} \ 1<\delta<\infty.
	\end{equation}
\end{lemma}
\begin{proof} \emph{Step 1: Verification of \eqref{eq:equation_discrepancy} in $H^{-1}(\Omega;\R^m)$}. 
	Let $\psi \in H_0^1(\Omega;\R^m)$, using
	\begin{equation*}
		\int_{\Omega} \Lambda_{\delta}\Big(\frac{x}{\varepsilon} \Big)a_{ij}^{\alpha \beta}\Big(\frac{x}{\varepsilon} \Big) \frac{\partial u^{\beta}_{\varepsilon,\delta} }{\partial x_j} \frac{\partial \psi^{\alpha}}{\partial x_i} \,dx = \int_{\Omega} \widehat{a}_{\delta,ij}^{\alpha \beta}  \frac{\partial \widehat{u}^{\beta}_{\delta} }{\partial x_j} \frac{\partial \psi^{\alpha}}{\partial x_i} \,dx  ,
	\end{equation*}
	and the definition of $w_{\eps,\delta}$, by direct computation we obtain
	\begin{equation}\label{eq:weak_dis}
		\begin{aligned}
			\int_{\Omega} & \Lambda_{\delta}\Big(\frac{x}{\varepsilon} \Big) a_{ij}^{\alpha \beta}\Big(\frac{x}{\varepsilon} \Big) \frac{\partial w^{\beta}_{\varepsilon,\delta} }{\partial x_j} \frac{\partial \psi^{\alpha}}{\partial x_i} \,dx \\%=  \int_{\Omega} \Big( \widehat{a}_{\delta,ij}^{\alpha \beta}  - \Lambda_{\delta}\Big(\frac{x}{\varepsilon} \Big) a_{ij}^{\alpha \beta}\Big(\frac{x}{\varepsilon} \Big) \Big) \frac{\partial \widehat{u}^{\beta}_{\delta} }{\partial x_j} \frac{\partial \psi^{\alpha}}{\partial x_i} \,dx  \\
			%& - \int_{\Omega} \Lambda_{\delta}\Big(\frac{x}{\varepsilon} \Big)a_{ij}^{\alpha \beta} \Big(\frac{x}{\varepsilon} \Big) \frac{\partial \chi_{\delta,k}^{\gamma \beta}}{\partial x_j} \Big(\frac{x}{\varepsilon}\Big)S_{\varepsilon} \Big(\eta_{\varepsilon} \frac{\partial \widehat{u}_{\delta}^{\gamma} }{\partial x_k} \Big)  \frac{\partial \psi^{\alpha}}{\partial x_i} \,dx  \\
			%&-\varepsilon \int_{\Omega} \Lambda_{\delta}\Big(\frac{x}{\varepsilon} \Big)a_{ij}^{\alpha \beta} \Big(\frac{x}{\varepsilon} \Big) \chi_{\delta,k}^{\gamma \beta}  \Big(\frac{x}{\varepsilon}\Big)\frac{\partial }{\partial x_j}S_{\varepsilon} \Big(\eta_{\varepsilon} \frac{\partial \widehat{u}_{\delta}^{\gamma} }{\partial x_k} \Big)  \frac{\partial \psi^{\alpha}}{\partial x_i} \,dx  \\
			&=  \int_{\Omega} \Big(  \widehat{a}_{\delta,ij}^{\alpha \beta}  - \Lambda_{\delta}\Big(\frac{x}{\varepsilon} \Big) a_{ij}^{\alpha \beta}\Big(\frac{x}{\varepsilon} \Big) \Big)
			\Big[	\frac{\partial \widehat{u}^{\beta}_{\delta} }{\partial x_j}-S_{\varepsilon} \Big(\eta_{\varepsilon} \frac{\partial \widehat{u}_{\delta}^{\beta} }{\partial x_j} \Big)  \Big] \frac{\partial \psi^{\alpha}}{\partial x_i} \,dx \\
			& \underbrace{ - \int_{\Omega} \Big[  \Lambda_{\delta}\Big(\frac{x}{\varepsilon} \Big) a_{ik}^{\alpha \gamma} \Big(\frac{x}{\varepsilon} \Big) + \Lambda_{\delta}\Big(\frac{x}{\varepsilon} \Big)a_{ij}^{\alpha \beta} \Big(\frac{x}{\varepsilon} \Big)\frac{\partial \chi_{\delta,k}^{\gamma \beta}}{\partial x_j} \Big(\frac{x}{\varepsilon}\Big)- \widehat{a}_{\delta,ik}^{\alpha \gamma}\Big] S_{\varepsilon} \Big(\eta_{\varepsilon} \frac{\partial \widehat{u}_{\delta}^{\gamma} }{\partial x_k} \Big)  \frac{\partial \psi^{\alpha}}{\partial x_i} \,dx}_{I}  \\
			&  -\varepsilon \int_{\Omega} \Lambda_{\delta}\Big(\frac{x}{\varepsilon} \Big)a_{ij}^{\alpha \beta} \Big(\frac{x}{\varepsilon} \Big) \chi_{\delta,k}^{\gamma \beta}  \Big(\frac{x}{\varepsilon}\Big)\frac{\partial }{\partial x_j}S_{\varepsilon} \Big(\eta_{\varepsilon} \frac{\partial \widehat{u}_{\delta}^{\gamma} }{\partial x_k} \Big)  \frac{\partial \psi^{\alpha}}{\partial x_i} \,dx .
		\end{aligned}
	\end{equation}
	For the terms grouped as $I$, we recall the definition of flux corrector $F$ in \eqref{def:Fdelta} and use \eqref{eq:FtoPsi} to compute
	\begin{equation}\label{eq:I2}
		\begin{aligned}
			I & = - \int_{\Omega} F^{\alpha \gamma}_{\delta,ik}\Big(\frac{x}{\varepsilon}\Big) S_{\varepsilon} \Big(\eta_{\varepsilon} \frac{\partial \widehat{u}_{\delta}^{\gamma} }{\partial x_k} \Big)  \frac{\partial \psi^{\alpha}}{\partial x_i} \,dx  \\
			& = - \varepsilon \int_{\Omega} \frac{\partial}{\partial x_k} \Big( \Psi_{\delta,kij}^{\alpha \beta}\Big(\frac{x}{\varepsilon} \Big) \Big) S_{\varepsilon} \Big( \eta_{\varepsilon} \frac{\partial \widehat{u}_{\delta}^{\beta}}{\partial x_j} \Big) \frac{\partial \psi^{\alpha}}{\partial x_i} \,dx\\
			& = \varepsilon \int_{\Omega} \Psi_{\delta,kij}^{\alpha \beta} \Big(\frac{x}{\varepsilon} \Big) \frac{\partial}{\partial x_k}  S_{\varepsilon} \Big( \eta_{\varepsilon} \frac{\partial \widehat{u}_{\delta}^{\beta}}{\partial x_j} \Big)\frac{\partial \psi^{\alpha}}{\partial x_i}\,dx .
		\end{aligned}
	\end{equation}
	Then \eqref{eq:equation_discrepancy} immediately follows from the computations above and the definitions of $P_{\eps,\delta},Q_{\eps,\delta}$ in \eqref{eq:Ped} and \eqref{eq:Qed}.
	
	\medskip

	\emph{Step 2: Energy estimates for $w_{\eps,\delta}$}. For any $\psi \in H_0^1(\Omega)$, we have
	\begin{equation}\label{weakdelta01}
		\int_{\Omega} \Lambda_{\delta}\Big(\frac{x}{\varepsilon} \Big) a_{ij}^{\alpha \beta}\Big(\frac{x}{\varepsilon} \Big) \frac{\partial w^{\beta}_{\varepsilon,\delta} }{\partial x_j} \frac{\partial \psi^{\alpha}}{\partial x_i} \,dx =  \int_{\Omega} \Lambda_{\delta}\Big(\frac{x}{\varepsilon} \Big) P^{\alpha}_{\varepsilon,\delta,i} \frac{\partial \psi^{\alpha}}{\partial x_i} \,dx +  \int_{\Omega}  Q^{\alpha}_{\varepsilon,\delta,i} \frac{\partial \psi^{\alpha}}{\partial x_i} \,dx .
	\end{equation}
	
	\emph{Case 1: $0<\delta <1$}. Let $\psi = w_{\varepsilon,\delta}$ in the formulation above, then by the ellipticity of $(a^{\alpha\beta}_{ij})$ we get
	\begin{equation}\label{weak011}
		\delta \int_{\Omega} | \nabla w_{\varepsilon,\delta} |^2 \,dx \leq C\big(\| P_{\varepsilon,\delta} \|_{L^2(\Omega)} + \|Q_{\varepsilon,\delta} \|_{L^2(\Omega) } \big) \| \nabla w_{\varepsilon,\delta} \|_{L^2(\Omega)}, 
	\end{equation}
	and hence, for some universal constant $C<\infty$ independent of $\eps,\delta$, 
	$$
	\| \nabla w_{\varepsilon,\delta} \|_{L^2(\Omega)} \leq C\delta^{-1}  \big(\| P_{\varepsilon,\delta} \|_{L^2(\Omega)} + \|Q_{\varepsilon,\delta} \|_{L^2(\Omega) } \big).
	$$ 
	
	Let $\mathcal{P}_{\varepsilon}$ be the extension operator given in Lemma \ref{lem:extension}, so that $\psi = \mathcal{P}_{\varepsilon} w_{\varepsilon,\delta}$ is an extension of $w_{\varepsilon,\delta}\rvert_{\Omega_\eps}$ to inside the obstacles. Use this extended function as $\psi$ in \eqref{weakdelta01} to get 
	\begin{equation*}
		\begin{aligned}
			\mu \int_{\Omega_{\varepsilon}}& |\nabla w_{\varepsilon,\delta}|^2\,dx  \le \int_{\Omega_\eps} A(x/\eps)\nabla w_{\eps,\delta} \cdot \nabla \mathcal{P}_\eps(w_{\eps,\delta}) \\
			&=
			\int_{\Omega} \Lambda_{\delta} \left(\frac{x}{\varepsilon}\right) A\left(\frac{x}{\varepsilon}\right)\nabla w_{\varepsilon,\delta} \cdot \nabla  \mathcal{P}_{\varepsilon} w_{\varepsilon,\delta}\,dx - \delta \int_{D_{\varepsilon}}  A\left(\frac{x}{\varepsilon}\right)\nabla w_{\varepsilon,\delta} \cdot \nabla  \mathcal{P}_{\varepsilon} w_{\varepsilon,\delta}\,dx \\
			&\leq \left|\int_\Omega \left(\Lambda_{\delta}(x/\eps)P^\alpha_{\eps,\delta} + Q^\alpha_{\eps,\delta}\right) \cdot \nabla \mathcal{P}_{\eps}(w_{\eps,\delta}^\alpha) \right| +\delta \left|\int_{D_{\varepsilon}}  A\left(\frac{x}{\varepsilon}\right)\nabla w_{\varepsilon,\delta} \cdot \nabla  \mathcal{P}_{\varepsilon} w_{\varepsilon,\delta}\,dx\right| \\
			& \leq C\big(\| P_{\varepsilon,\delta} \|_{L^2(\Omega)} + \|Q_{\varepsilon,\delta} \|_{L^2(\Omega) }  \big) \| \nabla w_{\varepsilon,\delta} \|_{L^2(\Omega_{\varepsilon})} + C \delta \| \nabla w_{\varepsilon,\delta} \|_{L^2(\Omega)} \| \nabla w_{\varepsilon,\delta} \|_{L^2(\Omega_{\varepsilon})} \\
			& \leq C\big(\| P_{\varepsilon,\delta} \|_{L^2(\Omega)} + \|Q_{\varepsilon,\delta} \|_{L^2(\Omega) }  \big) \| \nabla w_{\varepsilon,\delta} \|_{L^2(\Omega_{\varepsilon})}.
		\end{aligned}
	\end{equation*}
	Note the second inequality follows from \eqref{weakdelta01} and the boundedness of $\mathcal{P}_{\varepsilon}$, and also $\delta \le 1$. The last inequality follows from \eqref{weak011}. Estimate \eqref{01first} follows, for $\delta \in (0,1]$.
	\smallskip

	\emph{Case 2: $1<\delta <\infty$}. Let $\psi = w_{\varepsilon,\delta}$ in \eqref{weakdelta01}, by Cauchy inequality we get
	\begin{equation*}
		\begin{aligned}
			\big(\| \nabla &w_{\varepsilon,\delta} \|_{L^2(\Omega_{\varepsilon})}  + \delta^{1/2 } \| \nabla w_{\varepsilon,\delta} \|_{L^2(D_{\varepsilon})}\big)^2 \\
			& \leq C \| P_{\varepsilon,\delta} \|_{L^2(\Omega_{\varepsilon})} \| \nabla w_{\varepsilon,\delta} \|_{L^2(\Omega_{\varepsilon})}  +C\delta \| P_{\varepsilon,\delta} \|_{L^2(D_{\varepsilon})}  \| \nabla w_{\varepsilon,\delta} \|_{L^2(D_{\varepsilon})} \\
			&\ \ \ +C \|Q_{\varepsilon,\delta}\|_{L^2(\Omega)} \| \nabla w_{\varepsilon,\delta} \|_{L^2(\Omega)}  \\
			& \leq C\big(  \delta^{1/2}\|P_{\varepsilon,\delta}\|_{L^2(\Omega)} +\|Q_{\varepsilon,\delta}\|_{L^2(\Omega)}  \big)\big(\| \nabla w_{\varepsilon,\delta} \|_{L^2(\Omega_{\varepsilon})}  + \delta^{1/2 } \| \nabla w_{\varepsilon,\delta} \|_{L^2(D_{\varepsilon})}\big),
		\end{aligned}
	\end{equation*}
	so we get
	$$
	\| \nabla w_{\varepsilon,\delta} \|_{L^2(\Omega_{\varepsilon})}  + \delta^{1/2 } \| \nabla w_{\varepsilon,\delta} \|_{L^2(D_{\varepsilon})} \leq C\big(  \delta^{1/2}\|P_{\varepsilon,\delta}\|_{L^2(\Omega)} +\|Q_{\varepsilon,\delta}\|_{L^2(\Omega)}  \big),
	$$
	which implies \eqref{1inftyfirst}, for $\delta \in (1,\infty)$. 
\end{proof}

\subsection{A decomposition of the equations for the discrepancy function.} We rewrite the problem \eqref{eq:equation_discrepancy} for the discrepancy function $w_{\eps,\delta}$ as a transmission problem:
\begin{equation}\left\{
	\begin{aligned}
		&\big\{ \Lambda_{\delta}(x/\varepsilon) \mathcal{L}_{\varepsilon,1} (w_{\varepsilon,\delta}) \big\}^{\alpha}= -\Lambda_{\delta}(x/\varepsilon) \frac{\partial P^{\alpha}_{\varepsilon,\delta,i} }{\partial x_i}  - \frac{\partial Q_{\varepsilon,\delta,i}^{\alpha}}{\partial x_i} & \mathrm{in} \ \Omega \setminus \partial D_{\varepsilon}, \\
		& \Big\{ \delta \frac{\partial w_{\varepsilon,\delta}}{\partial \nu}\Big|_- - \frac{\partial w_{\varepsilon,\delta}}{\partial \nu}\Big|_+  \Big\}^{\alpha} =  \delta P^{\alpha}_{\varepsilon,\delta,i} n_i \big|_- - P^{\alpha}_{\varepsilon,\delta,i} n_i \big|_+  & \mathrm{on} \ \partial D_{\varepsilon} , \\
		& w_{\varepsilon,\delta} = 0 & \mathrm{on}\ \partial \Omega.
	\end{aligned}\right.
\end{equation}
Let $w_{\varepsilon,\delta}^1 \in H^1_0(\Omega;\R^m)$ be the solution of
\begin{equation}\label{w1}
	\left\{
	\begin{aligned}
		& \big\{  \mathcal{L}_{\varepsilon,1} (w^1_{\varepsilon,\delta}) \big\}^{\alpha}= - \frac{\partial P^{\alpha}_{\varepsilon,\delta,i} }{\partial x_i}   - \frac{\partial Q^{\alpha}_{\varepsilon,\delta,i} }{\partial x_i}   & \mathrm{in} \ \Omega, \\
		& w^1_{\varepsilon,\delta} = 0 & \mathrm{on} \ \partial \Omega.
	\end{aligned}\right.
\end{equation}
Let $w_{\varepsilon,\delta}^2 \in H^1(D_\eps;\R^m)$ be the solution of
\begin{equation}\label{w2}
	\left\{
	\begin{aligned}
		& \big\{  \mathcal{L}_{\varepsilon,1} (w^2_{\varepsilon,\delta}) \big\}^{\alpha}= - \frac{\partial P^{\alpha}_{\varepsilon,\delta,i} }{\partial x_i}   - \frac{1}{\delta} \frac{\partial Q^{\alpha}_{\varepsilon,\delta,i} }{\partial x_i}   & \mathrm{in} \ D_{\varepsilon}, \\
		& w^2_{\varepsilon,\delta} = w^1_{\varepsilon,\delta}  & \mathrm{on} \ \partial D_{\varepsilon}.
	\end{aligned}\right.
\end{equation}
Note that $w_{\varepsilon,\delta}^2$ is picewisely defined on each component of $D_\eps$. 

\begin{lemma}
	Let $w^1_{\varepsilon,\delta}$ and $w^2_{\varepsilon,\delta}$ be defined by \eqref{w1} and \eqref{w2}, respectively. Let $u=\mathcal{L}_{D_{\varepsilon}}^{-1}(f)$ be the unique solution (extended by zero outside $D_\eps$) of
	\begin{equation}\label{auxioperator}
		\left\{
		\begin{aligned}
			& \mathcal{L}_{\eps,1}(u)= f & \mathrm{in} \ D_{\varepsilon} , \\
			& u=0& \mathrm{on}\ \Omega \setminus D_{\varepsilon},
		\end{aligned}\right.
	\end{equation}
	Let $R_{\varepsilon,\delta} = (R_{\varepsilon,\delta}^{\alpha})$ be defined by
	\begin{equation}
	\label{eq:Rde}
		R_{\varepsilon,\delta}^{\alpha} = \Big[ a_{ij}^{\alpha \beta}\Big( \frac{x}{\varepsilon} \Big) +a_{ik}^{\alpha \gamma}\Big( \frac{x}{\varepsilon} \Big)  \frac{\partial }{\partial x_k} \chi_{\delta,j}^{\gamma \beta} \Big( \frac{x}{\varepsilon} \Big)  \Big] \frac{\partial}{\partial x_i}  S_{\varepsilon} \Big( \eta_{\varepsilon} \frac{\partial \widehat{u}_{\delta}^{\beta}}{\partial x_j} \Big). 
	\end{equation}
Then there exits a universal constant $C$ that is independent of $\eps$ and $\delta$ so that
	\begin{eqnarray}
	\label{estiw1}
		&&\| \nabla w^1_{\varepsilon,\delta} \|_{L^2(\Omega)} \leq C\big( \| P_{\varepsilon,\delta} \|_{L^2(\Omega) } +  \| Q_{\varepsilon,\delta} \|_{L^2(\Omega) }  \big),\\
	%\end{equation}
	%and 
	%\begin{equation}
	\label{estiw2modi}
		&&\big\| \nabla \big(w^2_{\varepsilon,\delta} - \delta^{-1} \mathcal{L}_{D_{\varepsilon}}^{-1} (f) \big) \big\|_{L^2(D_{\varepsilon})} \leq C\big( \| P_{\varepsilon,\delta} \big\|_{L^2(\Omega) } +  \| Q_{\varepsilon,\delta} \|_{L^2(\Omega) }   + \varepsilon \|R_{\varepsilon,\delta} \|_{L^2(\Omega)} \big).
	\end{eqnarray}
\end{lemma}
\begin{proof}
	Note that the differential operator $\mathcal{L}_{\eps,1}$ has uniform elliptic coefficients $(a^{\alpha\beta}_{ij}(\tfrac{\cdot}{\eps}))$; as a result \eqref{estiw1} is due to the basic energy estimate for the problem \eqref{w1} of $w^1_{\varepsilon,\delta}$. 
	
	% Next we need to estimate $w^2_{\varepsilon,\delta}$. For any $\psi \in H_0^1(D_{\varepsilon})$, we compute 
	% \begin{equation*}
	% 	\begin{aligned}
	% 		\int_{D_{\varepsilon}}& Q_{\varepsilon,\delta,i}^{\alpha} \frac{\partial \psi^{\alpha}}{\partial x_i}  \,dx\\
	% 		& = \int_{D_{\varepsilon}} F_{\delta,ij}^{\alpha \beta}\Big(  \frac{x}{\varepsilon}\Big)  \frac{\partial}{\partial x_i}  S_{\varepsilon} \Big( \eta_{\varepsilon} \frac{\partial \widehat{u}_{\delta}^{\beta}}{\partial x_j} \Big) \psi^{\alpha} \,dx \\
	% 		& \ \ \ - \int_{D_{\varepsilon}} \widehat{a}_{\delta,ij}^{\alpha \beta} \Big[ S_{\varepsilon} \Big(\eta_{\varepsilon} \frac{\partial \widehat{u}_{\delta}^{\beta} }{\partial x_j} \Big)  -\frac{\partial \widehat{u}^{\beta}_{\delta} }{\partial x_j} \Big]   \frac{\partial \psi^{\alpha}}{\partial x_i}\,dx \\
	% 		& = \delta \int_{D_{\varepsilon}}   \Big[ a_{ij}^{\alpha \beta}\Big( \frac{x}{\varepsilon} \Big) +a_{ik}^{\alpha \gamma}\Big( \frac{x}{\varepsilon} \Big)  \frac{\partial }{\partial x_k} \chi_{\delta,j}^{\gamma \beta} \Big( \frac{x}{\varepsilon} \Big)  \Big] \frac{\partial}{\partial x_i}  S_{\varepsilon} \Big( \eta_{\varepsilon} \frac{\partial \widehat{u}_{\delta}^{\beta}}{\partial x_j} \Big) \psi^{\alpha} \,dx \\
	% 		&\ \ \ + \int_{D_{\varepsilon}} F^{\alpha }\psi^{\alpha}\,dx,
	% 	\end{aligned}
	% \end{equation*}
	% where the first identity uses the anti-symmetry of the flux correcor $\Psi_{\delta}$, and the last identity follows from $- \mathrm{div}(\widehat{A}_{\delta} \nabla \widehat{u}_{\delta}) =F $. Therefore,

	To obtain estimate for $w^2_{\eps,\delta}$, using the definitions \eqref{eq:Qed} and \eqref{def:Fdelta}, the relation \eqref{eq:FtoPsi} and the homogenized problem \eqref{homogenized problem}, we first check by direct computation that
	\begin{equation}
		- \frac{ \partial Q_{\varepsilon,\delta,i} }{\partial x_i}  = \delta R^{\alpha}_{\varepsilon,\delta} +f \qquad \mathrm{in}\ D_{\varepsilon},
	\end{equation}
	where $R_{\eps,\delta}$ is defined as in \eqref{eq:Rde}. Comparing the equations \eqref{w1}\eqref{eq:w2modi} and \eqref{auxioperator} and using the above formula, we get
	\begin{equation}\label{eq:w2modi}
		\left\{
		\begin{aligned}
			& \mathcal{L}_{\varepsilon,1} \big(w^2_{\varepsilon,\delta} - \delta^{-1} \mathcal{L}_{D_{\varepsilon}}^{-1}(F) - w^1_{\varepsilon,\delta}\big)  = - \frac{\partial Q_{\varepsilon,\delta,i} }{\partial x_i}   +R_{\varepsilon,\delta}    & \mathrm{in} \ D_{\varepsilon}, \\
			& w^2_{\varepsilon,\delta} - \delta^{-1} \mathcal{L}_{D_{\varepsilon}}^{-1}(f) - w^1_{\varepsilon,\delta} =0 & \mathrm{on} \ \partial D_{\varepsilon}.
		\end{aligned}\right.
	\end{equation}
Multiply on both sides by $v = w^2_{\varepsilon,\delta} - \delta^{-1} \mathcal{L}_{D_{\varepsilon}}^{-1}(F) - w^1_{\varepsilon,\delta}$, which belongs to $H^1_0(D_\eps)$, and integrate by parts. We obtain, using again the uniform ellipticity of the coefficients in $\mathcal{L}_{\eps,1}$, 
	\begin{equation}\label{w_2es}
	\begin{aligned}
	\mu\|\nabla v\|_{L^2(D_\eps)}^2 \le \|Q_{\eps,\delta}\|_{L^2(D_\eps)}\|\nabla v\|_{L^2(D_\eps)} + \|R_{\eps,\delta}\|_{L^2(D_\eps)} \|v\|_{L^2(D_\eps)}.
	\end{aligned}
		% \begin{aligned}
		% 	 \int_{D_{\varepsilon}} &\big|\nabla (w^2_{\varepsilon,\delta} - \delta^{-1} \mathcal{L}_{D_{\varepsilon}}^{-1}f )\big| ^2 \\
		% 	 & \leq \int_{D_{\varepsilon}} P_{\varepsilon,\delta} \nabla \big(w^2_{\varepsilon,\delta} - \delta^{-1} \mathcal{L}_{D_{\varepsilon}}^{-1}(F) \big)\,dx- \int_{D_{\varepsilon}} P_{\varepsilon,\delta} \nabla w^1_{\varepsilon,\delta}\,dx \\
		% 	&\ \ \ + \int_{D_{\varepsilon}} R_{\varepsilon,\delta} \big( w^2_{\varepsilon,\delta} - \delta^{-1} \mathcal{L}_{D_{\varepsilon}}^{-1}(F)- w^1_{\varepsilon,\delta}  \big)\,dx \\
		% 	& \leq \| P_{\varepsilon,\delta} \|_{L^2(\Omega)} \| \nabla (w^2_{\varepsilon,\delta} - \delta^{-1} \mathcal{L}_{D_{\varepsilon}}^{-1}f ) \|_{L^2(D_{\varepsilon})} + \| P_{\varepsilon,\delta} \|_{L^2(\Omega)} \| \nabla w^1_{\varepsilon,\delta} \|_{L^2(\Omega)} \\
		% 	& \ \ \ \ + \| R_{\varepsilon,\delta} \|_{L^2(\Omega)} 
		% 	\|  w^2_{\varepsilon,\delta} - \delta^{-1} \mathcal{L}_{D_{\varepsilon}}^{-1}f  -w^1_{\varepsilon,\delta}  \|_{L^2(D_{\varepsilon})} \\
		% 	& \leq C\big( \| P_{\varepsilon,\delta} \|_{L^2(\Omega)}  +\varepsilon  \| R_{\varepsilon,\delta} \|_{L^2(\Omega)}  \big) \| \nabla (w^2_{\varepsilon,\delta} - \delta^{-1} \mathcal{L}_{D_{\varepsilon}}^{-1}f ) \|_{L^2(D_{\varepsilon})} \\
		% 	& \ \ \ \ + C\big(  \| P_{\varepsilon,\delta} \|_{L^2(\Omega)}  + \varepsilon \| R_{\varepsilon,\delta} \|_{L^2(\Omega)}   \big) \big( \| P_{\varepsilon,\delta} \|_{L^2(\Omega) } +  \| Q_{\varepsilon,\delta} \|_{L^2(\Omega) }  \big),
		% \end{aligned}
	\end{equation}
	Note that $D_\eps$ consists of a collection of obstacles with diameter of order $\eps$. Using Poicar\'e inequality on each of the obstacle we have
\begin{equation*}
		\| v \|_{L^2(D_{\varepsilon})} \leq C \varepsilon \| \nabla  v \|_{L^2(D_{\varepsilon})} \qquad \mathrm{for\ any\ }u \in H_0^1(D_{\varepsilon}).
	\end{equation*}
	%where the last inequality follows from \eqref{estiw1} and the following rescaled Poincar\'{e} inequality
	% \begin{equation*}
	% 	\| u \|_{L^2(D_{\varepsilon})} \leq C \varepsilon \| \nabla  u \|_{L^2(D_{\varepsilon})} \qquad \mathrm{for\ any\ }u \in H_0^1(D_{\varepsilon}).
	% \end{equation*}
	%The estimate \eqref{estiw2modi} immediatelly follows from \eqref{w_2es} the simple fact that $X^2 \leq A(X +B)$ implies that $|X| \leq 4(A+B)$ for $A,B>0$. And the estimate \eqref{estiwR} follows from \eqref{estiw1} and \eqref{estiw2modi}.
	We hence get
	\begin{equation*}
	\| \nabla\left(w^2_{\varepsilon,\delta} - \delta^{-1} \mathcal{L}_{D_{\varepsilon}}^{-1}(f) -w^1_{\eps,\delta}\right)\|_{L^2(D_{\varepsilon})} \leq C\left(\|Q_{\eps,\delta}\|_{L^2(D_\eps)} + \eps \|R_{\eps,\delta}\|_{L^2(D_\eps)}\right).
	\end{equation*}
	The desired estimate \eqref{w_2es} follows then from \eqref{estiw1}.
\end{proof}

\begin{lemma} 
Assume \eqref{ellipticc}-\eqref{periodicc} and $(\mathbf{H2})$,
\begin{equation}\label{verifymainconclusion}
		\| P_{\varepsilon,\delta} \|_{L^2(\Omega) } +  \| Q_{\varepsilon,\delta} \|_{L^2(\Omega) }   + \varepsilon \|R_{\varepsilon,\delta} \|_{L^2(\Omega)}  \leq C \varepsilon^{1/2} \big(\|f \|_{L^2(\Omega) } + \| \nabla_{\mathrm{tan}} g \|_{L^2(\partial \Omega)} \big).
	\end{equation}
\end{lemma}
	\begin{proof}
	By the definition of $P_{\varepsilon,\delta} , Q_{\varepsilon,\delta}, R_{\varepsilon,\delta}$ and Lemma \ref{appen_auxi_0}  below, we have
	\begin{equation*}
		\begin{aligned}
			\| P_{\varepsilon,\delta} \|_{L^2(\Omega) } &+  \| Q_{\varepsilon,\delta} \|_{L^2(\Omega) }   + \varepsilon \|R_{\varepsilon,\delta} \|_{L^2(\Omega)}  \\
			&\leq C\| A \|_{L^{\infty}(Y)}\big(1+ \|\widehat{A}_{\delta} |_{L^{\infty}(\Omega)} \big)\Big(  \| S_{\varepsilon } ( \eta_{\varepsilon} \nabla \widehat{u}_{\delta}  )  - \nabla \widehat{u}_{\delta}  \|_{L^2(\Omega)} \\
			& \ \ \ + \varepsilon \| \chi_{\delta} \|_{H^1(Y)} \| \nabla (\eta_{\varepsilon} \nabla \widehat{u}_{\delta}) \|_{L^2(\Omega)}  + \varepsilon \| \Psi_{\delta} \|_{L^2(\Omega)} \| \nabla (\eta_{\varepsilon} \nabla \widehat{u}_{\delta}) \|_{L^2(\Omega)}      \Big).
		\end{aligned}
	\end{equation*}
	Since 
	\begin{equation*}
		\| S_{\varepsilon } ( \eta_{\varepsilon} \nabla \widehat{u}_{\delta}  )  - \nabla \widehat{u}_{\delta}  \|_{L^2(\Omega)} \leq C\varepsilon \| \nabla (\eta_{\varepsilon} \nabla \widehat{u}_{\delta}) \|_{L^2(\Omega)}    + C \| \nabla \widehat{u}_{\delta}  \|_{L^2(O_{4\varepsilon} )},
	\end{equation*}
	and 
	\begin{equation*}
		\| \nabla (\eta_{\varepsilon} \nabla \widehat{u}_{\delta}) \|_{L^2(\Omega)}     \leq \varepsilon^{-1} \| \nabla \widehat{u}_{\delta} \|_{L^2( O_{4\varepsilon} )} + \| \nabla^2 \widehat{u}_{\delta} \|_{L^2(\Omega \setminus O_{3\varepsilon})},
	\end{equation*}
	\eqref{verifymainconclusion} then follows from Lemma \ref{elliptic_regularity_estimate}, the uniform ellipticity of $\widehat{A}_{\delta}$, the uniform $L^2$ bound of $\Psi_{\delta}$ and the uniform $H^1$ bound of $\chi_{\delta}$. 
	\end{proof}

We extend $w_{\eps,\delta}^2$ to the outside $D_\eps$ by $w_{\eps,\delta}^1$. In other words, define 
\begin{equation}
\label{eq:wedR}
w^{\mathrm{R}}_{\varepsilon,\delta} := w^1_{\varepsilon,\delta} \mathbbm{1}_{\Omega_{\varepsilon}} + w_{\varepsilon,\delta}^2 \mathbbm{1}_{D_{\varepsilon}}.
\end{equation}
Then we check that $w^{\mathrm{R}}_{\eps,\delta} \in H^1_0(\Omega;\R^m)$ and below it is referred to as the ``regular'' part of $w_{\eps,\delta}$. In view of the estimates \eqref{estiw1} and \eqref{w_2es}, we have 
	\begin{equation}\label{estiwR}
		\big\| \nabla \big(w^\mathrm{R}_{\varepsilon,\delta} - \delta^{-1} \mathcal{L}_{D_{\varepsilon}}^{-1} (F) \big) \big\|_{L^2(\Omega)} \leq C\big( \| P_{\varepsilon,\delta} \|_{L^2(\Omega) } +  \| Q_{\varepsilon,\delta} \|_{L^2(\Omega) }   + \varepsilon \|R_{\varepsilon,\delta} \|_{L^2(\Omega)} \big) . 
	\end{equation}

Similarly, denote by the differece between $w^{\mathrm{S}}_{\eps,\delta}$ and $w_{\eps,\delta}$ by
\begin{equation}
\label{eq:wedS}
w^{\mathrm{S}}_{\varepsilon,\delta} : = w_{\varepsilon,\delta} -  w^{\mathrm{R}}_{\varepsilon,\delta}.
\end{equation} 
Then $w^{\mathrm{S}}_{\eps,\delta} \in H^1_0(\Omega;\R^m)$ as well and below it is referred to as the ``singular part''. Moreover, by definition we have 
\begin{equation}
\label{eq:wSeq}
		\mathcal{L}_{\varepsilon,1} (w^{\mathrm{S}}_{\varepsilon,\delta}) = 0 \qquad \mathrm{in} \ \Omega \setminus \partial D_{\varepsilon}.
\end{equation}

The next lemma plays an important role in our proof of Theorem \ref{thm:suboptimal convergence rate}. It says that the distributions of energy in $w^{\mathrm{S}}_{\eps,\delta}$ (in terms of the $L^2$ norm of its gradient) is balanced inside and outside of the obstacles $D_\eps$.

\begin{lemma}\label{lem:singlularbelong}
	Assume  \eqref{ellipticc}-\eqref{periodicc} and $(\mathbf{H2})$, then 
	\begin{equation}
		C \int_{D_{\varepsilon}} |\nabla w^{\mathrm{S}}_{\varepsilon,\delta}|^2\,dx \leq \int_{\Omega_{\varepsilon}} |\nabla w^{\mathrm{S}}_{\varepsilon,\delta}|^2\,dx  \leq C^{-1} \int_{D_{\varepsilon}} |\nabla w^{\mathrm{S}}_{\varepsilon,\delta}|^2\,dx .
	\end{equation}
\end{lemma}
\begin{proof}
	For simplicity, we simply denote $w^{\mathrm{S}}_{\eps,\delta}$ by $w$ in this proof. We divide the proof to several steps.
	
	\textit{Step 1: We verify the first inequality by a standard extension argument}. Consider the extension operator $\mathcal{P}_\eps : H^1(\Omega_\eps) \to H^1(\Omega)$ defined in Lemma \ref{lem:extension}. Let $\widetilde w = \mathcal{P}_\eps(w\rvert_{\Omega_\eps})$. Then we have $\widetilde w-w$ vanishes at $\partial D_\eps$. Multiply on both sides of \eqref{eq:wSeq} and integrate in $D_\eps$, we find 
	\begin{equation*}
	\|\nabla w\|_{L^2(D_\eps)} \le C\|\widetilde w\|_{L^2(D_\eps)} \le C\|w\|_{L^2(\Omega_\eps)}.
	\end{equation*}
	where $C$ only depends on the ellipticity bounds and the model domain $\omega$.
	
	\textit{Step 2. We verify that for each inclusion $\omega^{\mathbf{n}}_{\varepsilon} \subset \Omega$,
	\begin{equation}
	\label{eq:wSorth1}
		\int_{\partial \omega_{\varepsilon}^{\mathbf{n}} } \frac{\partial w^{\mathrm{S}}_{\varepsilon,\delta}}{\partial \nu}\Big|_+ \,d\sigma =0.
	\end{equation}
	}
The above is understood as the pairing of $\partial_\nu w^{\mathrm{S}}_{\eps,\delta}\rvert_+$ with constant functions. By definition, along the boundary of each inclusion, we have
\begin{equation*}
\partial_\nu w^{\mathrm{S}}_{\eps,\delta}\rvert_+ = \partial_\nu w_{\eps,\delta}\rvert_+ - \partial_\nu w^1_{\eps,\delta}\rvert_+,
\end{equation*}
where we used the fact that $w^{\mathrm{R}}_{\eps,\delta} = w^1_{\eps,\delta}$ in $\Omega_\eps$, i.e., outside the inclusions. By an inspection of the equations satisfied by $w_{\eps,\delta}$ and $w^1_{\eps,\delta}$, we have
\begin{equation*}
\begin{aligned}
&\partial_\nu w_{\eps,\delta} \rvert_+ - N^i P_{\eps,\delta,i}\rvert_+ - N^i Q_{\eps,\delta,i}\rvert_+ = \delta \partial_\nu w_{\eps,\delta} \rvert_- - \delta N^i P_{\eps,\delta,i}\rvert_-  - N^i Q_{\eps,\delta,i}\rvert_-,\\
&\partial_\nu w^1_{\eps,\delta} \rvert_+ - N^i P_{\eps,\delta,i}\rvert_+ - N^i Q_{\eps,\delta,i}\rvert_+ = \partial_\nu w^1_{\eps,\delta} \rvert_- - N^i P_{\eps,\delta,i}\rvert_-  - N^i Q_{\eps,\delta,i}\rvert_-,
\end{aligned}
\end{equation*}
which leads to
\begin{equation*}
\partial_\nu w^{\mathrm{S}}_{\eps,\delta} \rvert_+ = \delta \partial_\nu w_{\eps,\delta} \rvert_- -  \partial_\nu w^1_{\eps,\delta}\rvert_- -(\delta -1)N^i P_{\eps,\delta,i}\rvert_-.
\end{equation*}
Integrate the equation of $w_{\eps,\delta}$ and the equation of $w^1_{\eps,\delta}$ in $\omega^{\mathrm{n}}_\eps$, we get
\begin{equation*}
\begin{aligned}
&\int_{\partial \omega^{\mathrm{n}}_\eps} \delta \partial_\nu w_{\eps,\delta} \rvert_- - \delta N^i P_{\eps,\delta,i} - N^i Q_{\eps,\delta,i} \,d\sigma = 0,\\
&\int_{\partial \omega^{\mathrm{n}}_\eps} \partial_\nu w^1_{\eps,\delta} \rvert_- - N^i P_{\eps,\delta,i} - N^i Q_{\eps,\delta,i} \,d\sigma = 0.
\end{aligned}
\end{equation*}
Combining the results above we get \eqref{eq:wSorth1}.

	\textit{Step 3. Proof of the second inequality}. Consider the map $\mathcal{Q}_\eps$ defined in Lemma \ref{lem:extendout}, and apply it to $w\rvert_{D_\eps}$ to define
	\begin{equation*}
	 \widetilde w = \mathcal{Q}_\eps(w\rvert_{D_\eps}).
	 \end{equation*} 
	 Then $\widetilde w\in H^1_0(\Omega)$ (in fact, $\tilde w$ is supported in the union of bubbles surrounding compoenents of $D_\eps$). Moreover,
	 \begin{equation*}
	 \widetilde w - w = q_\eps^{\mathbf{n}}, \qquad \text{in each }\, \omega_\eps^{\mathbf{n}}.
	 \end{equation*}
	 Multiply by $\widetilde w-w$ on both sides of \eqref{eq:wSeq} and integrate by parts. We obtain
	 \begin{equation*}
	 \int_{\Omega_\eps} A^\eps(x)\nabla w \cdot \nabla(w-\widetilde w)\,dx = \int_{\partial \Omega_\eps} \frac{\partial w}{\partial \nu} \cdot (w-\widetilde w) = -\sum_{\mathbf{n}} \int_{\partial \omega_\eps^{\mathbf{n}}} \frac{\partial w}{\partial \nu}\Big\rvert_{+}\cdot (w-\widetilde w) = 0.  
	 \end{equation*}
	 From this we get by ellipticity that
	 \begin{equation*}
	 \|\nabla w\|_{L^2(\Omega_\eps)} \le C\|\nabla \mathcal{Q}_\eps(w)\|_{L^2(\Omega_\eps)} \le C\|\nabla w\|_{L^2(D_\eps)}.
	 \end{equation*}
	 The last inequality follows from Lemma \ref{lem:extendout}. This completes the proof.
\end{proof}

We are in a position to complete the proof of Theorem \ref{thm:suboptimal convergence rate}.

\begin{proof}[Proof of Theorem \ref{thm:suboptimal convergence rate}] By the decomposition above (see \eqref{eq:wedR} ad \eqref{eq:wedS}), we have
$$
w_{\eps,\delta} = w_{\eps,\delta}^{\mathrm{R}} + w_{\eps,\delta}^{\mathrm{S}}.
$$
In view of \eqref{estiwR} we only need to estimate $\|\nabla w^{\mathrm{S}}_{\eps,\delta}\|_{L^2(\Omega)}$.
\medskip

	\emph{Case 1: $0<\delta <1$}. By Lemma \ref{lem:singlularbelong}, we have
	\begin{equation}\label{wSOmega}
		\begin{aligned}
			\| \nabla w^{\mathrm{S}}_{\varepsilon,\delta} \|_{L^2(\Omega)} &\leq C\| \nabla w^{\mathrm{S}}_{\varepsilon,\delta} \|_{L^2(\Omega_{\varepsilon})} \leq C\| \nabla w^{\mathrm{R}}_{\varepsilon,\delta} \|_{L^2(\Omega_{\varepsilon})}  +C\| \nabla w_{\varepsilon,\delta} \|_{L^2(\Omega_{\varepsilon})} \\
			& \leq C\big( \| P_{\varepsilon,\delta} \|_{L^2(\Omega) } +  \| Q_{\varepsilon,\delta} \|_{L^2(\Omega) }   + \varepsilon \|R_{\varepsilon,\delta} \|_{L^2(\Omega)} \big) ,
		\end{aligned}
	\end{equation}
	where the last equality follows from  
	\eqref{01first} and \eqref{estiwR} (noting that $\mathcal{L}_{D_\eps}^{-1}(f)$ is supported outside $\Omega_\eps$).
	
	\medskip

	\emph{Case 2: $1\le \delta <\infty$}. By Lemma \ref{lem:singlularbelong}, we have 
	\begin{equation*}
		\begin{aligned}
			\| \nabla w^{\mathrm{S}}_{\varepsilon,\delta} \|_{L^2(\Omega)} &\leq C\| \nabla w^{\mathrm{S}}_{\varepsilon,\delta} \|_{L^2(D_{\varepsilon})} \\
			&\leq C\| \nabla \big( w^{\mathrm{R}}_{\varepsilon,\delta}  -  \delta^{-1} \mathcal{L}_{D_{\varepsilon}}^{-1} (f) \big) \|_{L^2(D_{\varepsilon})}  +C \delta^{-1}   \| \nabla  \mathcal{L}_{D_{\varepsilon}}^{-1} (f) \|_{L^2(D_{\varepsilon})}  +C\| \nabla w_{\varepsilon,\delta} \|_{L^2(D_{\varepsilon})}\\
			%&\ \ \ +C\| \nabla w_{\varepsilon,\delta} \|_{L^2(D_{\varepsilon})} \\
			& \leq C\big( \| P_{\varepsilon,\delta} \|_{L^2(\Omega) } +  \| Q_{\varepsilon,\delta} \|_{L^2(\Omega) }   + \varepsilon \|R_{\varepsilon,\delta} \|_{L^2(\Omega)} + \varepsilon \| f \|_{L^2(\Omega)}\big) ,
		\end{aligned}
	\end{equation*}
	where the last equality follows from \eqref{estiwR}, \eqref{01first}, a standard energy estimate for \eqref{auxioperator}, and the fact that $\delta \ge 1$. 

	\medskip

	Combine the above estimates with \eqref{estiwR}. We conclude that, for all $\delta\in (0,\infty)$, we have found a universal constant $C<\infty$ such that
	\begin{equation*}
		\begin{aligned}
			\| \nabla \big( w_{\varepsilon,\delta} &  -  \delta^{-1} \mathcal{L}_{D_{\varepsilon}}^{-1} (F)\big) \|_{L^2(\Omega)}  %\leq \| \nabla w^{\mathrm{S}}_{\varepsilon,\delta} \|_{L^2(\Omega)}  + \big\| \nabla \big(w^\mathrm{R}_{\varepsilon,\delta} - \delta^{-1} \mathcal{L}_{D_{\varepsilon}}^{-1} (F) \big) \big\|_{L^2(\Omega)} \\
			%& 
			\leq C\big( \| P_{\varepsilon,\delta} \|_{L^2(\Omega) } +  \| Q_{\varepsilon,\delta} \|_{L^2(\Omega) }   + \varepsilon \|R_{\varepsilon,\delta} \|_{L^2(\Omega)} + \varepsilon \| F \|_{L^2(\Omega)}\big).
		\end{aligned}
	\end{equation*}
The desired estimate \eqref{eq:thm1-error} then follows from \eqref{verifymainconclusion}.
	\end{proof}

\begin{remark} \label{ratefortype 1}Note that using a boundary cut-off, for Type I domain, the same method of analysis still yields a similar $H^1$ convergence rate result away the boundary:
\begin{equation}
		\label{eq:thm1-error-modify}
			\begin{aligned}
				\Big\|   u_{\varepsilon,\delta} - \widehat{u}_{\delta} -  \varepsilon \chi_{\delta}\left(\frac{x}{\varepsilon}\right) S_{\varepsilon} \Big(\eta_{\varepsilon}\nabla \widehat{u}_{\delta}\Big) &- \delta^{-1} \mathcal{L}_{D_{\varepsilon}}^{-1}(f) \Big\|_{H^1(\Omega\setminus \mathcal{O}_{c})} \\
				&\leq C\varepsilon^{1/2} \big( \|f \|_{L^2(\Omega) } +\|\nabla_{\mathrm{tan}}g\|_{L^2(\partial \Omega)} \big)
			\end{aligned}
		\end{equation}
  where $c>0$ is a fixed number independent of $\eps$.
\end{remark}

%%%%%%%
%%%%%%%%%
%%%%%%%%%
\section{Caccioppoli-type inequalities}

We will prove the uniform Lipschitz estimate for
\begin{equation}
\label{eq:hpdemod}
\left\{
\begin{aligned}
&\mathcal{L}_{\eps,\delta}(u_\eps) = f_{\eps,\delta}, \qquad &&x\in \Omega,\\
&u_{\eps,\delta} = g, \qquad &&x\in \partial \Omega.
\end{aligned}
	\right.
\end{equation}
Here, $f_{\eps,\delta}=f$ for $\delta\in (1,\infty)$, and $f_{\eps,\delta}=\Lambda_{\varepsilon ,\delta} f$ if $\delta\in (0,1]$. We follow the now standard approach by Shen \cite{shen_periodic_2018}. A key step in the scheme is to establish (interior and boundary) Caccioppoli inequalities. In the high contrast setting, this turns out to be a non-trivial task especially for large $\delta$. The following rescaling property of $\mathcal{L}_{\eps,\delta}$ will be used throughout the rest of the paper: if we define $\tilde u(\cdot) = u(r \cdot)$ and $\tilde f(\cdot)= r^2 f(r\cdot)$, 
$$
\mathcal{L}_{\eps,\delta}(u) = f \quad \Leftrightarrow \quad \mathcal{L}_{\eps/r,\delta}(\tilde u) = \tilde f.
$$
In particular when $r=\eps$, we relate the operator $\mathcal{L}_{\eps,\delta}$ to $\mathcal{L}_{1,\delta}$. The latter operator was studied by Shen in \cite{shen_large-scale_2021}, where Caccioppoli inequalities and large scale interior Lipschitz regularity were proved. We modify Shen's proof to obtain Caccioppoli inequalities for large $\delta$ allowing the presence of source term $f$ and boundary data $g$. Those combined with the convergence rate result of the previous section allows us to obtain boundary Lipschitz regularity estimate in the high contrast setting. 

From here, we always assume assumptions $(\mathbf{H1})$ and $(\mathbf{H2})$ hold. We start from two basic energy estimates with cut-off functions.

\begin{lemma}\label{lem:energycutint} If $u\in H^1(\Omega;\R^m)$ is a solution to $\mathcal{L}_{\eps,\delta}(u) = f$ in $\Omega$. Then for any $\varphi \in C^1_0(\Omega)$,
\begin{equation}
\label{eq:energycutint}
\int_\Omega \Lambda_{\varepsilon,\delta} (x) \varphi^2 |\nabla u|^2 \le C\int_\Omega \Lambda_{\varepsilon,\delta} (x)  |u|^2 |\nabla \varphi|^2 + C\int_\Omega |f|\varphi^2 |u|.
\end{equation}
\end{lemma}
\begin{proof} Let $v = u\varphi^2 \in H^1_0(\Omega)$. Then we have
\begin{equation*}
\int_\Omega \Lambda_{\varepsilon,\delta} (x)  A(\tfrac{x}{\eps})\nabla u \nabla v = \int_\Omega f\cdot v,
\end{equation*}
the conclusion follows by integration by parts and the ellipticity of $A$, we refer \cite[Theorem 2.1.3]{shen_periodic_2018} to the readers for details.
\end{proof}

For a boundary version of Lemma \ref{lem:energycutint}, we use the following standard notation: let $x_0\in \partial \Omega$ and $r>0$, let $\mathbf{D}_r = \mathbf{D}_r(x_0,\psi)$ be the boundary cylinder
\begin{equation*}
\mathbf{D}_r = \Phi\{(x',z) \,:\, |x'-x_0|<r, 0<z - \psi(x')<M_0 r\},
\end{equation*}
where $(x',z) \in \R^{d-1}\times \mathbb{R}$ denotes local coordinate close to $x_0$, $\Phi = \Phi_{x_0}$ is an affine transform that maps $(0,0)$ in the new coordinate system to $x_0$, and allign the subspace $\{(x',0)\,:\,x'\in \R^{d-1}\}$ to the tangent space of $T_{x_0} \partial \Omega$, and $\psi: \R^{d-1}\to \R$ is a $C^{1,\alpha}$ (or Lipschitz) function in the new coordinate whose graph (after mapped by $\Phi$) corresponds to $\partial \Omega$, and $M_0$ is a regularity character of $\psi$. Let
\begin{equation*}
\Delta_r = \partial \mathbf{D}_r \cap \partial \Omega.
\end{equation*}
Note that $\Delta_r$ is part of $\partial \Omega$ and is a proper subset of $\partial \mathbf{D}_r$. It is clear that the boundary cylinders $\{\mathbf{D}_r(x_0)\}$ are comparable with $\{B_r(x_0)\cap \Omega\}$, in the sense that one class can cover the other. 

\begin{lemma}\label{lem:energycutbdr}
	If $u\in H^1(\Omega;\R^m)$ be a weak solution of $\mathcal{L}_{\varepsilon,\delta}( u) = f$ in $\Omega$ with $u= 0$ on $\partial\Omega$, then for any $\varphi \in C^1(\mathbf{D}_R)$ with support of $\varphi$ away from $\partial \mathbf{D}_R\cap \Omega$, we have
	\begin{equation}\label{eq:cacci1}
		%\int_{\mathbf{D}_R} \Lambda_{\varepsilon,\delta} (x)|\nabla u|^2|\psi|^2 \,dx\leq C \int_{\mathbf{D}_R} \Lambda_{\delta}(x/\varepsilon) | u|^2|\nabla \psi|^2 \,dx
		\int_{\mathbf{D}_R} \Lambda_{\varepsilon,\delta} (x)\varphi^2 |\nabla u|^2 \le C \int_{\mathbf{D}_R} \Lambda_{\varepsilon,\delta} (x) |u|^2 |\nabla \varphi|^2 + C\int_{\mathbf{D}_R} |f|\varphi^2 |u|.
	\end{equation}
\end{lemma}

The proof is exactly as in the interior case, because, thanks to the zero Dirichlet data $u=0$ on $\Delta_R$, the function $\varphi^2 u$ belongs to $H^1_0(\mathbf{D}_R)$. We hence omit the details.

%%%%%%%%
% relatively soft inclusions
\subsection{Caccioppoli inequalities for small $\delta$}

For relatively soft inclusions, the Caccioppoli inequality can be obtained by the standard method provided that the source term $f$ is of order $\delta$ in the inclusions. Recall that for $\delta \in (0,1]$, $f_{\eps,\delta}$ is defined by \eqref{eq:modf}.

Recall that $\widetilde{\omega}$ denotes an enlarging dilation of $\omega$ so that $\overline{\omega}\subset \widetilde{\omega} \subset  Q_{7/8}(0)$, and $\widetilde{\omega}^{\mathbf{n}}_\eps=\eps(\mathbf{n}+\widetilde{\omega})$ (simplified by $\widetilde{\omega}_{\varepsilon}$) is the dilation of $\omega^{\mathbf{n}}_{\varepsilon}$ (simplified by $\omega_{\varepsilon}$). The next lemma provides a way to control the energy \emph{inside} $\omega_\eps$ of a solution to the elliptic system by its energy \emph{outside}, namely the energy in $\womega\setminus \overline{\omega}_\eps$. 

\begin{lemma}\label{lemma5.3} If $u\in H^1(\womega_\eps;\R^m)$ solves $\mathcal{L}_{\eps,1}(u) = f$ in $\omega_\eps$, then
\begin{equation*}
%\int_{\omega_\eps} |\nabla u|^2 \le C\int_{\womega_\eps \setminus \ol\omega_\eps} |\nabla u|^2 + C\eps^2 \int_{\omega_\eps} |f|^2.
\|\nabla u\|_{L^2(\omega_\eps)} \le C\left(\eps\|f\|_{L^2(\omega_\eps)} + \|\nabla u\|_{L^2(\womega_\eps\setminus \ol \omega_\eps)}\right).
\end{equation*}
\end{lemma}
\begin{proof} This result is essentially Lemma 2.3 of \cite{shen_large-scale_2021} and it is proved by a simple extension argument. By rescaling we may assume $\eps =1$. Let $v$ be the extension of $u - \fint_{\womega\setminus \ol \omega}u$ from $\womega\setminus \ol\omega$ to $\womega$, then $\|v\|_{L^2(\womega)} \le C\|\nabla u\|_{L^2(\womega\setminus \ol\omega)}$. The desired result then follows by testing $u-v\in H^1_0(\omega)$ against the equation $\mathcal{L}_{1,1}(u)=f$. 
\end{proof}

\begin{theorem}[Caccioppoli inequalities for small $\delta$]
\label{lem:cacciosmalldel}
Suppose $\delta\in (0,1]$ and let $f_{\eps,\delta}$ be defined by \eqref{eq:modf}. 
\begin{itemize}
	\item[{\upshape(i)}] \emph{(Interior estimate)} Let $R\ge 4\eps$ and $Q_{2R}$ be a cube in $\Omega$. If $\mathcal{L}_{\eps,\delta}(u)=f_{\eps,\delta}$ in $Q_{2R}$, 
	then
\begin{equation}
\label{eq:cacciointsd}
\fint_{Q_R} |\nabla u|^2 \le C\left(\frac{1}{R^2}\fint_{Q_{2R}} |u|^2 + R^2\fint_{Q_{2R}}|f|^2\right).
\end{equation}
\item[{\upshape(ii)}] \emph{(Boundary estimate)} Let $R\ge 4\eps$   and $\mathbf{D}_R$ be a boundary cylinder. If $\mathcal{L}_{\eps,\delta}(u)=f_{\eps,\delta}$ in $\mathbf{D}_{2R}$ and $u=g$ in $\Delta_{2R}$, then
\begin{equation}
\label{eq:cacciobdrsd}
\fint_{\mathbf{D}_{3R/2}}|\nabla u|^2  \le C\left\{\frac{1}{R^2}\fint_{\mathbf{D}_{2R}} |u|^2 + R^2 \fint_{\mathbf{D}_{2R}}|f|^2 + \frac{1}{R^2}\|g\|^2_{L^\infty(\Delta_{2R})} + \|\nabla_{\rm tan} g\|^2_{L^\infty(\Delta_{2R})}\right\}.
\end{equation}
\end{itemize}
\end{theorem}

\begin{proof}[Proof of (i) -- the interior case] 
Let $\varphi \in C_0^1(Q_{2R};[0,1])$ such that $\varphi  \equiv 1$ in $Q_{3R/2}$ and $\| \nabla \varphi \|_{L^{\infty}} \leq C/R$, plug it in \eqref{eq:energycutint} and replacing $f$ there by $f_{\eps,\delta}$. Using Young's inequality, we obtain
 \begin{equation*}
 \begin{aligned}
	 	\int_{Q_{3R/2}} \Lambda_{\varepsilon,\delta}(x) |\nabla u|^2\leq \ &\frac{C}{R^2}\int_{Q_{2R}} \Lambda_{\varepsilon,\delta} (x)| u|^2 +C R^2 \int_{Q_{2R}} \Lambda_{\varepsilon,\delta}(x)|f|^2\\
	 	\leq  \ &\frac{C}{R^2}\int_{Q_{2R}} | u|^2 + CR^2 \int_{Q_{2R}} |f|^2.
\end{aligned}
	 \end{equation*}
	% \begin{equation*}
	% 	\int_{D_{3R/4}} \Lambda_{\delta}(x/\varepsilon) |\nabla u|^2\,dx \leq \frac{C}{R^2}\int_{D_R} \Lambda_{\delta} (x/\varepsilon)| u|^2\,dx .
	% \end{equation*}
	Now for each $\omega_\eps \subset D_\eps$ that has non-empty intersection with $Q_R$, its enlarged domain $\womega_\eps$ is contained in $Q_{R+2\eps} \subset Q_{3R/2}$. By Lemma \ref{lemma5.3}, we have
	\begin{equation*}
	\int_{Q_R \cap D_\eps} |\nabla u|^2 \le C\left(\int_{Q_{3R/2}\cap \Omega_\eps} |\nabla u|^2 + \eps^2 \int_{Q_R\cap D_\eps} \delta^{-2}|f_{\eps,\delta}|^2\right).
	\end{equation*}
	We combine the two estimates above and get
	\begin{equation*}
	\begin{aligned}
		\int_{Q_R} |\nabla u |^2\,dx \leq \ &C\left(\frac{1}{R^2} \int_{Q_{2R} } | u|^2  +R^2 \int_{Q_{2R}} |f|^2 + \varepsilon^2 \int_{Q_R\cap D_\eps} \delta^{-2}|f_{\eps,\delta}|^2\right)\\
		\leq \ & C\left(\frac{1}{R^2}\int_{Q_{2R}} |u|^2 + R^2\int_{Q_{2R}} |f|^2\right),
	\end{aligned}
	\end{equation*}
	here we used the fact that $\eps < R$ to deduce the second line. This completes the proof.
\end{proof}

\begin{proof}[Proof of (ii) -- the boundary case] By rescaling we assume $R=1$. We first solve the problem
\begin{equation*}
 \mathcal{L}_{\eps,\delta}(v) = h \ \  \text{in }\, \mathbf{D}_{2+4\eps}, \qquad v=\widetilde g \ \  \text{on }\, \partial \mathbf{D}_{2}.
 \end{equation*} 
Here, we let the inclusions in $\mathbf{D}_{2+4\eps}$ be exactly those of $D_\eps$ that had non-empty intersections with $\mathbf{D}_{2}$. Let $F$ denote the union of these inclusions, then the pair $(\mathbf{D}_{2+4\eps},F)$ satisfying Assumption $\mathbf{G}$ is a type-II perforated domain. The source term $h$ is still $\Lambda_{\eps,\delta}(x) f(x)$ except that the weight $\Lambda_{\eps,\delta}$ are determined by the locations of the inclusions just described. The boundary data $\widetilde g$ is an extension of $g\rvert_{\Delta_1}$ that is compactly supported in $\Delta_{3/2}$ and is zero on $\partial \mathbf{D}_2\cap \Omega$. We can apply Theorem \ref{thm:unifenergy} to get
\begin{equation*}
\|\nabla v\|_{L^2(\mathbf{D}_2)} \le C\left(\|f\|_{L^2(\mathbf{D}_2)} + \|g\|_{H^{1/2}(\Delta_2)}\right).
\end{equation*}
Note that $\mathcal{L}_{\eps,\delta}(u-v) = 0$ in $\mathbf{D}_{3/2}$ and $u-v=0$ in $\Delta_{3/2}$, so the general setting is reduced to the special setting of $g=0$ (in fact also $f=0$). Thanks to Lemma \ref{lem:energycutbdr}, the special setting is exactly the same as in the interior setting. 
\end{proof}

%%%%%%%%
% relatively hard inclusions
\subsection{Caccioppoli inequalities for large $\delta$}
\label{sec:caccioppoli}

For relatively hard obstacles, the Caccioppoli inequalities are harder to get. Shen \cite{shen_large-scale_2021} nevertheless established the interior version with zero source term $f$. We use his method to prove the boundary version and in a general setting.

\begin{theorem}[Caccioppoli inequality for large $\delta$ without source term]
\label{thm:caccioldshen}
Suppose $\delta \in (1,\infty)$. Let $R\ge 16\eps$, if $u\in H^1(Q_{2R};\R^m)$ that solves $\mathcal{L}_{\eps,\delta}(u)=0$ in $Q_{2R}$, and for any positive integer $\ell \in \N^*$, we have
\begin{equation}
\label{eq:caccioldshen}
\int_{Q_R} |\nabla u|^2 \le C_{\ell}\left(\frac{1}{R^2} \int_{Q_{2R}} |u|^2 + \left(\frac{\eps}{R}\right)^{2\ell} \int_{Q_{2R}}|\nabla u|^2 \right).
\end{equation}
\end{theorem}
This result is precisely Theorem 3.3 of Shen \cite{shen_large-scale_2021} after the rescaling $u \to \tilde u(\cdot) = u(\eps \cdot)$. Note that the source term is zero in the equation. Using the method in the proof of item (ii) of Theorem \ref{lem:cacciosmalldel}, we can get the general setting with a nonzero source $f$. 

On the other hand, we may repeat the argument of Shen while allowing the presence of the source term. The next three lemmas provide the detailed process.

\begin{lemma}\label{lem:energyinlarged} Let $\delta \in (1,\infty)$ and let $\omega_\eps \subset D_\eps$ be a typical obstacle. If $u\in H^1(\womega_\eps;\R^m)$ solves $\mathcal{L}_{\eps,\delta}(u) = f$ in $\womega_\eps$, we have
\begin{equation*}
\delta \|\nabla u\|_{L^2(\omega_\eps)} \le C\left\{\|f\|_{L^p(\womega_\eps)} + \|\nabla u\|_{L^2(\womega_\eps \setminus \ol\omega_\eps)} \right\},
\end{equation*}
for $p=\frac{2d}{d+2}$ if $d\ge 3$ and $p>1$ for $d=2$. 
\end{lemma}

\begin{proof} The choice of $p$ guarantees $\eps^{2-d/p} \le \eps^{1-d/2}$. By rescaling, we only need to consider the case $\eps=1$ and then the result is Lemma 3.1 of Shen \cite{shen_large-scale_2021}. We refer to the proof there. The key idea is to extend the function $u$ from the inside of $\omega_\eps$ to a function $v \in H^1_0(\womega_\eps;\R^m)$ and use it as the test function of the equation. 
\end{proof}

\begin{lemma}\label{lem:cacciold+} Let $\delta \in (1,\infty)$, $\varepsilon ,h\in (0,1)$, $R\ge 16\eps$ and $Q_R$ be a cube inside $\Omega$. If $u\in H^1(Q_R;\R^m)$ solves $\mathcal{L}_{\eps,\delta}(u) = f$ in $Q_R$, then for any $r \in [R/2,R-8\eps]$, we have
\begin{equation*}
\int_{Q_r}|\nabla u|^2 \,dx \le \frac{C}{h(R-r)^2} \int_{Q_R} |u|^2 \,dx + \left(h + \frac{C}{\delta}\right) \int_{Q_R} |\nabla u|^2\,dx + CR^2 \int_{Q_R}|f|^2.
\end{equation*}
\end{lemma}
\begin{proof} First note that the result only concerns interior of $\Omega$. By rescaling we may assume $\eps=1$. The result is then essentially Lemma 3.2 of Shen \cite{shen_large-scale_2021} except that a nonzero source term $f$ is allowed. We repeat the proof for the sake of completeness. 

Fix a cut-off function $\varphi \in C^1_0(Q_{R-2};[0,1])$ such that $\varphi \equiv 1$ in $Q_{r+3}$. For each $k\in \Z^d$ so that $\womega^k = \womega + k$ is still in $Q_R$, let $w_k$ be an extension of $(u\varphi^2 - c_k )\rvert_{\omega^k}$ to $H^1_0(\womega^k)$, where $c_k$ is the mean of $u\varphi^2$ in $\omega^k$. %is chosen so that we can apply Poincar\'e-Whirtinger to $u\varphi^2 - c_k$, 
Then we have
\begin{equation}\label{proof5.7}
\begin{aligned}
\|w_k\|_{H^1(\womega^k)} &\le  C\|\nabla (u\varphi^2)\|_{L^2(\omega^k)} \le C\|\nabla u\|_{L^2(\omega^k)} + C\|u\nabla \varphi\|_{L^2(\omega^k)}\\
&\le C\delta^{-1} \|\nabla u\|_{L^2(\womega^k\setminus \ol\omega^k)} + C\delta^{-1} \|f\|_{L^2(\womega^k)} +  C\|u\nabla \varphi\|_{L^2(\omega^k)}.
\end{aligned}
\end{equation}
Here the first inequality results from the extension estimate and the Poincar\'e-Wirtinger inequality. The second line is due to Lemma \ref{lem:energyinlarged} (by noting $p<2$ for $d\ge 3$ and by choosing $p=2$ for $d=2$).

Now let $\phi = u\varphi^2 - \sum_k w_k \in H^1_0(Q_R)$, where $k$ is summed over the range described above. Let $F$ denote the closure of the union of such $\omega^k$'s. Then because $\phi =c_k$ are constants in $\omega^k$, we get
\begin{equation*}
\int_{Q_R\setminus F} A\nabla u\cdot \nabla \phi \,dx = \int_{Q_R} f\cdot \phi.
\end{equation*}
In view of the definition of $\phi$, we obtain
\begin{equation}
\label{eq:caccio-1-1}
\left|\int_{Q_R\setminus F} A\nabla u \cdot \nabla(u\varphi^2)\,dx\right|  \le \int_{Q_R}|f||\phi| + C\sum_k \|\nabla u\|_{L^2(\womega^k\setminus \ol\omega^k)} \|\nabla w_k\|_{L^2(\womega^k \setminus \ol\omega^k)}.
\end{equation}
For the term involving $w_k$, we use \eqref{proof5.7} and Young's inequality and bound the sum by
\begin{equation*}
\begin{aligned}
& \sum_k \left([h + C\delta^{-1}]\|\nabla u\|_{L^2(\womega^k\setminus \ol\omega^k)}^2 + C\delta^{-1}\|f\|^2_{L^2(\omega_k)}+Ch^{-1} \int_{\omega^k}|u|^2|\nabla \varphi|^2 \right)\\
\le & \left(\frac{C}{\delta}+ h\right) \int_{Q_R\setminus F}|\nabla u|^2 + CR^2\int_{Q_R} |f|^2 + \frac{C}{h(R-r)^2} \int_{Q_R\setminus \ol Q_r} |u|^2,
\end{aligned}
\end{equation*}
where we use the relation $R^2>16^2>\delta^{-1}$, and the fact that $|\nabla \varphi|$ is supported in $Q_{R}\setminus Q_{r+3}$ and bounded by $C|R-r|^{-1}$. 

For the integral $\int_{Q_R} |f||\phi|$, we have
\begin{equation*}
 \int_{Q_R} |f||\phi |\le \|f\|_{L^2(Q_R)}\|u\|_{L^2(Q_R)} + \sum_k \|f\|_{L^2(\womega^k)}\|w_k\|_{L^2(\womega_k)},
\end{equation*} 
using Poincar\'e inequality for $w_k$ and \eqref{proof5.7}, we conclude that
$$
\int_{Q_R}|f| |\phi| \le CR^2\int_{Q_R}|f|^2 + \frac{C}{\delta}\int_{Q_R}|\nabla u|^2 + \frac{C}{(R-r)^2}\int_{Q_R}|u|^2.
$$

By plugging $\nabla (u\varphi^2) =\varphi^2 \nabla u + 2\varphi u\otimes \nabla \varphi$ in the integral on the left of \eqref{eq:caccio-1-1}, using the fact $\varphi=1$ in $Q_{r+3}$ and the estimate
$$
\left|\int_{Q_R \setminus F} 2\varphi  A \nabla u :(u\otimes \nabla \varphi) \right| \le h \int_{Q_R}|\nabla u|^2 + \frac{C}{h (R-r)^2} \int_{Q_R\setminus Q_r} |u|^2,
$$
which is easily checked as before, we obtain
\begin{equation*}
    \int_{Q_{r+3} \setminus
     F } |\nabla u|^2 \leq \frac{C}{h(R-r)^2} \int_{Q_R} |u|^2 \,dx + \left(h + \frac{C}{\delta}\right) \int_{Q_R} |\nabla u|^2\,dx + CR^2 \int_{Q_R}|f|^2.
\end{equation*}

Finally we conclude the proof by using the estimate
\begin{equation*}
    \int_{Q_r} |\nabla u |^2 \leq CR^2\int_{Q_R} |f|^2 + C\int_{Q_{r+3}\setminus F} |\nabla u|^2,
\end{equation*}
which is easily implied by Lemma \ref{lem:energyinlarged}.
\end{proof}

\begin{theorem}[Caccioppoli inequality for large $\delta$: interior estimate]\label{lem:cacciold+i} Let $\delta \in (1,\infty)$. Let $u\in H^1(Q_R;\R^m)$ be a weak solution of $\mathcal{L}_{\eps,\delta}(u) = f$ in $Q_R$ for some $R\ge 32\eps$. Then for any $\ell \ge 1$,
\begin{equation*}
\int_{Q_{R/2}} |\nabla u|^2 \le  C_{\ell} \left\{ \frac{1}{R^2} \int_{Q_R} |u|^2 + \left(\frac{\eps}{R}\right)^{2\ell} \int_{Q_R} |\nabla u|^2 + R^2 \int_{Q_R} |f|^2 \right\}.
\end{equation*}
\end{theorem}
\begin{proof} We modify the proof of Shen \cite[Theorem 3.3]{shen_large-scale_2021} and add $f$ to the right hand side. By rescaling we assume $\eps=1$. Given $R\ge 32$, and an integer $\ell \ge 1$, let $k$ be the largest integer so that $R2^{-k-1}\ge 8$. For $1\le i\le k$, let $r_i = R(1-2^{-i})$. Then we verify that, for all $1\le i\le k$,
$$
r_{i+1}\le 16, \quad \tfrac12 r_{i+1} \le r_i \le r_{i+1}-8.
$$
So, for any $h\in(0,1)$, we can apply the estimate in the previous Lemma \ref{lem:cacciold+} and obtain
\begin{equation*}
\int_{Q_{r_i}} |\nabla u|^2 \le \frac{C}{h (r_{i+1}-r_i)^2} \int_{Q_{r_{i+1}}}|u|^2 + (h + C\delta^{-1}) \int_{Q_{r_{i+1}}}|\nabla u|^2 + Cr_{i+1}^2\int_{Q_{r_{i+1}}}|f|^2.
\end{equation*}
By iteration, we finally get
\begin{equation*}
\begin{aligned}
\int_{Q_{R/2}} |\nabla u|^2 \le\ & \frac{C}{h} \sum_{i=1}^k \frac{(h+C\delta^{-1})^{i-1}}{(r_{i+1}-r_i)^2} \int_{Q_R}|u|^2 + C\sum_{i=1}^k (h+C\delta^{-1})^{i-1}r_{i+1}^2 \int_{Q_R}|f|^2  \\
&+ (h+C\delta^{-1})^k \int_{Q_R}|\nabla u|^2.
\end{aligned}
\end{equation*}
Plug in the value $r_{i+1}-r_i =2^{-i-1}R$, we have
\begin{equation*}
\begin{aligned}
\int_{Q_{R/2}} |\nabla u|^2 \le \ &\frac{4C}{hR^2(h+C\delta^{-1})} \sum_{i=1}^{\infty} (4h+4C\delta^{-1})^i \int_{Q_R}|u|^2 \\
&+ CR^2\sum_{i=1}^{\infty} (h+C\delta^{-1})^{i-1} \int_{Q_R}|f|^2  + (h+C\delta^{-1})^k \int_{Q_R}|\nabla u|^2.
\end{aligned}
\end{equation*}

We now fix $\ell \geq 1$, note that if $\delta \in [1,2^{2\ell+2}C)$, the conclusion directly follows from the classical setting. Hereafter we assume $\delta \geq 2^{2\ell+2}C$.
We choose $h = 2^{-2\ell-2}$ so that $4h+4C\delta^{-1} \le 2^{-2\ell+1}$. Hence, 
$\sum_{i=1}^{\infty} (4 h +4C\delta^{-1})^i \le 1$, and the estimate above yields
\begin{equation*}
\int_{Q_{R/2}} |\nabla u|^2  \le \frac{C_{\ell}}{R^2}\int_{Q_R}|u|^2 +C_{\ell}R^2\int_{Q_R}|f|^2 + (2^{-2\ell+1})^k \int_{Q_R}|\nabla u|^2.
\end{equation*}
Note that we fixed $R\ge 32$ first, and then choose $k$ so that $2^{k+4}\le R< 2^{k+5}$, so $(2^{-2\ell+1})^k  \approx  R^{-2\ell}$.
This completes the proof.
\end{proof}

By modifying the estimate a little bit, more precisely, by carrying the boundary terms in the interation steps (like the source term $f$ above), we can prove a boundary version of the result above.

\begin{theorem}[Caccioppoli inequality for large $\delta$: boundary estimate]\label{lem:boundaryCaccioppoli} Let $\delta \in (1,\infty)$. Let $u\in H^1(\mathbf{D}_R;\R^m)$ be a weak solution of $\mathcal{L}_{\eps,\delta}(u) = f$ in $\mathbf{D}_R$ and $u=g$ on $\Delta_R$ for some $R\ge 32\eps$. Then for any $\ell \ge 1$,
\begin{equation}
\label{eq:cacciobdrld}
\begin{aligned}
\int_{\mathbf{D}_{3R/2}} |\nabla u|^2 \le \ & \frac{C_\ell}{R^2} \int_{\mathbf{D}_{2R}} |u|^2 + C_{\ell}\left(\frac{\eps}{R}\right)^{2\ell} \int_{\mathbf{D}_{2R}} |\nabla u|^2  \\
&+ C_{\ell} \left(R^2 \int_{\mathbf{D}_{2R}} |f|^2 + \frac{1}{R^2}\|g\|_{L^\infty(\Delta_{2R})}^2 + \|\nabla_{\rm tan} g\|_{L^\infty(\Delta_{2R})}^2\right).
\end{aligned}
\end{equation}
\end{theorem}

We only outline the proof as it is very similar to the proof of item (ii) of Theorem \ref{lem:cacciosmalldel}. First, in the special setting of $g=0$, the estimate \eqref{eq:cacciobdrld} can be proved just like in the interior setting, following the steps of Lemma \ref{lem:energyinlarged}, Lemma \ref{lem:cacciold+} and Theorem \ref{lem:cacciold+i}. Indeed, given a cut-off function $\varphi$ that vanishes near $\partial \mathbf{D}_R \cap \Omega$, the function $\varphi^2 u$ belongs to $H^1_0(\mathbf{D}_R)$ as $u=0$ on $\Delta_R$, and hence it can be used to get a boundary version of Lemma \ref{lem:cacciold+} and then an ``iteration'' argument like the one used in Theorem \ref{lem:cacciold+i} yield the result. The more general setting can be reduced to the special setting using the idea in the proof of item (ii) of Theorem \ref{lem:cacciosmalldel}; we omit the details.

%%%%%%%%

%%%%%%%
%%%%%%%
\section{Uniform Lipschitz regularity}
\label{sec:prooflipschitz}

In this section we prove Theorem \ref{thm:globalLip} and hence establish the uniform (in $\eps$ and $\delta$) Lipschitz regularity of the solution of $\mathcal{L}_{\eps,\delta}(u) =f$ provided that for small $\delta$ the source term $f$ is small inside the obstacles. It is well known that via the routine covering and rescaling arguments, the global uniform Lipschitz regularity follows from a large-scale (down to $\eps$-scale) Lipschitz estimate and a local Lipschitz estimate (in the blown-up domain).

The main difference compared with Shen's work is that, Shen did not use the fact that for each fixed $\delta$, there is a homogenized problem and, moreover, the homogenized problem behaves quite well in the sense that it enjoys estimates that is uniform in $\delta$, for $\delta \in (1,\infty)$. For $\delta \ll 1$, although the homogenized problem is no longer a good approximation, a simple additional corrector is sufficient.  

\begin{theorem}[Large-scale boundary Lipschitz estimate]
		\label{lslipshitz}
		Assume that $A$ satisfies \eqref{ellipticc}\eqref{periodicc}, $f\in L^p(\Omega;\R^m)$ with $p\ge d$. For $\eps\in (0,1)$ and $\delta\in (0,\infty)$, let $f_{\eps,\delta}$ is defined by \eqref{eq:modf}. Let $u_{\varepsilon,\delta}$ be a weak solution of $\mathcal{L}_{\varepsilon,\delta}(u_{\varepsilon,\delta}) =f_{\eps,\delta}$ in $\mathbf{D}_{2}$ and $u=g$ in $\Delta_{2}$. Then there is $c>1$ so that for any $r \in [c \eps,1]$ we have
		\begin{equation}
		\label{eq:lslipschitz}
			\|\nabla u_{\varepsilon,\delta}\|_{\underline{L}^2 (\mathbf{D}_r)}\leq C\left(\| \nabla u_{\varepsilon,\delta} \|_{\underline{L}^2(\mathbf{D}_2)} + \|f\|_{\underline{L}^p(\mathbf{D}_2)} + \|g\|_{C^{1,\sigma}(\Delta_2)}\right).
		\end{equation}
\end{theorem}

Clearly, an interior version of the above large-scale Lipschitz estimate can be similarly formulated and proved. We focus only on the boundary setting because the interior version is relatively easier and because it is already available in Shen \cite[Theorem 7.3]{shen_large-scale_2021}. We point out that Shen's method is different and is not based on the approximation scheme.

The large-scale Lipschitz estimate above combined with a blow-up argument and a local Lipschitz estimate yield the following full-scale Lipschitz regularity. Again we only present and prove the boundary version.

\begin{theorem}[Boundary Lipschitz estimate]
	\label{fslipschitz}
	Assume the conditions of Theorem \ref{lslipshitz} and futher assume that $A$ satisfies \eqref{smoothc} and that $\Omega$ and $\Omega_\eps$ satisfies the geometric condition. Let $\mathbf{D}_{5/2}$ be a boundary cylinder of $\Omega$, and let $u_{\varepsilon,\delta} \in H^1(\mathbf{D}_{5/2};\R^m)$ bs a weak solution of $\mathcal{L}_{\varepsilon,\delta}(u_{\varepsilon,\delta}) =f_{\eps,\delta}$ in $\mathbf{D}_{5/2}$ with boundary condition $u_{\varepsilon,\delta} = g$ on $\Delta_{5/2}$. Then 
	% $u_{\varepsilon,\delta} \in H^1(B(x_0,r)\cap \Omega;\R^m)$ bs a weak solution of $\mathcal{L}_{\varepsilon,\delta}(u_{\varepsilon,\delta}) =f_{\eps,\delta}$ in $B(x_0,r)\cap \Omega$ with boundary condition $u_{\varepsilon,\delta} = g$ on $B(x_0,r)\cap \partial\Omega$, for some $x_0\in \partial\Omega$ and $0<r\le 2$. Then
	% \begin{equation}\label{eq:boundaryegularity}
	% 	\begin{aligned}
	% 	\|\nabla u_{\varepsilon,\delta} \|_{L^{\infty}(B(x_0,r/2)\cap \Omega)} &\leq C\Big\{ \|\nabla u_{\varepsilon,\delta} \|_{\underline{L}^2(B(x_0,r)\cap \Omega)} + r\|f\|_{\underline{L}^p(B(x_0,r)\cap \Omega)} \\
	% 	&\qquad + \|\nabla_{\rm tan} g\|_{L^\infty(B(x_0,r)\cap \Omega)} + r^\rho\lvert\nabla_{\rm tan} g\rvert_{C^{0,\rho}(B(x_0,r)\cap \partial\Omega)}\Big\}.
	% 	\end{aligned}
	% \end{equation}
	\begin{equation}\label{eq:boundaryegularity}
		\|\nabla u_{\varepsilon,\delta} \|_{L^{\infty}(\mathbf{D}_{1/2})} \leq C\Big\{ \|\nabla u_{\varepsilon,\delta} \|_{L^2(\mathbf{D}_{5/2})} + \|f\|_{L^p(\mathbf{D}_{5/2})} + \|\nabla_{\rm tan} g\|_{C^{0,\alpha}(\Delta_{5/2})}\Big\}.
  \end{equation}
\end{theorem}

%To prove this result we shall combine the large-scale Lipschitz regularity with the following local Lipschitz estimate

We only detail the proof of boundary Lipschitz regularity. Shen has essentially proved the interior Lipschitz regularity in \cite[Theorem 7.3]{shen_large-scale_2021}, but as far as we know, the uniform boundary Lipschitz regularity of \eqref{model} is new. We prove it using the approximation scheme which is now standard, and this method is easier to carry out for the interior Lipschitz estimate. We mention that, even for interior estimate, our method differs from \cite{shen_large-scale_2021} because there Shen used a method that relies on the special structure of homogeneous equations, combined with Green's function estimates. It seems that the approximation scheme has a large scope for applications, and what was missing for this scheme, before our present work, was the lack of the convergece rate result in \eqref{eq:thm1-error} for $\delta \gg 1$. 

We give the proof of the full-scale boundary Lipschitz estimate, assuming the large-scale Lipschitz estimate in Theorem \ref{lslipshitz}. 

% \begin{lemma}[local Lipschitz estimate] \label{locallip} Suppose that the assumptions of Theorem \ref{fslipschitz} hold. 
% \begin{itemize}
% 	\item[\upshape(1)](Boundary version) If $\mathbf{D}_{2r}$ is a boundary cylinder of $\Omega$ for some $0<r<1$, and $\mathcal{L}_{\eps,\delta}(u) = f$ in $\mathbf{D}_{2r}$ and $u=g$ in $\Delta_{2r}$, for some $0< r <1$, then we have
% \begin{equation}
% \label{eq:locallip}
% \|\nabla u\|_{L^\infty(\mathbf{D}_r)} \le C\left(\|\nabla u\|_{\underline{L}^2(\mathbf{D}_{2r})} + r\|f\|_{\underline{L}^p(\mathbf{D}_{2r})} + \|\nabla_{\rm tan} g\|_{L^\infty}\right)
% \end{equation}
% \item[\upshape(2)](Interior version) If $\mathcal{L}_{\eps,\delta}(u) = f$ in $Q_{2r}$, for some $0< r <1$, then we have
% \begin{equation}
% \label{eq:locallip-int}
% \|\nabla u\|_{L^\infty(\mathbf{D}_r)} \le C\left(\|\nabla u\|_{\underline{L}^2(\mathbf{D}_{2r})} + r\|f\|_{\underline{L}^p(\mathbf{D}_{2r})} + \|\nabla_{\rm tan} g\|_{L^\infty}\right)
% \end{equation}
% \end{lemma}

\begin{proof}[Proof of Theorem \ref{fslipschitz}] For simplicity we write $u$ for $u_{\eps,\delta}$. Let $c\ge 8$ in Theorem \ref{lslipshitz}, we first assume that $\eps < 1/c$. Then for any $x_0\in \Delta_{1/2}$, we have
\begin{equation}\label{eq:localgeo}
\mathbf{D}_{c\eps}(x_0) \subset \mathbf{D}_{2}(x_0) \subset \mathbf{D}_{5/2}.
\end{equation}
On the other hand, we assume that $\mathbf{D}_{2\eps}(x_0)$ is contained in $\Omega_\eps$ (that is, there is no obstacles inside $\mathbf{D}_{2\eps}(x_0)$). In view of the geometric condition (\textbf{GC}), this assumption is always valid if we further shrink the domain $\mathbf{D}_{2\eps}(x_0)$ here.  

Next we use the standard method by blow-uping the domain $D_{2\eps}$ via the mapping $x\mapsto y$ with $y(x) = \eps^{-1}(x-x_0)$, we denote the blown-up domain by $\widetilde{\mathbf{D}}_2$. Without loss of generality assume $x_0=0$ and consider the rescaled function $w(y) = \eps^{-1}u(\eps y)$, then a rescaling computation shows
\begin{equation*}
\mathcal{L}_{1,1}(w) = \widetilde{f} \quad\text{in}\quad \widetilde{\mathbf{D}}_2,\qquad\text{and}\quad 
w = \widetilde{g} \quad \text{in} \quad \widetilde{\Delta}_2,
\end{equation*}
with $\widetilde f(y)=\eps f(\eps y)$, $\widetilde g(y)=\eps^{-1}g(\eps y)$. Note that $\delta=1$ in the domain involved because there is no inclusion inside $\mathbf{D}_\eps(x_0)$. Moreover, $\widetilde{\Delta}_2$ is now the graph of the function $\eps^{-1}\psi(\eps y)$ where $\psi$ is the $C^{1,\alpha}$ function that defines $\Delta_{2\eps}(x_0)$. Note that $\widetilde{\Delta}_2$ satisfies the same regularity bounds of $\partial \Omega$. Then we can apply the standard local $C^{1,\alpha}$ regularity estimate for the operator $\mathcal{L}_{1,1}$ with H\"older continuous coefficients in $C^{1,\alpha}$ domains, and get
\begin{equation}
\label{eq:localreg}
\|\nabla w\|_{L^\infty(\widetilde{\mathbf{D}}_1)} \le C\left\{\|\nabla w\|_{L^2(\widetilde{\mathbf{D}}_2)} + \|\widetilde{f}\|_{L^p(\widetilde{\mathbf{D}}_2)} + \|\nabla_{\rm tan} \widetilde{g}\|_{C^{0,\alpha}(\widetilde{\Delta}_{2})}\right\}.
\end{equation}
By the change of variable $y=\eps x$, we get
\begin{equation}
\label{eq:localscaled}
\begin{aligned}
\|\nabla u\|_{L^\infty(\mathbf{D}_\eps(x_0))} &\le C\Big\{\|\nabla u\|_{\underline{L}^2(\mathbf{D}_{2\eps}(x_0))} + \eps\|f\|_{\underline{L}^p(\mathbf{D}_{2\eps}(x_0))} \\
&\qquad + \|\nabla_{\rm tan}g\|_{L^\infty(\Delta_{2\eps}(x_0))} +\eps^{\alpha} \lvert\nabla_{\rm tan}g\rvert_{C^{0,\alpha}(\Delta_{2\eps}(x_0))} \Big\}\\
&\le C\Big\{\|\nabla u\|_{\underline{L}^2(\mathbf{D}_{c\eps}(x_0))} + \|f\|_{{L}^p(\mathbf{D}_{2}(x_0))} +  \|\nabla_{\rm tan}g\|_{C^{0,\alpha}(\Delta_{2}(x_0))} \Big\}\\
&\le C\Big\{\|\nabla u\|_{L^2(\mathbf{D}_{2}(x_0))} + \|f\|_{{L}^p(\mathbf{D}_{2}(x_0))} +  \|\nabla_{\rm tan}g\|_{C^{0,\alpha}(\Delta_{2}(x_0))} \Big\}.
\end{aligned}
\end{equation}
Here to get the second inequality we used the fact that $p>d$ and $\eps < 1$. The third inequality follows from Theorem \ref{lslipshitz}. In view of \eqref{eq:localgeo}, we have bound $|\nabla u(x)|$ by the right hand side of \eqref{eq:boundaryegularity} for all $x \in \mathbf{D}_{1/2}$ in the 'boundary layer', i.e., those can be covered by $\mathbf{D}_\eps(x_0)$ for some $x_0\in \Delta_{1/2}$.

For any other point $x$ in $\mathbf{D}_{1/2}$, we can find some $Y^{\mathbf{n}}_{\varepsilon}$ such that $x\in Y^{\mathbf{n}}_{\varepsilon}$, in the blown-up domain, the operator is $\mathcal{L}_{1,\delta}$, so we need a local regularity estimate that is uniform for $\delta \in (0,\infty)$: for $w\in H^1(Q_1;\R^m)$ solving
$$
\mathcal{L}_{1,\delta}(w) = \widetilde f_\delta \quad \text{in}\quad Q_1,
$$
where  $\widetilde{f}_\delta$ is modifed to be $\delta \widetilde{f}$ in the inclusions for $\delta\in (0,1)$. Then for any $\delta \in (0,\infty)$, one has\begin{equation}\label{localestimate}
\|\nabla  w\|_{L^\infty(Q_{1/2})} \le C\left\{\|\nabla  w\|_{L^2(Q_1)} + \|\widetilde f_{\delta}\|_{L^p(Q_1)}\right\}.
\end{equation}
It plays the role of \eqref{eq:localreg}. Fortunately, this is proved by Shen; see \cite[Lemma 5.1]{shen_large-scale_2021} or equations (5.9) and (5.10) there. Hence, the proof of Theorem \ref{fslipschitz} is done by \eqref{localestimate} and the large-scale interior Lipschitz estimate for $\eps < 1/c$.

Finally, for $\eps\ge 1/c$, the tensor $A(x/\eps)$ is uniformly (in $\eps$) H\"older continous, and the result boils down to the local regularity estimates mentioned above. 
\end{proof}

The interior version of the full-scale Lipschitz estimate can similarly be proved. The proof of Theorem \ref{elliptic_regularity_estimate} is then complete by the usual covering argument. Hence it remains only to prove the large-scale Lipschitz regularity. 

%%%%%%%%%%%%%
% large-scale Lipschitz regularity
\subsection{Proof of the large-scale Lipschitz estimate}

We focus on the boundary version, i.e., the proof of Theorem \ref{lslipshitz}. Our approach is the now more or less standard approximation scheme. Although the approach originates from the work of Armstrong and Smart, we refer to the monograph \cite[Chapter 6, section 5]{shen_periodic_2018} for the improved and standardized steps which we follow. 

The first step is the following approximation result. 
\begin{lemma}\label{interlem1}
	Suppose that $u_{\varepsilon,\delta}$ is a weak solution of $\mathcal{L}_{\varepsilon,\delta}(u_{\varepsilon,\delta}) =f_{\eps,\delta}$ in $\mathbf{D}_{2r}$ with $u_{\varepsilon,\delta} = g $ on $\Delta_{2r}$ for some $r\ge 4\varepsilon$. There exists $v\in H^1(\mathbf{D}_r)$ such that $\widehat{\mathcal{L}}_{\delta}(v)= f_{\varepsilon,\delta} $ in $\mathbf{D}_r$ with $v=g$ on $\Delta_r$, and
	\begin{equation}
	\label{eq:L2approx}
	\begin{aligned}
		\| u_{\varepsilon,\delta} - v \|_{\underline{L}^2(\mathbf{D}_r)} \leq \ &C \left( \frac{\varepsilon}{r} \right)^{1/2} \left\{\| u_{\varepsilon,\delta} \|_{\underline{L}^2(\mathbf{D}_{2r})} + r^2\|f\|_{\underline{L}^2(\mathbf{D}_{2r})} + \|g\|_{L^\infty(\Delta_{2r})} + r\|\nabla_{\mathrm{tan}}g\|_{L^\infty(\mathbf{D}_{2r})} \right\} \\
		& + C \left( \frac{\varepsilon}{r}\right)^{3/2} r\| \nabla u_{\varepsilon,\delta} \|_{\underline{L}^2(\mathbf{D}_{2r})}.
		\end{aligned}
	\end{equation}
\end{lemma}
\begin{proof}
	By rescaling we may assume $r=1$. For notational simplicity we write $u$ for $u_{\eps,\delta}$. For $\delta > 1$, Let $R=1$ and $\ell=1$ in Theorem \ref{lem:boundaryCaccioppoli}, we get
	\begin{equation*}
	%\label{t13m2-1}
	\begin{aligned}
		\int_{\mathbf{D}_{3/2}} |\nabla u|^2\,dx \leq \ &C \left(\int_{\mathbf{D}_2} |u|^2\,dx + \int_{\mathbf{D}_2} |f|^2\,dx + \|g\|_{L^\infty(\Delta_2)}^2 + \|\nabla_{\rm tan}g\|_{L^2(\Delta_2)}^2\right.\\
		&  + \left.\varepsilon^2 \int_{\mathbf{D}_2} |\nabla u|^2\,dx\right).
		\end{aligned}
	\end{equation*}
	For $\delta\in (0,1]$, we use \eqref{eq:cacciobdrsd} instead, and the above still holds (even without the last $O(\eps^2)$ term). By the co-area formula and arguing by contradiction, we can find some $t\in [5/4,3/2]$ for which 
	\begin{equation}\label{t13m2}
		%\int_{\partial D_t\setminus \Delta_2} |\nabla u|^2\,d\sigma \leq C E(u),%\int_{\mathbf{D}_3} |u_{\varepsilon,\delta}|^2\,dx + C\varepsilon^2 \int_{\mathbf{D}_3} |\nabla u_{\varepsilon,\delta}|^2\,dx .
		\|\nabla u\|_{L^2(\partial \mathbf{D}_t\setminus \Delta_2)} \le C\left(\|u\|_{L^2(\mathbf{D}_2)} + \|f\|_{L^2(\mathbf{D}_t)} + \|g\|_{L^\infty(\Delta_2)} + \|\nabla_{\rm tan}g\|_{L^\infty(\Delta_2)} + \eps \|\nabla u\|_{L^2(\mathbf{D}_2)}\right).
	\end{equation}
	Let $h = u\rvert_{\partial\mathbf{D}_t\setminus \Delta_2}$ and $h=g$ in $\Delta_2$.  Let $v \in H^1(\mathbf{D}_t)$ solve
	\begin{equation*}
	 \widehat{\mathcal{L}}_\delta(v) = f_{\eps,\delta} \quad  \text{in }\; \mathbf{D}_t, \qquad v = h \quad \text{on }\, \partial \mathbf{D}_t.
	 \end{equation*}
	% We modify the coefficients of $\mathcal{L}_{\eps,\delta}$ inside $\mathbf{D}_t$ by removing the obstacles that has non-empty intersection with $\mathbf{D}_t\setminus \mathbf{D}_{t-2\eps}$, and let $\widetilde{\mathcal{L}}_{\eps,\delta}$ be the corresponding differential operator. Then for $\widetilde{\mathcal{L}}_{\eps,\delta}$ the geometric condition is satisfied. Let $w$ solve
	% \begin{equation*}
	%  \widetilde{\mathcal{L}}_{\eps,\delta} (w) = f_{\eps,\delta} \; \text{in }\; \mathbf{D}_t, \qquad w = h \;\text{on }\, \partial \mathbf{D}_t.
	%  \end{equation*}  
Then by Theorem \ref{thm:suboptimal convergence rate} and  
 Remark \ref{ratefortype 1} (note that $\mathbf{D}_t$ may be a Type-I domain) we have
	\begin{equation*}
	\begin{aligned}
		\Big\|  u - v - \delta^{-1}\mathcal{L}^{-1}_{D_{\eps}}(f_{\eps,\delta}) -  \varepsilon \chi_{\delta}\left(\frac{x}{\varepsilon}\right) S_{\varepsilon} \big(\eta_{\varepsilon}\nabla v \big)  \Big\|_{H^1(\mathbf{D}_{t-1/4})} &\leq C\varepsilon^{1/2}\left(\|f\|_{L^2(\mathbf{D}_t)} + \|\nabla_{\mathrm{tan}} h\|_{L^2(\partial \mathbf{D}_t)}\right)\\
		& \le C\eps^{1/2} E(u),
		\end{aligned}
	\end{equation*}
	where $E(u)$ denote the quantity inside the parathesis in the right hand side of \eqref{t13m2}. 
	It is easy to check the following estimates (see Remark \ref{rem:fterm} and Lemma \ref{appen_auxi_0})
	\begin{equation*}
	\begin{aligned}
	\delta^{-1}\|\mathcal{L}_{D_\eps}^{-1}(f_{\eps,\delta})\|_{L^2(\mathbf{D}_t)} &\le C\eps^2\|f\|_{L^2(\mathbf{D}_2)},\\
		\Big\|   \chi_{\delta}\left(\frac{x}{\varepsilon}\right) S_{\varepsilon} \big(\eta_{\varepsilon}\nabla v \big)   \Big\|_{L^2(\mathbf{D}_t)} &\leq C\| \chi_{\delta} \|_{L^2(Y)} \| \nabla v \|_{L^2(\mathbf{D}_t)} \\
		&\leq C \| \chi_{\delta} \|_{L^2(Y)} \| \left(\|f\|_{L^2(\mathbf{D}_t)} + \|\nabla_{\mathrm{tan}} h \|_{L^2(\partial \mathbf{D}_t)}\right). %\le CE(u).
	\end{aligned}
	\end{equation*}
	Using the estimate in \eqref{t13m2} we hence obtain
	\begin{equation}\label{for1}
		\begin{aligned}
			\| u - v\|_{L^2(\mathbf{D}_1)} & \leq C (\varepsilon^{1/2} + \varepsilon \| \chi_{\delta} \|_{L^2(Y)} ) E(u) \\
			& \leq C\varepsilon^{1/2}\left(\| u\|_{L^2(\mathbf{D}_2)} + \|f\|_{L^2(\mathbf{D}_2)} + \|\nabla_{\rm \tan} g\|_{L^\infty(\Delta_2)})\right) + C\varepsilon^{3/2} \| \nabla u \|_{L^2(\mathbf{D}_2)}.
		\end{aligned}
	\end{equation}
	The proof is hence completed.
\end{proof}

Note that for $\delta\in (0,1)$ (with the modified $f_{\eps,\delta}=\delta f$ inside $D_\eps$), we do not need the last term  on the right hand side of \eqref{eq:L2approx} and the proof of Theorem \ref{lslipshitz} reduces to the standard procedure; see e.g.\,\cite[Section 6.4]{shen_periodic_2018}. The proof below is a modification of the standard method to handel the extra gradient term.

Let $f$ and $g$ be as in Theorem \ref{lslipshitz}. For $t\le 2$, $p>d$, and for functions $w \in H^1(\mathbf{D}_2)$, we define the following notations:
\begin{equation}
\label{eq:Hdef}
\begin{aligned}
H[w](t) = \ &\frac{1}{t} \inf_{\substack{E\in \R^{m\times d}\\ q\in \R^m}} \Bigg\{\left(\fint_{\mathbf{D}_t} |w-Ex-q|^2\right)^{\frac12} + t^2\left(\fint_{\mathbf{D}_t} |f|^p\right)^{\frac1p} + \|g-Ex-q\|_{L^\infty(\Delta_t)} \\
& + t\|\nabla_{\rm tan}(g-Ex)\|_{L^\infty(\Delta_t)}+ t^{1+\rho}\lvert\nabla_{\rm tan} (g-Ex)\rvert_{C^{0,\rho}(\Delta_t)}\Bigg\},
\end{aligned}
\end{equation}
where $\rho$ is a number in $(0,\alpha)$ and $\lvert\cdot\rvert_{C^{0,\rho}}$ denotes the usual H\"older semi-norm. 
% \begin{equation}
% \label{eq:Hdef}
% \begin{aligned}
% \Phi[w](t) = &\frac{1}{t} \inf_{q\in \R^m} \Big\{\left(\fint_{\mathbf{D}_{2t}} |w-q|^2\right)^{\frac12} + t^2\left(\fint_{\mathbf{D}_{2t}} |f_{\eps,\delta}|^p\right)^{\frac1p}\\
% &\qquad + \|g-q\|_{L^\infty(\Delta_{2t})} + r\|\nabla_{\rm tan} g\|_{L^\infty(\Delta_{2t})}\Big\}.
% \end{aligned}
% \end{equation} 
We also defined $h[w](t)$ to be the norm $|E_t|$ of a matrix $E_t$ for which
\begin{equation}
\label{eq:hdef}
\begin{aligned}
H[w](t) = &\frac{1}{t} \inf_{q\in \R^m} \Bigg\{\left(\fint_{\mathbf{D}_t} |w-E_t x-q|^2\right)^{\frac12} + t^2\left(\fint_{\mathbf{D}_t} |f|^p\right)^{\frac1p} + \|g-E_t x-q\|_{L^\infty(\Delta_t)} \\
&\qquad + t\|\nabla_{\rm tan} (g-E_t x)\|_{L^\infty(\Delta_t)} + t^{1+\rho}\vert\nabla_{\rm tan} (g-E_t x)\rvert_{C^\rho(\Delta_t)}\Bigg\}.
\end{aligned}
\end{equation}

\begin{lemma}\label{campahomogenized}
	Suppose that $\widehat{\mathcal{L}}_{\delta}(w) = f$ in $\mathbf{D}_r$ with $w= g $ on $\Delta_r$ for some $r>0$. Then there exists a universal constant $\theta \in (0,1/4)$ (independent of $r,f$ or $g$) such that 
	\begin{equation}
		H[w](\theta r) \leq \frac{1}{2} H[w](r).
	\end{equation}
\end{lemma}
\begin{proof}
	We have shown in Theorem \ref{uniform elliptic homogenized tensor} that the homogenized operator $\widehat{\mathcal{L}}_\delta$ has constant elliptic tensor with uniform in $\delta$ ellipticity bounds. The result of this lemma then follows from $C^{1,\alpha}$ regularity estimate of constant coefficients. We refer to \cite[Lemma 7.1]{MR3541853} for the proof with $f=0$. The setting with $f\in L^p$ and $p>d$ is the same.
\end{proof}

% For $u \in H^1(\mathbf{D}_r)$ with $u = f $ on $\partial \mathbf{D}_r$, let
% \begin{equation}
% 	\begin{aligned}
% 		H(r;u) = \frac{1}{r} \inf_{P\ linear}\Big\{ & \| u - P \|_{\underline{L}^2(\mathbf{D}_r)} + \| f-P\|_{L^{\infty}(\Delta_r)}  \\
% 		&+ r\| \nabla_{\mathrm{tan}} (f-P) \|_{L^{\infty}(\Delta_r)}  + r^{1+\beta} \| \nabla_{\mathrm{tan}} (f-P) \|_{C^{0,\beta}(\Delta_r)} \Big\},
% 	\end{aligned}
% \end{equation}
% where $\beta = \alpha /2$ and the infimum is taken over all linear functions $P$. Let $h(r;u) := \| M_r\|$, where $M_r $ is a matrix such that
% \begin{equation}
% 	\begin{aligned}
% 		H(r;u) = \frac{1}{r} \inf_{q \in \mathbb{R}^m}\Big\{ & \| u - M_r x - q \|_{\underline{L}^2(\mathbf{D}_r)} + \| f-M_r x-q\|_{L^{\infty}(\Delta_r)}  \\
% 		&+ r\| \nabla_{\mathrm{tan}} (f-M_r x) \|_{L^{\infty}(\Delta_r)}  + r^{1+\beta} \| \nabla_{\mathrm{tan}} (f-M_r x) \|_{C^{0,\beta}(\Delta_r)} \Big\}.
% 	\end{aligned}
% \end{equation}

\begin{lemma}\label{lem:inter} Let $\theta$ be determined in Lemma \ref{campahomogenized} and $c>0$ be a number so that $c\theta \ge 4$. Suppose $u_{\varepsilon,\delta}$ is a weak solution of $\mathcal{L}_{\varepsilon,\delta}(u_{\varepsilon,\delta}) =f_{\eps,\delta}$ in $\mathbf{D}_{2r}$ with $u_{\varepsilon,\delta} = g $ on $\Delta_{2r}$ for some $r\ge c\varepsilon$, then
	\begin{equation}\label{inter:lem}
		\begin{aligned}
			H[ u_{\varepsilon,\delta}](\theta r) & \leq \frac{1}{2} H[u_{\varepsilon,\delta}](r) \\
			&\ \ \ + C  \left( \frac{\varepsilon}{r} \right)^{1/2} \Big( H[u_{\varepsilon,\delta}](2r) +h[u_{\varepsilon,\delta}](2r) \Big)+ C \left( \frac{\varepsilon}{r}\right)^{3/2}  \| \nabla u_{\varepsilon,\delta} \|_{\underline{L}^2(\mathbf{D}_{2r})}.
		\end{aligned}
	\end{equation}
\end{lemma}

\begin{proof}
	Let $v$ be the function chosen in Lemma \ref{interlem1}. By Lemma \ref{campahomogenized} 
	\begin{equation}\label{Hfor constant}
		H[v](\theta r) \leq \frac{1}{2} H[v](r).
	\end{equation}
	Let $c$ be a positive number so that $\theta c\ge 4$, then by \eqref{Hfor constant}, we have
	\begin{equation*}
		\begin{aligned}
			H[u_{\varepsilon,\delta}](\theta r)  &\leq  H[v](\theta r) + \frac{\| u_{\varepsilon,\delta} - v \|_{\underline{L}^2(\mathbf{D}_{\theta r})}}{\theta r} \\
			&\leq \frac{1}{2} H[u_{\varepsilon,\delta}](r) + \frac1{2r} \|u_{\eps,\delta}-v\|_{\underline{L}^2(\mathbf{D}_r)} + \frac{1}{\theta r}  \|u_{\eps,\delta}-v\|_{\underline{L}^2(\mathbf{D}_{\theta r})}.
			\end{aligned}
			\end{equation*}
By Lemma \ref{interlem1}, we obtain
			\begin{equation*}
\begin{aligned}
			H[u_{\eps,\delta}](\theta r) \le  \ &\frac{1}{2} H[u_{\varepsilon,\delta}](r) + C \left( \frac{\varepsilon}{r}\right)^{3/2} \| \nabla u_{\varepsilon,\delta} \|_{\underline{L}^2(\mathbf{D}_{2r})}\\
			& + C \left(\frac{\varepsilon}{r} \right)^{1/2}\frac1r \left\{\| u_{\varepsilon,\delta} \|_{\underline{L}^2(\mathbf{D}_{2r})} + r^2\|f\|_{\underline{L}^2(\mathbf{D}_{2r})} + \|g\|_{L^\infty(\Delta_{2r})} + r\|\nabla_{\mathrm{tan}}g\|_{L^\infty(\mathbf{D}_{2r})} \right\}.
		\end{aligned}
	\end{equation*}
	Since adding a constant to $u_{\eps,\delta}$ and $g$ does not change the setting, we in fact get
\begin{equation*}
\begin{aligned}
			H[u_{\eps,\delta}](\theta r) &\le \frac{1}{2} H[u_{\varepsilon,\delta}](r) + C \left( \frac{\varepsilon}{r}\right)^{3/2} \| \nabla u_{\varepsilon,\delta} \|_{\underline{L}^2(\mathbf{D}_{2r})}   \\
			 & \ \ \ \, + C \left(\frac{\varepsilon}{r} \right)^{1/2}\inf_{q\in \R^m} \frac1r \Big\{\| u_{\varepsilon,\delta} -q \|_{\underline{L}^2(\mathbf{D}_{2r})} +r^2\|f\|_{\underline{L}^p(\mathbf{D}_{2r})} \\
   &\ \ \ \ \ \ \ \ \ \ \ \ \ \ \ \ \ \ \ \ \ \ \ \ \ \ \ \ \ \ \,  + \|g-q\|_{L^\infty(\Delta_{2r})} + r\|\nabla_{\mathrm{tan}}(g-q)\|_{L^\infty(\mathbf{D}_{2r})} \Big\}.
		\end{aligned}
	\end{equation*}
Then the desired estimate follows because the infimum in the last line is clearly dominated by
	\begin{equation*}
		H[u_{\eps,\delta}](2r) + h[u_{\eps,\delta}](2r).
	\end{equation*}
The proof is complete.
\end{proof}

\begin{proof}[Proof of Theorem \ref{lslipshitz}] For notational simplicity we will write $u$ for $u_{\eps,\delta}$, $H(r) = H[u](r)$, and $h(r) = h[u](r)$. Let $c$ be a number so that $c\ge 8$ and $c\theta \ge 4$. We consider below $r\ge c\eps$. We divide the proof to several steps.

\textit{Step 1}. Let $\ell=1$ in Theorem \ref{lem:boundaryCaccioppoli}, we have
	\begin{equation*}
			\| \nabla u \|_{\underline{L}^2(\mathbf{D}_r)} \leq C\left\{\frac{\varepsilon}{r} \| \nabla u\|_{\underline{L}^2(\mathbf{D}_{2r})} + \frac{1}{r}\| u\|_{\underline{L}^2(\mathbf{D}_{2r}) } + r\|f\|_{\underline{L}^2(\mathbf{D}_{2r})} + \frac{1}{r}\|g\|_{L^\infty(\Delta_{2r})} + \|\nabla_{\rm tan}g\|_{L^\infty(\Delta_{2r})}\right\}.
			\end{equation*}
			Because we can insert any constant $q$ simultaneously to $u$ and $g$ above, and note that $\|f\|_{\underline{L}^2(\mathbf{D}_{2r})} \le \|f\|_{\underline{L}^p(\mathbf{D}_{2r})}$ for $p>2$, we actually get
			\begin{equation*}
			\|\nabla u\|_{\underline{L}^2(\mathbf{D}_r)} \leq \frac{C\varepsilon}{r} \| \nabla u\|_{\underline{L}^2(\mathbf{D}_{2r})} + C\Phi(2r), 
			\end{equation*}
			where
			$$
			\Phi(2r) = \inf_{q\in \R^m} \frac1r \left\{\|u-q\|_{\underline{L}^2(\mathbf{D}_{2r})} + r^2\|f\|_{\underline{L}^p(\mathbf{D}_{2r})} + \|g-q\|_{L^\infty(\Delta_{2r})} + r\|\nabla_{\rm tan} g\|_{L^\infty(\Delta_{2r})}\right\}.
			$$
			From the definitions of $H$ and $h$ we verify $\Phi(2r) \le C(H(2r)+h(2r))$, and hence we conclude that
			\begin{equation*}
			\|\nabla u\|_{\underline{L}^2(\mathbf{D}_r)} \leq C\left(H(2r)+h(2r)\right)+\frac{C\varepsilon}{r}  \| \nabla u \|_{\underline{L}^2(\mathbf{D}_{2r})}.
	\end{equation*}
	Iterating the inequality above, for the largest integer $M\in \N^*$ such that $2^M r\leq 2$, we get
	\begin{equation*}
		\begin{aligned}
			\| \nabla u \|_{\underline{L}^2(\mathbf{D}_r)} & \leq \sum_{n=1}^M  \frac{C^n \varepsilon^{n-1}}{2^{(n-1)(n-2)/2}r^{n-1}} \Big( H(2^nr) + h(2^nr)  \Big)  \\
			&  \ \  \ + \frac{C^M\varepsilon^M}{2^{M(M-1)/2} r^M} \| \nabla u \|_{\underline{L}^2(\mathbf{D}_{2^M r})}.
		\end{aligned}
	\end{equation*}
	Since for any $\varepsilon>0$ and $r>\varepsilon$, 
	\begin{equation*}
		\sup_{n \in \mathbb{N}^*} 2^n \cdot \frac{C^n \varepsilon^{n-1}}{2^{(n-1)(n-2)/2}r^{n-1}} +
		\sup_{M \in \mathbb{N}^*} \frac{C^M\varepsilon^M}{2^{M(M-1)/2} r^M} <\infty,
	\end{equation*}
	we obtain
\begin{equation}\label{410}
%\begin{aligned}
		\| \nabla u\|_{\underline{L}^2(\mathbf{D}_r)}  %&\leq 
		\le C\sum_{n=1}^M \frac{H(2^nr) + h(2^nr)}{2^n} +%C \| \nabla u\|_{\underline{L}^2(\mathbf{D}_{2^Mr})}\\
		%&\leq C\sum_{n=1}^{ \infty} \frac{H(2^nr) + h(2^nr)}{3^n} +
		C  \| \nabla u \|_{\underline{L}^2(\mathbf{D}_2)}.
	%\end{aligned}
	\end{equation}

\textit{Step 2}. As is standard, from the definitions \eqref{eq:Hdef} and \eqref{eq:hdef} it is not hard to verify that
	\begin{equation}\label{propertya}
	\begin{aligned}
		&H(t) \leq CH(2r) \qquad && \forall t \in [r,2r],\\
	%\end{equation}
	%and
	%\begin{equation}\label{propertyb}
		&|h(t) -h(s) | \leq CH(2r) \qquad &&\forall t,s \in [r,2r].
		\end{aligned}
	\end{equation}
	% By \eqref{propertyb} we have $h(r) \leq h(3r)+CH(3r)$, so for any $\varepsilon<a <1/3$, 
	% \begin{equation*}
	% 	\begin{aligned}
	% 		\int_a^{1/3} \frac{h(r)}{r}\,dr &\leq \int_a^{1/3} \frac{h(3r)}{r}\,dr +C\int_a^{1/3} \frac{H(3r)}{r}\,dr  \\
	% 		&\leq \int_{3a}^1 \frac{h(r)}{r}\,dr + C\int_{3a}^1 \frac{H(r)}{r}\,dr ,
	% 	\end{aligned}
	% \end{equation*}
	% which implies that 
	% \begin{equation}\label{hrint}
	% 	\int_a^{3a} \frac{h(r)}{r} \,dr \leq \int_{1/3}^1 \frac{h(r)}{r}\,dr +C\int_{3a}^1 \frac{H(r)}{r}\,dr .
	% \end{equation}
	% Still by \eqref{propertyb} we have $h(a ) \leq h(r )+ CH(3a)$ for $a\leq r\leq 3a$, this combine with \eqref{propertya} yield
	% \begin{equation}\label{415}
	% 	\begin{aligned}
	% 		h(a)+H(a) &\leq h(r) + CH(3a)\\
	% 		& \leq C\int_a^{3a} \frac{h(r)}{r}\,dr +CH(3a) \\
	% 		& \leq C\int_{1/3}^1 \frac{h(r)}{r}\,dr +C\int_{3a}^1 \frac{H(r)}{r}\,dr +CH(3a) \\ 
	% 		& \leq CH(1)+Ch(1)+ C\int_{3a}^1 \frac{H(r)}{r}\,dr +CH(3a),
	% 	\end{aligned}
	% \end{equation}
	% where the last second inequality uses \eqref{hrint}. 
	
	% We extend the domain of $H(r)$ to $[0,3]$, by defining $H(r) = H(1)$ for $r\in [1,3]$, then \eqref{propertya} still holds, and we have
	% \begin{equation}\label{416}
	% 	(\log 3) H(3a) \leq \int_a^{3a} \frac{\max_{r\leq t\leq 3r}H(t)}{r}\,dr \leq \int_a^{3a} \frac{H(3r)}{r}\,dr \leq \int_{3a}^1 \frac{H(r)}{r}\,dr+ CH(1).
	% \end{equation}
	% By \eqref{415} and \eqref{416} we obtain
	% We may reveal the above computation if we want.
	Moreover, those inequalities lead to
	\begin{equation}\label{417}
		h(a)+H(a) \leq C\left\{ H(2)+h(2) + \int_a^2 \frac{H(s)}{s}\,ds\right\}, \quad \forall a\in [\eps,2].
	\end{equation}
	% {\color{blue} Here is the detailed proof. By the inequality above we get
	% $$
	% h(r)\le h(2r) + CH(2r), \quad \forall r\in [\eps,1].
	% $$
	% Hence, we have for any $a\in [\eps,1/2]$,
	% $$
	% \int_a^{1} h(r)/r \,dr \le \int_a^1 h(2r)/r + CH(2r)/r\,dr = \int_{2a}^2 h(r)/r + C\int_a^{1} H(2r)/r\,dr.
	% $$
	% This implies (note that $2a\le 1$)
	% $$
	% \int_a^{2a} h(r)/r \le \int_1^2 h(r)/r + C\int_a^1 H(2r)/r\,dr.
	% $$
	% For $r\in [1,2]$ we use that $h(r) \le h(2)+H(2)$ to get
	% $$
	% \int_1^2 h(r)/r\,dr \le \int_1^2 h(r)\,dr \le CH(2) + Ch(2) + \int_a^1 H(2r)/r.
	% $$
	% and hence for any $a\in [\eps,1/2]$ we have
	% $$
	% H(a) + h(a) \le C\left\{H(2a) \right}
	% $$}
	We refer to \cite[page 157]{shen_periodic_2018} for the verification of the last estimate. We then get
	\begin{equation*}
		\begin{aligned}
			\sum_{n=1}^{M} \frac{H(2^nr) + h(2^nr)}{2^n} &\leq C\sum_{n=1}^M \frac{1}{2^n} \left[H(2)+h(2) + \int_{2^{n}r}^2 \frac{H(t)}{t}\,dt \right]\\
			&\leq C\left( H(2)+ h(2)+  \int_{r}^2 \frac{H(t)}{t}\,dt\right).
		\end{aligned}
	\end{equation*}
By taking $E=0$ and $q=0$ in \eqref{eq:Hdef} we have
\begin{equation}
\label{eq:H2bdd}
H(2) \le \|u\|_{\underline{L}^2(\mathbf{D}_2)} + \|f\|_{\underline{L}^p(\mathbf{D}_2)} + \|g\|_{C^{1,\rho}(\Delta_2)},
\end{equation}
and since $E_2x + q$ is harmonic for any $q\in \R^m$, we also have 
$$
h(2) \le C\inf_{q\in \R^m} \|E_1 x+q\|_{\underline{L}^2(\mathbf{D}_2)} \le C\{H(2) + \|u\|_{\underline{L}^2(\mathbf{D}_2)}\}.
$$
We hence get
$$
H(2) + h(2) \le C\left(\|\nabla u\|_{\underline{L}^2(\mathbf{D}_2)} + \|f\|_{\underline{L}^p(\mathbf{D}_2)} + \|\nabla_{\rm tan} g\|_{C^{0,\rho}(\Delta_2)}\right).
$$
In the last line we used the freedom to insert any constant $q$ to $u$ and $g$, and applied Poincar\'e's inequality. 

\emph{Step 3}: Combine the estimates above. We get, for $c\eps \le r\le 1$,
\begin{equation}\label{518}
		\| \nabla u\|_{\underline{L}^2(\mathbf{D}_r)}  \leq C\left(\|\nabla u\|_{\underline{L}^2(\mathbf{D}_2)} + \|f\|_{\underline{L}^p(\mathbf{D}_2)} + \|\nabla_{\rm tan} g\|_{C^{0,\rho}(\Delta_2)}\right) +C \int_r^2 \frac{H(t)}{t}\,dt.
\end{equation}

\textit{Step 4}. We estimate the integral $\int_r^2 \frac{H(t)}{t}\,dt$. By Lemma \ref{lem:inter}, \eqref{417} and  \eqref{518}, we get
	\begin{equation}
	\label{eq:step4}
			H(\theta r)  \leq \frac{1}{2}H(r) + C\Big(\frac{\varepsilon}{r} \Big)^{1/2} \Phi_2 +C \Big(\frac{\varepsilon}{r} \Big)^{1/2} \int_{r}^2  \frac{H(t)}{t}\,dt,
	\end{equation}
	for any $c\eps \le r \le 1/2$, where $\Phi_2 = \|\nabla u\|_{\underline{L}^2(\mathbf{D}_2)} + \|f\|_{\underline{L}^p(\mathbf{D}_2)} + \|\nabla_{\rm tan} g\|_{C^{0,\rho}(\Delta_2)}$. Let $c'>c$ be a large number so that 
	\begin{equation}
	\label{eq:Ctbc}
	\int_0^{1/c'} \sqrt{t} \le \frac{1}{4C}
	\end{equation}
where $C$ is the constant in front of the integral $\int_r^2 H(t)/t\,dt$ in \eqref{eq:step4}. We compute%For any fixed number $\gamma>1$, we integrate the above inequality on the interval $[\gamma \varepsilon,1]$ with the measure $dr/r$,
	\begin{equation}\label{lastly}
		\begin{aligned}
			\int_{c' \theta\varepsilon}^{2\theta } \frac{H(r)}{r}\,dr = \int_{c'\eps}^2 \frac{H(\theta r)}{r}\,dr &\leq \frac{1}{2} \int_{c'\varepsilon}^{1/2}  \frac{H( r)}{r}\,dr + \int_{1/2}^2 \frac{H(\theta r)}{r}\,dr  + C\Phi_2 \int_{c'\eps}^{1/2} \left(\frac{\eps}{r}\right)^{\frac12} \frac1r \,dr\\
			& \ \ \ + C \int_{c'\varepsilon}^{1/2} \left( \frac{\varepsilon}{r}\right)^{1/2}  \int_{r}^2 \frac{H(t)}{t}\,dt \,\frac{dr}{r}.
		\end{aligned}
	\end{equation}
	The second integral on the right hand side above can be bounded by $C_\theta H(2)$ where $C_\theta$ is a constant  only depending on $\theta$. For the third integral, we can bound using the fact that
	$$
	\int_{c' \varepsilon}^1  \left( \frac{\varepsilon}{r} \right)^{1/2}\frac{dr}{r} = \int_{\eps}^{1/c'} t^{-1/2}\,dt \le \frac{1}{4C}.
	$$
	For the last integral on the right hand side of \eqref{lastly}, by Fubini's Theorem, we have
	\begin{equation*}
		\begin{aligned}
			\int_{c'\varepsilon}^2 \left( \frac{\varepsilon}{r}\right)^{1/2}  \int_{r}^2 \frac{H(t)}{t}\,dt \,\frac{dr}{r}  & = \int_{c' \varepsilon}^2  \int_{c' \varepsilon}^t \left( \frac{\varepsilon}{r}\right)^{1/2} \frac{dr}{r}\frac{H(t)}{t}\,dt \\
			& = \int_{c' \varepsilon}^2 H(t) \int_{\varepsilon/t}^{1/c'} s^{-1/2}\,ds\, \frac{dt}{t} \\
			& \leq \int_0^{1/c'} s^{-1/2} \,ds\int_{c' \varepsilon}^2 \frac{H(t)}{t}\,dt \\
			& \leq \frac{1}{4C} \int_{c' \varepsilon}^2 \frac{H(t)}{t}\,dt.
		\end{aligned}
	\end{equation*}
	Apply those estimates above in \eqref{lastly}, we get
	\begin{equation*}
		\int_{c' \theta\varepsilon}^{2\theta } \frac{H( r)}{r}\,dr  \leq \frac{3}{4} \int_{c'\varepsilon}^2  \frac{H( r)}{r}\,dr + C\Phi_2 \le \frac{3}{4} \int_{c'\varepsilon}^{2\theta}  \frac{H( r)}{r}\,dr + C\Phi_2,
	\end{equation*}
	where the second inequality follows from \eqref{propertya} and \eqref{eq:H2bdd}. It follows that
	\begin{equation*}
		\int_{c'\theta \eps}^{2} \frac{H(t)}{t}\,dt \le  \int_{c'\theta \eps}^{2\theta} \frac{H(t)}{t}\,dt +  \int_{2\theta}^{2} \frac{H(t)}{t}\,dt \leq C \Phi_2.
	\end{equation*}
	By another iterative usage of \eqref{propertya} (up to finitely many steps independent of $\eps$), we get
$$
\int_r^2 \frac{H(t)}{t}\,dt \le C\Phi_2
$$
for all $r\in [c\eps,1]$. The desired estimate \eqref{eq:lslipschitz} follows from above and \eqref{518}.
\end{proof}

\appendix

\section{A weak form of layer potentials}\label{appen: weak layer potential}

 In this appendix, we construct the `weak' form of single layer potential and the NP operator, and explain how does the results in Section \ref{uniformcontrollcell} are still valid provided the H\"{o}lder continuity condition \eqref{smoothc} is removed.
 
Let $\psi \in H_0^{-1/2}(\partial \omega) $, we consider the following variational problem:
	\begin{equation}\label{def:variational}
	\int_Y A \nabla u \cdot \nabla v \,dy + \langle \psi , v \rangle_{-1/2,1/2}=0, \qquad \mathrm{for \ any}\ v \in H^1(\mathbb{T}^d),
\end{equation}
where $\langle \cdot,\cdot\rangle_{-1/2,1/2}$ denotes $H^{-1/2} - H^{1/2}$ pair. By Lax-Milgram theorem and Poincar\'{e} inequality, there exists a unique solution $u \in H^1(\mathbb{T}^d)$ such that $\int_Y u(y)\,dy=0$, we denote such $u$ by $\mathcal{S} \psi$.
It is clear that $\mathcal{L}\mathcal{S} \psi = 0 $ in $\mathbb{T}^d \setminus \partial \omega$ and
\begin{equation}\label{appen:identity}
    I = \left.\frac{\partial \mathcal{S}}{\partial \nu_A} \right|_+  - \left.\frac{\partial \mathcal{S}}{\partial \nu_A} \right|_-.
\end{equation}

We then define the boundary operator $K$ by 
\begin{equation}
	K: H \rightarrow H, \quad \psi \mapsto \left. \frac{\partial \mathcal{S}\psi}{\partial \nu_A} \right|_+ - \frac{1}{2} \psi,
\end{equation}
where $H$ is the Hilbert space $H_0^{-1/2}(\partial \omega)$ with a new inner product: for any $\phi, \psi \in H_0^{-1/2}(\partial \omega)$,
\begin{equation}
	\langle \phi,\psi \rangle_H : = \int_Y A \nabla \mathcal{S} \phi \cdot \nabla \mathcal{S}\psi\,dy.
\end{equation}
By definition of $K$ and \eqref{appen:identity}, we have
\begin{equation}
	\left. \frac{\partial \mathcal{S}\psi}{\partial \nu_A} \right|_+  = \left( \frac{1}{2}I +K \right) \psi, \qquad \mathrm{and} \qquad \left. \frac{\partial \mathcal{S}\psi}{\partial \nu_A} \right|_-  = \left( -\frac{1}{2}I +K \right) \psi.
\end{equation}

\begin{proposition}\label{spectral K}
	 $K$ is self-adjoint, and the spectrum of $K $ is contained in $(-1/2,1/2)$.
\end{proposition}
\begin{proof}
	\textit{Step 1}. We compute that
	\begin{equation*}
		\langle K \phi,\psi \rangle_H = \frac{1}{2} \int_{\omega} A \nabla \mathcal{S}\phi \cdot \nabla \mathcal{S}\psi \,dy -\frac{1}{2} \int_{Y\setminus \omega} A \nabla \mathcal{S}\phi \cdot \nabla \mathcal{S}\psi \,dy,
	\end{equation*}
this shows that $K$ is self-adjoint, in particular, the spectrum is real and for any $\lambda \in \mathbb{R}$, 
\begin{equation}
	\langle (\lambda I -K ) \phi,\phi \rangle_H = \left( \lambda -\frac{1}{2} \right) \int_{\omega} A |\nabla \mathcal{S}\phi |^2  \,dy + \left( \lambda + \frac{1}{2} \right) \int_{Y\setminus \omega} A |\nabla \mathcal{S}\phi|^2\,dy.
\end{equation}
If $\lambda \notin  (-1/2,1/2) $, then 
\begin{equation}\label{garding}
	\langle (\lambda I -K ) \phi,\phi \rangle_H \geq \min\left\{ \int_{\omega} A |\nabla \mathcal{S}\phi |^2  \,dy  , \  \int_{Y\setminus \omega} A |\nabla \mathcal{S}\phi|^2\,dy \right\} .
\end{equation}

\textit{Step 2}. Let $Q:H^1(\omega) \rightarrow H^1(\mathbb{T}^d)$ be the usual bounded extension operator, we define $v \in H^1(\mathbb{T}^d)$ by
\begin{equation*} v=\left\{
    \begin{aligned}
        &\mathcal{S}\phi - \int_{\omega} 
\mathcal{S}\phi && \mathrm{in}\ \omega,\\
&Q(v|_{\omega}) && \mathrm{in} \ \mathbb{T}^d\setminus \omega.
\end{aligned}\right.
\end{equation*}
Let $u = \mathcal{S}\phi$, and test $v$ and $\mathcal{S}\phi$ in \eqref{def:variational}, respectively, note that $\langle \phi, v \rangle_{-1/2,1/2} = \langle \phi, \mathcal{S}\phi \rangle_{-1/2,1/2}$, we get
\begin{equation*}
    \int_Y A\nabla \mathcal{S}\phi \nabla v = \int_Y A|\nabla \mathcal{S}\phi |^2,
\end{equation*}
from the definition of $v$, we have
\begin{equation*}
    \int_{Y\setminus \omega} |\nabla \mathcal{S}\phi|^2 \leq \int_{Y\setminus \omega} A\nabla \mathcal{S}\phi\nabla v,
\end{equation*}
thus $\| \nabla \mathcal{S}\phi \|_{L^2(Y\setminus \omega)} \leq C\| \nabla v\|_{L^2(Y\setminus \omega)}  \leq C\| v\|_{H^1(\omega)}\leq C \| \nabla \mathcal{S}\phi \|_{L^2(\omega)}$. Therefore,
\begin{equation}\label{inte:periodic}
    \int_{\omega} A|\nabla \mathcal{S}\phi|^2\geq C \int_Y A|\nabla \mathcal{S}\phi|^2 = C \| \phi \|_{H}^2.
\end{equation}

\textit{Step 3}. Let $P:H^1(\mathbb{T}^d \setminus \omega) \rightarrow H^1(\mathbb{T}^d)$ be the usual bounded extension operator, we define $w \in H^1(\mathbb{T}^d)$ by
\begin{equation*} w=\left\{
    \begin{aligned}
        &P(w|_{\mathbb{T}^d\setminus \omega})&& \mathrm{in}\ \omega,\\
&\mathcal{S}\phi - \int_{Y\setminus \omega} 
\mathcal{S}\phi  && \mathrm{in} \ \mathbb{T}^d\setminus \omega.
\end{aligned}\right.
\end{equation*}
Let $u = \mathcal{S}\phi$, and test $w$ and $\mathcal{S}\phi$ in \eqref{def:variational}, respectively, note that $\langle \phi, w\rangle_{-1/2,1/2} = \langle \phi, \mathcal{S}\phi \rangle_{-1/2,1/2}$, we get
\begin{equation*}
    \int_Y A\nabla \mathcal{S}\phi \nabla w = \int_Y A|\nabla \mathcal{S}\phi |^2,
\end{equation*}
from the definition of $w$, we have
\begin{equation*}
    \int_{\omega} |\nabla \mathcal{S}\phi|^2 \leq \int_{\omega} A\nabla \mathcal{S}\phi\nabla w,
\end{equation*}
thus $\| \nabla \mathcal{S}\phi \|_{L^2( \omega)} \leq C\| \nabla w\|_{L^2( \omega)}  \leq C\| w\|_{H^1(Y\setminus \omega)}\leq C \| \nabla \mathcal{S}\phi \|_{L^2(Y\setminus \omega)}$. Therefore,
\begin{equation}\label{exte:periodic}
    \int_{Y\setminus \omega} A|\nabla \mathcal{S}\phi|^2\geq C \int_Y A|\nabla \mathcal{S}\phi|^2 = C \| \phi \|_{H}^2.
\end{equation}

\textit{Step 4}. Combine 
\eqref{garding}, \eqref{inte:periodic}, \eqref{exte:periodic}, we obtain the coercive inequality: for $\lambda \notin (-1/2,1/2)$, $\langle (\lambda I -K ) \phi,\phi \rangle_H  \geq C\|\phi\|_H^2$, this implies that $\lambda I-K$ is bijective, and hence the spectrum of $K$ is in $(-1/2,1/2)$. 
\end{proof}

\noindent \textit{Proof in words of results of Section \ref{uniformcontrollcell} without condition \eqref{smoothc}.} 
In the proof of results in Section \ref{uniformcontrollcell}, all properties of layer potentials that we need to use  are listed in Section \ref{sec:periodiclayerpotential}, by replacing single layer potential there by the weak single layer potential, and NP operator $\mathcal{K}^*$ there by $K$, respectively, we obtain the similar properties:
\begin{itemize}
		\item[(a)] For $\psi \in H_0^{-1/2}(\partial \omega)$, $\mathcal{S}\psi \in H^1(\mathbb{T}^d)$.
		\item[(b)] For $\psi \in H_0^{-1/2}(\partial \omega)$, jump relation holds, i.e.,
			\begin{equation*}
			\frac{\partial}{\partial \nu_A} \big(\mathcal{S}\psi \big) \Big|_{\omega\pm} = \Big(\pm \frac{1}{2}I + K \Big)\psi.
		\end{equation*}
	\item[(c)] For $\psi \in H^{-1/2}_0(\partial \omega)$, $\mathcal{L}(\mathcal{S}\psi) =0$ in $\mathbb{T}^d \setminus \partial \omega$.
	\item[(d)] The spectrum of $K$ on $H^{-1/2}_0(\partial \omega)$ is contained in $(-1/2,1/2)$.
	\end{itemize}
Therefore, all arguments in the proof of results in Section \ref{uniformcontrollcell} are valid by using these weak layer potentials, the conclusion still holds even if \eqref{smoothc} is removed.
\hfill $\square$

%%%%%%%%
%%%%%%%%
	
\bibliographystyle{siam}
\bibliography{mybib}
	
\end{document}